\documentclass[reqno]{amsart}

\usepackage{amsmath}
\usepackage{amsthm}
\usepackage{amssymb}
\usepackage{dsfont}
\usepackage{graphicx}
\usepackage[english]{babel}
\usepackage{layout}
\usepackage{color}
\usepackage[dvipsnames]{xcolor}
\usepackage{enumitem}
\usepackage{float}
\usepackage[font=footnotesize, labelfont=bf]{caption}

\usepackage[colorlinks,citecolor=blue,urlcolor=black]{hyperref}

\newtheorem{theorem}{Theorem}[section]
\newtheorem{corollary}[theorem]{Corollary}
\newtheorem{lemma}[theorem]{Lemma}
\newtheorem{proposition}[theorem]{Proposition}
\theoremstyle{definition}
\newtheorem{definition}[theorem]{Definition}
\newtheorem{remark}[theorem]{Remark}

\newtheorem*{theorem*}{Theorem}

\def\Xint#1{\mathchoice
	{\XXint\displaystyle\textstyle{#1}}%
	{\XXint\textstyle\scriptstyle{#1}}%
	{\XXint\scriptstyle\scriptscriptstyle{#1}}%
	{\XXint\scriptscriptstyle\scriptscriptstyle{#1}}%
	\!\int}
\def\XXint#1#2#3{{\setbox0=\hbox{$#1{#2#3}{\int}$ }
		\vcenter{\hbox{$#2#3$ }}\kern-.57\wd0}}

\def\dashint{\Xint-}

\newcommand{\dx}{\,\mathrm{d}x}
\newcommand{\dP}{\,\mathrm{d}\mathbb{P}}
\newcommand{\e}{\varepsilon}
\newcommand{\dist}{{\rm{dist}}}
\newcommand{\Lw}{\mathcal{L}(\omega)}
\newcommand{\Nw}{\mathcal{N}(\omega)}
\newcommand{\R}{\mathbb{R}}
\newcommand{\Q}{\mathbb{Q}}
\newcommand{\N}{\mathbb{N}}
\newcommand{\Z}{\mathbb{Z}}
\newcommand{\w}{\omega}
\newcommand{\Ard}{\mathcal{A}^R(D)}
\newcommand{\Areg}{\mathcal{A}^R}
\newcommand{\Hd}{\mathcal{H}^{d-1}}
\newcommand{\Hdtwo}{\mathcal{H}^{d-2}}
\newcommand{\dHd}{\ensuremath{\,\mathrm{d}\Hd}}

\newcommand{\I}{\ensuremath{\mathcal{I}}}

\newcommand{\wcont}{\subset\subset}

\newcommand{\eg}{{\it e.g.}, }
\newcommand{\ie}{{\it i.e.}, }

\definecolor{mygreen}{RGB}{0,150,0}

\begin{document}
	
\author{Annika Bach}
\address[Annika Bach]{Zentrum Mathematik, Technische Universit\"at M\"unchen, Boltzmannstra{\ss}e 3, 85748 Garching bei M\"unchen, Germany}
\email{annika.bach@ma.tum.de}	

\author{Marco Cicalese}
\address[Marco Cicalese]{Zentrum Mathematik, Technische Universit\"at M\"unchen, Boltzmannstra{\ss}e 3, 85748 Garching bei M\"unchen, Germany}
\email{cicalese@ma.tum.de}

\author{Matthias Ruf}
\address[Matthias Ruf]{EPFL SB MATH, Ecole polytechnique fédérale de Lausanne, Station 8, 1015 Lausanne, Switzerland}
\email{matthias.ruf@epfl.ch}

\title[Random discretizations of the Ambrosio-Tortorelli functional]{Random finite-difference discretizations of the Ambrosio-Tortorelli functional with optimal mesh size}

\begin{abstract}
We propose and analyze a finite-difference discretization of the Ambrosio-Tortorelli functional. It is known that if the discretization is made with respect to an underlying periodic lattice of spacing $\delta$, the discretized functionals $\Gamma$-converge to the Mumford-Shah functional only if $\delta\ll\e$, $\e$ being the elliptic approximation parameter of the Ambrosio-Tortorelli functional. Discretizing with respect to stationary, ergodic and isotropic random lattices we prove this $\Gamma$-convergence result also for $\delta\sim\e$, a regime at which the discretization with respect to a periodic lattice converges instead to an anisotropic version of the Mumford-Shah functional. Moreover, we show that this scaling is optimal in the sense that it is the largest possible discretization scale for which the $\Gamma$-limit is of Mumford-Shah type. Finally, we present some numerical results highlighting the isotropic behavior of our random discrete functionals.
\end{abstract}


\maketitle


\section{Introduction}
The minimization of the Mumford-Shah functional has been introduced in the framework of image analysis as a simple and yet powerful variational method for image-segmentation problems (see, \eg\cite{AuKo,MSreview,MoSo,VeCha}. In this field, a main task consists in detecting relevant object contours of (possibly distorted) digital images. Representing a gray-scale image on a domain $D\subset \R^{d}$ as a function $g:D\to[0,1]$ encoding at each point of $D$ the gray-level of the image, a ``cartoon" version of $g$ is obtained by minimizing in the pair $(u,K)$ the functional 
\begin{equation}
\int_{D\setminus K}|\nabla u|^{2}\dx+\beta\,\mathcal{H}^{d-1}(K)+\gamma\int_{D}|u-g|^2\dx.
\end{equation}
In this setting $K\subset D$ is a piece-wise regular and relatively closed set with finite $(d-1)$-dimensional Hausdorff measure $\mathcal{H}^{d-1}$, the function $u$ belongs to $C^{1}(D\setminus K)$ and $\beta$ and $\gamma$ are nonnegative parameters. Loosely speaking, the minimization of the above functional results in a pair $(u,K)$ where $u$ is smooth and close to the input image $g$ outside a set $K$ whose $\Hd$-measure has to be as small as possible. In this sense $K$ may be interpreted as the set of contours of the ``cartoon'' image $u$, or in other words the set of relevant object contours of $g$. Besides being a simple model for image segmentation (in this case the relevant space dimension is $d=2$), the Mumford-Shah functional has applications also in higher dimensions. The case $d=3$ is particularly important for its mechanical interpretation, as the functional coincides with the Griffith's fracture energy in the anti-plane case (see \cite{BFM}).

A weaker formulation of the problem was proposed in \cite{AmDeG} and led to the introduction of the space $SBV$ of special functions of bounded variation on which the Mumford-Shah functional is defined as
\begin{equation}\label{eq:introMS}
M\!S(u)=\int_{D}|\nabla u|^2\dx+\beta\,\mathcal{H}^{d-1}(S_u)+\gamma\int_{D}|u-g|^2\dx.
\end{equation}
In this new setting the functional depends only on the function $u$, and the role of $K$ is now played by $S_{u}$ the set of discontinuity points of $u$, so that a solution of the original problem can be obtained by proving regularity of the pair $(u,K)$, where $u$ is a minimizer of $MS$ and $K=\overline{ S_{u}}$ (see \cite{Fo} for a recent review on this research direction). The Mumford-Shah functional belongs to the family of so-called free-discontinuity functionals, whose variational analysis has been initiated in \cite{Amb} and it is the object of many papers in the last decades (see, \eg the monograph \cite{AFP} and the references therein).

It turns out that minimizing the Mumford-Shah functional numerically is a difficult task mainly due to the presence of the surface term $\mathcal{H}^{d-1}(S_u)$. Hence, several kind of approximations have been proposed (cf., \eg \cite{AmTo,AmTo2,BrDM,Go,PCBC}). Among them, the most popular is perhaps the one introduced by Ambrosio and Tortorelli in \cite{AmTo,AmTo2}. Given a small parameter $\e>0$ and $0<\eta_{\e}\ll\e$ the elliptic approximation $AT_{\e}:W^{1,2}(D)\times W^{1,2}(D)\to [0,+\infty]$ is given by
\begin{equation}\label{eq:ATintro}
AT_{\e}(u,v)=\int_D (v^2+\eta_{\e})|\nabla u|^2\,\mathrm{d}x+\frac{\beta}{2}\int_D\frac{(v-1)^2}{\e}+\e|\nabla v|^2\,\mathrm{d}x+\gamma\int_D|u-g|^2\,\mathrm{d}x.
\end{equation}
It is well-known that as $\e\to 0$ the family $AT_{\e}$ approximates the Mumford-Shah functional in the sense of $\Gamma$-convergence (cf. \cite{AmTo,AmTo2}). Since the functionals $AT_\e$ are equicoercive this implies that, up to subsequences, the first component $u_{\e}$ of any global minimizer $(u_{\e},v_{\e})$ of $AT_{\e}$ converges to a global minimizer $u$ of $M\!S$. The second component $v_{\e}$ is a sequence of edge variables that provides a diffuse approximation of $S_{u}$. The functionals $AT_{\e}$ being elliptic, finite-elements or finite-difference schemes can be implemented. On the one hand, $\e$ should be taken very small in order to be sure that the diffuse approximation of $S_u$ produces almost sharp edges. On the other hand, to guarantee that finite elements/differences still approximate the Mumford-Shah functional, former mathematical results assumed the mesh-size used in the discretization step to be infinitesimal with respect to $\e$ (see \cite{BeCo,Bou}). Moreover, in \cite{BBZ} Braides, Zeppieri and the first author have proven that such a condition is indeed necessary to obtain the isotropic surface term $\mathcal{H}^{d-1}(S_u)$ in the $\Gamma$-limit when using a finite-difference discretization on a square lattice (see also \cite{BY} for a similar result concerning the Modica-Mortola functional). Dropping the fidelity term $\gamma\int_D|u-g|^2\dx$, which does not affect the $\Gamma$-convergence analysis, we briefly describe their result. For $\delta_{\e}>0$ such that $\lim_\e\delta_\e=0$, in \cite{BBZ} the authors considered functionals defined for $u,v:\delta_{\e}\Z^d\cap D\to\R$ as
\begin{equation}\label{eq:introBBZ}
E_{\e,\delta_{\e}}(u,v)=\frac{1}{2}\Bigg(\sum_{\substack{i,j\in\delta_{\e}\Z^d\cap D\\ |i-j|=\delta_{\e}}}\hspace*{-1em}\delta_{\e}^dv(i)^2\left|\frac{u(i)-u(j)}{\delta_{\e}}\right|^2+\hspace*{-0.5em}\sum_{i\in\delta_{\e}\Z^2\cap D}\hspace*{-0.5em}\delta_{\e}^d\frac{(v(i)-1)^2}{\e}+\frac{1}{2}\hspace*{-.4em}\sum_{\substack{i,j\in\delta_{\e}\Z^d\cap D\\ |i-j|=\delta_{\e}}}\hspace*{-1em}\e \delta_{\e}^d\left|\frac{v(i)-v(j)}{\delta_{\e}}\right|^2\Bigg).
\end{equation}
In \cite[Theorem 2.1]{BBZ} it has been proven that the $\Gamma$-limit of $E_{\e,\delta_{\e}}$ depends on \mbox{$\ell:=\lim_{\e\to 0}(\delta_{\e}/\e)$} according to the following scheme:
\begin{itemize}
\item[-] if $\ell=0$ then $\Gamma$-$\lim_{\e}E_{\e,\delta_{\e}}$  is the Mumford-Shah functional \eqref{eq:introMS},
\item[-] if $\ell>0$ and $d=2$ then $\Gamma$-$\lim_{\e}E_{\e,\delta_{\e}}$ is an anisotropic free-discontinuity functional,
\item[-] if $\ell=+\infty$ then $\Gamma$-$\lim_{\e}E_{\e,\delta_{\e}}$ is finite only on $W^{1,2}(D)$ where it coincides with $\int_D|\nabla u|^2\dx$.
\end{itemize}
The case $\ell=0$ has also been considered in the recent paper \cite{CSS}, where the authors prove a similar result in dimension $d=2$ and $d=3$ for finite-difference discretizations of Ambrosio-Totorelli-type approximations of the Griffith functional in the context of brittle fracture.
The scheme above points out that this discretization works only for a very fine mesh-size $\delta_{\e}\ll\e$, while it approximates only an anisotropic version of the $MS$ functional for $\delta_{\e}\sim\e$. However an approximation at a scale $\delta_{\e}\sim\e$ is preferable, since it has a lower computational cost with respect to one at a scale $\delta_{\e}\ll\e$. One possible way to avoid the emergence of anisotropy in the limit, while keeping the computational cost low could be to take into account long-range interactions in the approximation of the gradient of the edge variable $v$ (similar to the approach in \cite{Ch99} in the case of the so-called weak-membrane energy) and not only neighboring differences as done in \cite{BBZ}.
Here we take a different approach which draws some inspiration from the recent results in \cite{R18} and exploit the fact that statistically isotropic point sets have the flexibility to approximate interfaces without any directional bias also in the case that only short-range interactions are taken into account.  
More precisely, we use discretizations on random point sets to circumvent anisotropic limits. Namely, we replace periodic lattice in \eqref{eq:introBBZ} by so-called stochastic lattices and then define a random family of discretizations of the Ambrosio-Tortorelli functional \eqref{eq:ATintro} with mesh-size $\delta_{\e}=\e$ for which we can prove $\Gamma$-convergence to the isotropic Mumford-Shah functional almost surely (a.s.). 
We point out that this is a purely theoretical result which suggests a possible way to numerically approximate isotropic free-discontinuity functionals with discrete ones on random grids. In the last section of this paper we select two specific test images to point out some qualitative differences between a segmentation based on such an approximation and that based on a nearest-neighbors discretization of the Ambrosio-Tortorelli functionals on the square lattice. Since the number of nearest-neighbor interactions in the random lattice is in average larger than the one in the square lattice (see also Seciton \ref{sec:numerics}), an in-depth comparison of the two approaches (deterministic and random) should also include possibly long-range deterministic discretizations. The natural question of the quantitative comparison of these different approaches is an interesting problem on its own and is out of the scope of this paper.

We highlight that, although the starting point of the present analysis, namely the discretization on a stochastic lattice, is the same as in \cite{R18}, the proof of the convergence result is quite challenging and requires new ideas. In particular, as we will explain in details later, our result needs a fine characterization of the surface-energy density in \eqref{intro:subseq}, which turns out to be quite involved as it has to take into account the interaction of the two variables in the Ambrosio-Tortorelli approximation.\\

In what follows we give a more detailed description of the results contained in this paper. 
Given a probability space $(\Omega,\mathcal{F},\mathbb{P})$ for each $\omega\in\Omega$ we consider a countable point set $\Lw\subset\R^d$ that satisfies suitable geometric constraints preventing the formation of clusters or arbitrarily large holes (cf. Definition \ref{defadmissible}). Then, given $\e>0$ we introduce a random discretization of the functional in \eqref{eq:ATintro} as the family of functionals $F_{\e}(\w)$ defined on maps $u,v:\e\Lw\cap D\to\R$ by  
\begin{equation}\label{eq:introdefF}
F_{\e}(\w)(u,v)=F_{\e}^{b}(\w)(u,v)+F_{\e}^{s}(\w)(v),
\end{equation}
where $F_{\e}^{b}(\w)$ and $F_{\e}^{s}(\w)$ denote the bulk and surface terms of the discretization, respectively. They are defined as
\begin{equation}\label{eq:introdefbulk}
F_{\e}^{b}(\w)(u,v)=\frac{1}{2}\sum_{\substack{(x,y)\in\mathcal{E}(\w)\\ \e x,\e y\in D}}\e^{d}v(\e x)^2\left|\frac{u(\e x)-u(\e y)}{\e}\right|^{2}
\end{equation}
and
\begin{equation}\label{eq:introdefsurf}
F_{\e}^{s}(\w)(v)=\frac{\beta}{2}\Big(\sum_{\e x\in \e\Lw\cap D}\e^{d-1}(v(\e x)-1)^2+\frac{1}{2}\sum_{\substack{(x,y)\in\mathcal{E}(\w)\\ \e x,\e y\in D}}\e^{d+1}\left|\frac{v(\e x)-v(\e y)}{\e}\right|^{2}\Big).
\end{equation}
In the above sums $\mathcal{E}(\w)\subset\Lw\times\Lw$ denotes a suitable set of short-range edges (for instance the Voronoi neighbors; see Definition \ref{defgoodedges} for general assumptions). Our main result (Theorem \ref{MSapprox}) reads as follows: Assuming the random graph $(\mathcal{L},\mathcal{E})$ to be stationary, ergodic and isotropic in distribution (for a precise definition see Section \ref{subsec:stochlatt}) there exist two positive constants $c_1,c_2$ such that with full probability the functionals $F_{\e}(\w)$ $\Gamma$-converge to the deterministic functional 
\begin{equation}\label{intro:isotropic}
F(u)=
\displaystyle c_1\int_D|\nabla u|^2\,\mathrm{d}x+c_2\mathcal{H}^{d-1}(S_u). 
\end{equation} 
Some remarks are in order:
\begin{enumerate}[label=(\roman*)]
	\item A point process that satisfies all our assumptions is given by the random parking process \cite{glpe,Pe}.
	\item The coefficients $c_1$ and $c_2$ are not given in a closed form but can be estimated by solving two asymptotic minimization problems (see Section \ref{s.presentation}). Moreover, their ratio can be tuned via the parameter $\beta$ since $c_2$ is proportional to $\beta$, while $c_1$ does depend only on the graph $(\mathcal{L},\mathcal{E})$.
	\item Our approach requires to determine only the Voronoi neighbors, but not the volume of teh Voronoi cells or other related geometric quantities. One can also avoid the determination of the Voronoi neighbors using a $k$-NN algorithm with a sufficiently large $k$ (see also the discussion in \cite[Remark 2.7]{R18}).
	\item In the definition of the discrete approximation \eqref{eq:introdefF}, \eqref{eq:introdefbulk}, \eqref{eq:introdefsurf}, we have taken the mesh-size equal to $\e$. Except for the value of the constant $c_2$, the above result and the analysis of this paper remain unchanged if we consider a mesh-size that is only proportional to $\e$ (see also Theorem \ref{t.representationl}). Interestingly, this is the largest possible discretization scale for which the $\Gamma$-limit is of Mumford-Shah type (see Corollary \ref{c.optimal} and the discussion below).
	\item The addition of a fidelity term of the form 
	\begin{equation}\label{eq:fidelity}
	\gamma\sum_{\e x\in\e\Lw\cap D}\e^d|u(\e x)-g_{\e}(\e x)|^2
	\end{equation}
	to the discrete approximations $F_{\e}(\w)(u,v)$ can be analyzed exactly as in \cite[Theorem 3.8]{R18} and leads to an additive term $c_3\int_D|u-g|^2\dx$ in the limit functional, provided the discrete approximation $g_{\e}$ of $g$ converges in $L^2(D)$. Moreover, under this assumption the global minimizers of the modified discrete functionals converge in $L^2(D)$ to the minimizers of the new limit functional 
	\begin{equation*}
	F_g(u)=F(u)+c_3\int_D|u-g|^2\dx.
	\end{equation*}
	In this paper we will neglect the fidelity term for the sake of notational simplicity.  
\end{enumerate}

\medskip
\noindent We now explain briefly the strategy to prove the approximation result described above. It consists of two main steps, a first deterministic one and a second stochastic one. Applying the so-called localization method of $\Gamma$-convergence together with \cite[Theorem 1]{BFLM}, in the first step we show that for a single realization $(\Lw,\mathcal{E}(\w))$ the functionals $F_\e(\w)$ $\Gamma$-converge up to subsequences to a free-discontinuity functional of the form
\begin{equation}\label{intro:subseq}
F(\w)(u)=\int_Df(\w,x,\nabla u)\,\mathrm{d}x+\int_{S_u}\varphi(\w,x,\nu_u)\,\mathrm{d}\mathcal{H}^{d-1},
\end{equation}
(see Theorem \ref{mainrep}). Based on this integral representation, in the second step we establish a stochastic homogenization result (Theorem \ref{mainthm1}), which states that for a stationary and ergodic graph $(\mathcal{L},\mathcal{E})$ the whole sequence $(F_{\e}(\w))$ $\Gamma$-converges a.s. to the functional
\begin{equation}\label{intro:homogenization}
F(u)=\int_Df_{\rm hom}(\nabla u)\,\mathrm{d}x+\int_{S_u}\varphi_{\rm hom}(\nu_u)\,\mathrm{d}\mathcal{H}^{d-1}.
\end{equation}  
In contrast to \eqref{intro:subseq} the densities $f_{\rm hom}$ and $\varphi_{\rm hom}$ in \eqref{intro:homogenization} do not depend on $x$ and are deterministic. Moreover, assuming that in addition the graph $(\mathcal{L},\mathcal{E})$ is isotropic, one can show that also $f_{\rm hom}$ and $\varphi_{\rm hom}$ are isotropic, which finally allows us to write the $\Gamma$-limit in the form \eqref{intro:isotropic}.

We highlight that a crucial step in this procedure consists in proving that a separation of bulk and surface contributions takes place in the limit. More precisely, we show that the bulk density $f(\w,\cdot,\cdot)$ in \eqref{intro:subseq} coincides with the density of the $\Gamma$-limit of the quadratic functionals $u\mapsto F_{\e}^{b}(\w)(u,1)$ defined in \eqref{eq:introdefbulk}, while the surface density $\varphi(\w,\cdot,\cdot)$ is determined by solving a $u$-dependent non-convex constrained optimization problem involving only the surface contribution $F_{\e}^s(\w)$ (see Remark \ref{r.blowupformulas}). Such a separation of energy contributions in the characterization of the surface density has already been a major issue in \cite{BBZ}. There the authors use a geometric construction to show that in dimension 2 the discrete bulk energy can be neglected in the formula of the surface integrand (cf. \cite[Theorem 5.10]{BBZ}). This explicit construction is however not feasible for a stochastic lattice. Instead our approach is more abstract. It makes use of a weighted coarea formula (cf. Lemma \ref{separationofscales1}) that works both in the case of stochastic and deterministic lattices and in any dimension. Hence the characterization of $\varphi(\w,\cdot,\cdot)$ can be seen as one of the main novelties in this paper. Moreover, it is a key ingredient in the proof of the stochastic homogenization result. More in detail, it leads to the definition of a suitable subadditive stochastic process that can be analyzed as in \cite{ACR,BCR,CDMSZ17a} via ergodic theorems and finally to the almost sure existence of the $\Gamma$-limit as in \eqref{intro:homogenization}. 

The above integral representation can be extended to the case where the discretization parameter $\delta_\e$ is only proportional to $\e$. More in detail, we show that for $\delta_\e=\ell\e$ with $\ell\in (0,+\infty)$ the volume integrand in \eqref{intro:subseq} remains unchanged, while the surface integrand depends on the ratio $\ell$ and blows up linearly as $\ell\to+\infty$ (cf. Theorem \ref{t.representationl}). This indicates that as in the deterministic case considered in \cite{BBZ}, the stochastic discretization of the Ambrosio-Tortorelli functionals on a scale $\delta_\e\gg\e$ cannot converge to a functional that is finite on $SBV(D)\setminus W^{1,2}(D)$. Indeed, this is shown in Corollary \ref{c.optimal}. In that sense the discretization of the Ambrosio-Tortorelli functionals defined in \eqref{eq:introdefF} can be interpreted as optimal since it approximates the Mumford-Shah functional at the largest possible discretization scale.

\medskip

The paper is organized as follows. In Section \ref{Sec:prelim} we introduce the notation used throughout the paper, before presenting the general results in Section \ref{s.presentation}. The latter section contains our main approximation result Theorem \ref{MSapprox} together with the integral-representation result and the stochastic homogenization theorem mentioned above, which we consider to be of independent interest for the reader. In particular, we also present here the asymptotic minimization formula characterizing $\varphi(\w,\cdot,\cdot)$ and we discover a natural relation between our discrete Ambrosio-Tortorelli functionals and weak-membrane energies. 
The proofs of the general results are carried out in Sections \ref{s.proofs} and \ref{s.stochhom}. Section \ref{s.proofs} contains the proof of the integral-representation result and the asymptotic formulas for the integrands, while the stochastic homogenization result is proven in Section~\ref{s.stochhom}. 
Eventually, in Section \ref{sec:numerics} we briefly explain how to use our approximation result in practice, i.e., we describe the construction of a suitable stochastic lattice. We also include some numerical results based on an alternating minimization scheme highlighting the different behavior of the discrete functionals in \eqref{eq:introBBZ} and \eqref{eq:introdefF} regarding (an)isotropy.

\section{Setting of the problem and preliminaries }\label{Sec:prelim}
\subsection{General notation}
We first introduce some notation that will be used in this paper. Given a measurable set $A\subset\R^d$ we denote by $|A|$ its $d$-dimensional Lebesgue measure, and by $\mathcal{H}^{k}(A)$ its $k$-dimensional Hausdorff measure. We denote by $\mathds{1}_A$ the characteristic function of $A$. If $A$ is finite, $\#A$ denotes its cardinality. Given an open set $O\subseteq\R^d$, we denote by $\mathcal{A}(O)$ the family of all bounded, open subsets of $O$ and by $\mathcal{A}^R(O)$ the family of bounded, open subsets with Lipschitz boundary. Given $A\in\mathcal{A}^R(O)$ and $\delta>0$ we set
\begin{equation*}
\partial_\delta A:=\{x\in\R^d\colon \dist(x,\partial A)\leq\delta\}.
\end{equation*}
For $x\in\R^d$ we denote by $|x|$ the Euclidean norm. As usual $B_{\varrho}(x_0)$ denotes the open ball with radius $\varrho$ centered at $x_0\in\R^d$. We write $B_{\varrho}$ when $x_0=0$. Given $\nu\in S^{d-1}$, we let $\nu_1=\nu,\nu_2,\dots,\nu_d$ be an orthonormal basis of $\R^d$ and we define the cube $Q_{\nu}$ as
\begin{equation}\label{eq:defcube}
Q_{\nu}=\left\{z\in\R^d:\;|\langle z,\nu_i\rangle| <1/2\quad\forall i=1,\ldots,d\right\},
\end{equation}
where the brackets $\langle\cdot,\cdot\rangle$ denote the scalar product. Given $x_0\in\R^d$ and $\varrho>0$, we set $Q_{\nu}(x_0,\varrho)=x_0+\varrho Q_{\nu}$. We also denote by $H_\nu(x_0)$ the hyperplane orthogonal to $\nu$ and passing through $x_0$. If $x_0=0$ we simply write $H_\nu$.

For $p\in [1,+\infty]$ we use standard notation $L^p(D)$ for the Lebesgue spaces and $W^{1,p}(D)$ for the Sobolev spaces.
We denote by $SBV(D)$ the space of special functions of bounded variation in $D$ (for the general theory see, \eg \cite{AFP}).
If  $u\in SBV(D)$ we denote by $\nabla u$ its approximate gradient, by $S_u$ the approximate discontinuity set of $u$, by $\nu_u$ the generalized outer normal to $S_u$, and $u^+$ and $u^-$ are the traces of $u$ on both sides of $S_u$. Moreover, we consider the larger space $GSBV(D)$, which consists of all functions $u\in L^1(D)$ such that for each $k\in\N$ the truncation of $u$ at level $k$ defined as $T_{k}u:=-k\vee(u\wedge k)$ belongs to $SBV(D)$. Furthermore, we set
\[SBV^2(D):=\{u\in SBV(D): \nabla u\in L^2(D)\ \text{and}\ \Hd(S_u)<+\infty\}\]
and
\[GSBV^2(D):=\{u\in GSBV(D): \nabla u\in L^2(D)\ \text{and}\ \Hd(S_u)<+\infty\}.\]
It can be shown that $SBV^2(D)\cap L^\infty(D)=GSBV^2(D)\cap L^\infty(D)$.

For $x_0\in \R^d$, $\nu\in S^{d-1}$ and $a,b\in\R$ we define the function $u_{x_0,\nu}^{a,b}:\R^{d}\to\R$ as
\begin{equation}\label{eq:purejump}
u_{x_0,\nu}^{a,b}(x):=\begin{cases} a &\mbox{if $\langle x-x_0,\nu\rangle >0$,}
\\
b &\mbox{otherwise.}
\end{cases}
\end{equation}
Moreover, for $x_0,\xi\in\R^d$ we denote by $u_{x_0,\xi}$ the affine function defined as
\begin{equation}\label{def:affine}
u_{x_0,\xi}(x):=\langle\xi,x-x_0\rangle.
\end{equation}
Finally, the letter $C$ stands for a generic positive constant that may change every time it appears.

\subsection{Stochastic lattices}\label{subsec:stochlatt}
Throughout this paper we let $\Omega$ be a probability space with a complete $\sigma$-algebra $\mathcal{F}$ and probability measure $\mathbb{P}$. We call a random variable $\mathcal{L}:\Omega\to(\R^d)^{\mathbb{N}}$ a stochastic lattice. A realization of the stochastic lattice will be denoted by $\Lw$ and we also refer to it as a stochastic lattice. The following definition essentially forbids clustering of points as well as arbitrarily big empty regions in space. 

\begin{definition}[Admissible lattices]\label{defadmissible}
Let $\mathcal{L}$ be a stochastic lattice. $\mathcal{L}$ is called admissible if there exist $R>r>0$ such that the following two conditions hold a.s.:
\begin{itemize}
	\item[(i)] $\dist(x,\mathcal{L}(\w))< R\quad$ for all $x\in\R^d$;
	\item [(ii)] $\dist(x,\Lw\setminus\{x\})\geq r\quad$ for all $x\in\Lw$.
\end{itemize}
\end{definition}

\begin{remark}\label{voronoi}
We also make use of the associated Voronoi tessellation $\mathcal{V}(\w)=\{\mathcal{C}(x)\}_{x\in\Lw}$, where the (random) Voronoi cells with nuclei $x\in\Lw$ are defined as
\begin{equation*}
\mathcal{C}(x):=\{z\in\R^d:\;|z-x|\leq |z-y|\quad\text{for all }y\in\Lw\}.
\end{equation*}	
If $\Lw$ is admissible, then \cite[Lemma 2.3]{ACR} yields the inclusions $B_{\frac{r}{2}}(x)\subset \mathcal{C}(x)\subset B_R(x)$.
\end{remark}
Next we introduce some notions from ergodic theory that build the basis for stochastic homogenization.

\begin{definition}\label{def.groupaction}
	We say that a family of measurable functions $\{\tau_z\}_{z\in \mathbb{Z}^d},\tau_z:\Omega\to\Omega$, is an additive group action on $\Omega$ if
	\begin{equation*}
	\tau_0={\rm id}\quad\text{and}\quad\tau_{z_1+z_2}=\tau_{z_2}\circ\tau_{z_1}\quad\text{for all  } z_1,z_2\in\mathbb{Z}^d.
	\end{equation*} 
	An additive group action is called measure preserving if 
	\begin{equation*}
	\mathbb{P}(\tau_z B)=\mathbb{P}(B)\quad \text{ for all  } B\in\mathcal{F},\,z\in\mathbb{Z}^d.
	\end{equation*}
	Moreover, $\{\tau_z\}_{z\in\mathbb{Z}^d}$ is called ergodic if, in addition, for all $B\in\mathcal{F}$ we have the implication
	\begin{equation*}
	(\tau_z(B)=B\quad\forall z\in \mathbb{Z}^d)\quad\Rightarrow\quad\mathbb{P}(B)\in\{0,1\}.
	\end{equation*}
\end{definition}

\begin{definition}\label{defstatiolattice}
	A stochastic lattice $\mathcal{L}$ is said to be stationary if there exists an additive, measure preserving group action $\{\tau_z\}_{z\in\mathbb{Z}^d}$ on $\Omega$ such that for all $z\in\mathbb{Z}^d$
	\begin{equation*}
	\mathcal{L}\circ\tau_z=\mathcal{L}+z.
	\end{equation*}
	If in addition $\{\tau_z\}_{z\in\mathbb{Z}^d}$ is ergodic, then $\mathcal{L}$ is called ergodic, too.
	\\
	We call $\mathcal{L}$ isotropic, if for every $R\in SO(d)$ there exists a measure preserving function $\tau^{\prime}_R:\Omega\to\Omega$ such that
	\begin{equation*}
	\mathcal{L}\circ\tau^{\prime}_R=R\mathcal{L}.
	\end{equation*}
\end{definition}
\noindent In order to define gradient-like structures, we equip a stochastic lattice with a set of directed edges. 
\begin{definition}[Admissible edges]\label{defgoodedges}
Let $\mathcal{L}$ be an admissible stochastic lattice and $\mathcal{E}\subset\mathcal{L}^2$. We say that $\mathcal{E}$ is a collection of admissible undirected\footnote{One can also consider directed edges as done in \cite{R18} but then the arguments get more intricate.} edges if for all $i,j\in\mathbb{N}$ the set $\{\w\in\Omega:\;(\Lw_{i},\Lw_{j})\in\mathcal{E}(\w)\}$ is $\mathcal{F}$-measurable and
\begin{itemize}
	\item[(i)]  there exists $M>R$ such that a.s. \begin{equation}\label{finiterange}
	\sup\{|x-y|:\; (x,y)\in\mathcal{E}(\w)\}< M;
	\end{equation}
	\item[(ii)] the Voronoi neighbors $\Nw$ are contained in $\mathcal{E}(\w)$, i.e., 
	\begin{equation}\label{nncontained}
	\Nw:=\{(x,y)\in\Lw^2:\;\mathcal{H}^{d-1}(\mathcal{C}(x)\cap\mathcal{C}(y))\in (0,+\infty)\}\subset \mathcal{E}(\w).
	\end{equation}
\end{itemize}
If $\mathcal{L}$ is stationary or isotropic, we say that the edges $\mathcal{E}$ are stationary or isotropic if $\mathcal{E}\circ\tau_z=\mathcal{E}+(z,z)$ for all $z\in\mathbb{Z}^d$ or $\mathcal{E}\circ\tau^{\prime}_R=R\mathcal{E}$ for all $R\in SO(d)$. 	

\noindent For every $x\in\Lw$ we also set $\mathcal{E}(\w)(x):=\{y\in\Lw\colon (x,y)\in\mathcal{E}(\w)\}$.
\end{definition}
\noindent Enlarging $M$ if necessary, by Remark \ref{voronoi} we may assume without loss of generality that
\begin{equation}\label{neighbours}
\sup_{x\in\Lw}\#\mathcal{E}(\w)(x)\leq M.
\end{equation}
\subsection{Discretized Ambrosio-Tortorelli functionals}
In order to define the discrete approximation of the Ambrosio-Tortorelli functional \eqref{eq:ATintro} we scale a stochastic lattice by the same small parameter $\e>0$. Given a fixed bounded Lipschitz domain $D\subset\R^d$ and two functions $u,v:\e\Lw\cap D\to\R$ we define the localized discretization on an open set $A\in\mathcal{A}(\R^d)$ by
\begin{equation}\label{eq:defF}
F_{\e}(\w)(u,v,A):=F_{\e}^{b}(\w)(u,v,A)+F_{\e}^{s}(\w)(v,A),
\end{equation}
where the bulk and surface terms are defined as
\begin{equation}\label{eq:defbulk}
F_{\e}^{b}(\w)(u,v,A):=\frac{1}{2}\sum_{\substack{(x,y)\in\mathcal{E}(\w)\\ \e x,\e y\in A}}\e^{d}v(\e x)^2\left|\frac{u(\e x)-u(\e y)}{\e}\right|^{2}
\end{equation}
and
\begin{equation}\label{eq:defsurf}
F_{\e}^{s}(\w)(v,A):=\frac{\beta}{2}\Big(\sum_{\e x\in \e\Lw\cap A}\e^{d-1}(v(\e x)-1)^2+\frac{1}{2}\sum_{\substack{(x,y)\in\mathcal{E}(\w)\\ \e x,\e y\in A}}\e^{d-1}|v(\e x)-v(\e y)|^{2}\Big),
\end{equation}
respectively. If $A=D$ we write simply $F_{\e}(\w)(\cdot,\cdot)$ for $F_{\e}(\w)(\cdot,\cdot,D)$.

In order to recast our approximation problem in the framework of $\Gamma$-convergence (we refer the reader to \cite{GCB, DM} for a general overview of this topic), we will identify discrete functions with their piecewise constant interpolations on the Voronoi cells of the lattice, that is with functions of the class
\begin{equation*}
\mathcal{PC}_{\e}^{\w}:=\{u:\R^d\to\R:\;u_{|\e \mathcal{C}(x)}\text{ is constant for all }x\in\Lw\}.
\end{equation*}
With a slight abuse of notation we extend the functional to $F_{\e}(\w):L^{1}(D)\times L^1(D)\times\mathcal{A}(D)\to[0,+\infty]$ by setting
\begin{equation}\label{deffunctional}
F_{\e}(\w)(u,v,A):=
\begin{cases} 
F_{\e}(\w)(u,v,A) &\mbox{if $u,v\in\mathcal{PC}_{\e}^{\w}$, $0\leq v\leq 1$,}\\
+\infty &\mbox{otherwise.}
\end{cases}
\end{equation}
\section{Presentation of the general results}\label{s.presentation}
In this section we present the main results of the paper.
\subsection{Integral representation and separation of bulk and surface contributions}\label{s.intrep}
Our first main result is stated below in Theorem \ref{mainrep}. It shows that for every admissible lattice $\mathcal{L}$ the discrete functionals defined in \eqref{deffunctional} $\Gamma$-converge (up to subsequences) in the strong $L^1(D)\times L^1(D)$-topology to a free-discontinuity functional. Moreover, bulk and surface contributions essentially decouple in the limit. More precisely, the volume integrand coincides with the density of the discrete quadratic functionals $u\mapsto F_{\e}^{b}(\w)(u,1)$ given by \eqref{eq:defbulk}, while the surface integrand is determined by solving a $u$-dependent constrained minimization problem which involves only the surface energy $F_\e^s$ (cf. Remark \ref{r.blowupformulas}). Note that these results are true pointwise for a fixed realization of the random graph as long as the realization satisfies the geometric conditions in Definitions \ref{defadmissible} and \ref{defgoodedges}. In order to give the precise statement of the theorem we first recall a convergence result for the functionals $F_{\e}^{b}(\w)(\cdot,1)$ (here we implicitly consider as domain of this functional the set $\mathcal{PC}_{\e}^{\w}$) which is a direct consequence of \cite[Theorem 3]{ACG2} and of the fact that the $\Gamma$-limit of quadratic functionals is quadratic, too.
\begin{theorem}[\cite{ACG2}]\label{ACGmain}
Let $\Lw$ be an admissible stochastic lattice with admissible edges. For every sequence $\e\to 0$ there exists a subsequence $\e_n$ (possibly depending on the realization) such that for every $A\in\Ard$ the functionals $F_{\e_n}^b(\w)(\cdot,1,A)$ $\Gamma$-converge in the strong $L^2(D)$-topology to a functional $F^b(\w)(\cdot,A):L^2(D)\to [0,+\infty]$ that is finite only on $W^{1,2}(A)$, where it takes the form
\begin{equation*}
F^b(\w)(u,A)=
\int_A f(\w,x,\nabla u)\,\mathrm{d}x
\end{equation*}
for some non-negative Carath\'eodory-function $f(\w,\cdot,\cdot)$ that is quadratic in the second variable for a.e. $x\in D$ and satisfies the growth conditions
\begin{equation*}
\frac{1}{C}|\xi|^2\leq f(\w,x,\xi)\leq C|\xi|^2.
\end{equation*}
\end{theorem} 
We are now in a position to state our first main result.
\begin{theorem}\label{mainrep}
Let $\mathcal{L}(\w)$ be an admissible stochastic lattice with admissible edges. For every sequence $\e\to 0$ there exists a subsequence $\e_n$ (possibly depending on the realization) such that for every $A\in\Ard$ the functionals $F_{\e_n}(\w)(\cdot,\cdot,A)$ $\Gamma$-converge in the strong $L^1(D)\times L^1(D)$-topology to a free-discontinuity functional $F(\w)(\cdot,\cdot,A):L^1(D)\times L^1(D)\to[0,+\infty]$ of the form
\begin{align*}
F(\w)(u,v,A)=
\begin{cases}
\displaystyle\int_A f(\w,x,\nabla u)\dx+\int_{S_u\cap A}\varphi(\w,x,\nu_u)\dHd &\text{if $u\in GSBV^2(A)$, $v=1$ a.e. in $A$},\\
+\infty &\text{otherwise in $L^1(D)\times L^1(D)$},
\end{cases}
\end{align*}
where $\varphi(\w,\cdot,\cdot)$ is a measurable function and $f(\w,\cdot,\cdot)$ is given by Theorem \ref{ACGmain}.
\end{theorem}
\begin{remark}\label{r.blowupformulas}
Both the integrands $\varphi(\w,\cdot,\cdot)$ and $f(\w,\cdot,\cdot)$ provided by Theorem \ref{mainrep} can be characterized by asymptotic formulas. We write them after introducing some notation. For every $A\in\mathcal{A}(\R^d)$, $\delta>0$ and every pointwise well-defined function $\bar{u}\in L^\infty_{\rm loc}(\R^d)$ we denote by $\mathcal{PC}_{\e,\delta}^{\w}(\bar{u},A)$ the set 
\begin{equation}\label{def:discreteboundarycond}
\mathcal{PC}_{\e,\delta}^{\w}(\bar{u},A):=\{u\in\mathcal{PC}_{\e}^{\w}:\,u(\e x)=\bar{u}(\e x)\quad\text{if }\e x\in\Lw\cap\partial_\delta A\}
\end{equation} 
of those $\mathcal{PC}_{\e}^{\w}$-functions whose values agree with those of $\bar{u}$ in a discretized $\delta$-neighborhood of $\partial A$. Then for a.e. $x_0\in D$ and every $\xi\in\R^d$ it holds that
\begin{equation*}
f(\w,x_0,\xi)=\lim_{\varrho\to 0}\varrho^{-d}\lim_{n\to+\infty}\inf\{F_{\e_n}^b(\w)(u,1,Q_{e_1}(x_0,\varrho))\colon u\in\mathcal{PC}_{\e_n, M\e_n}^\w(u_{x_0,\xi},Q_{e_1}(x_0,\varrho))\},
\end{equation*}
where $u_{x_0,\xi}$ is the affine function defined in \eqref{def:affine} and $M$ is the maximal range of interactions in Definition \ref{defgoodedges}. Moreover, for every $x_0\in D$, $a\in\R$ and $\nu\in S^{d-1}$ we define the class of functions
\begin{equation}\label{asymptoticformula}
\mathcal{S}_{\e,\delta}^{\w}(u_{x_0,\nu}^{a,0},Q_{\nu}(x_0,\varrho))=\{u\in \mathcal{PC}_{\e,\delta}^{\w}(u_{x_0,\nu}^{a,0},Q_{\nu}(x_0,\varrho)):\;u(\e x)\in \{a,0\}\text{ for all }x\in\Lw\}
\end{equation} 
and we introduce the function 
\begin{equation}\label{eq:barvboundary}
v_{x_0,\nu}^{\e}(x):=
\begin{cases}
0 &\text{if}\ |\langle x-x_0,\nu\rangle|\leq M \e,\\
1 &\text{otherwise.}
\end{cases}
\end{equation}
We also consider the minimization problem
\begin{multline}\label{def:surf:int}
\varphi_{\e,\delta}^{\w}(u_{x_0,\nu}^{a,0},Q_\nu(x_0,\varrho))=\inf\Big\{F_{\e}^{s}(\w)(v,Q_\nu(x_0,\rho))\colon v\in\mathcal{PC}_{\e,M\e}^\w(v_{x_0,\nu}^\e,Q_\nu(x_0,\varrho)),\\
\exists\, u\in\mathcal{S}_{\e,\delta}^\w(u_{x_0,\nu}^{a,0},Q_\nu(x_0,\varrho)):\, F_{\e}^{b}(\w)(u,v,Q_{\nu}(x_0,\varrho))=0\Big\}.
\end{multline}
For every $(x_0,\nu)\in D\times S^{d-1}$ we then have that the surface density of $F(\w)$ in Theorem \ref{mainrep} is given by
\begin{equation}\label{eq:formula_phi}
\varphi(\w,x_0,\nu)=\limsup_{\varrho\to 0}\varrho^{1-d}\lim_{\delta\to 0}\limsup_{n\to +\infty}\varphi_{\e_n,\delta}^\w(u_{x_0,\nu}^{1,0},Q_\nu(x_0,\varrho)).
\end{equation}
Note that the boundary conditions for $v$ in the definition of \eqref{def:surf:int} are posed on a much smaller layer than those for $u$. This is only due to technical reasons in the proof of Lemma \ref{separationofscales1}. Alternatively we could also require that $v\in\mathcal{PC}_{\e,\delta-M\e}^\w(v_{x_0,\nu}^\e,Q_\nu(x_0,\varrho))$, but this would overburden the notation.
\end{remark}
\subsection{Stochastic homogenization and convergence to the Mumford-Shah functional}
Our second main result relies on the statistical properties of the lattice and the edges. More precisely, when $\mathcal{L}$ and $\mathcal{E}$ are stationary we can prove the following stochastic homogenization result, which shows in particular that in this case the $\Gamma$-limit provided by Theorem \ref{mainrep} is independent of the converging subsequence and hence the whole sequence converges. 
\begin{theorem}\label{mainthm1}
Let $\mathcal{L}$ be an admissible stationary stochastic lattice with admissible stationary edges in the sense of Definitions \ref{defadmissible} and \ref{defgoodedges}. Then for $\mathbb{P}$-a.e. $\w\in\Omega$ and for every $\xi\in\R^d$, $\nu\in S^{d-1}$ there exist the limits
\begin{align}
f_{\rm hom}(\w,\xi) &=\lim_{t\to+\infty}t^{-d}\inf\{F_{1}^{b}(\w)(u,1,Q(0,t))\colon u\in\mathcal{PC}_{1,M}^\w(u_{0,\xi},Q(0,t))\},\label{ex:hom:bulk}\\
\varphi_{\rm hom}(\w,\nu) &=\lim_{t\to+\infty}t^{1-d}\varphi_{1,M}^\w(u_{0,\nu}^{1,0},Q_{\nu}(0,t)),\label{ex:hom:surf}
\end{align}
where $\varphi_{1,M}^\w$ is defined as in \eqref{def:surf:int}. Moreover, the functionals $F_{\e}(\w)$ $\Gamma$- converge in the strong $L^1(D)\times L^1(D)$-topology to the functional $F_{\hom}(\w):L^1(D)\times L^1(D)\to[0,+\infty]$ defined by
\begin{align*}
F_{\rm hom}(\w)(u,v):=
\begin{cases}
\displaystyle\int_D f_{\rm hom}(\w,\nabla u)\dx+\int_{S_u}\varphi_{\rm hom}(\w,\nu_u)\dHd &\text{if $u\in GSBV^2(D)$, $v=1$ a.e. in $D$},\\
+\infty &\text{otherwise in $L^1(D)\times L^1(D)$}.
\end{cases}
\end{align*}
If in addition $\mathcal{L}$ is ergodic then $f_{\rm hom}$ and $\varphi_{\rm hom}$ are independent of $\w$.
\end{theorem}
In order to make the densities $f_{\rm hom}$ and $\varphi_{\rm hom}$ isotropic, we suggest to take as stochastic lattice the so-called {\it random parking process}. We refer the interested reader to the two papers \cite{Pe,glpe}. We recall that the random parking process defines a stochastic lattice $\mathcal{L}_{RP}$ that is admissible, stationary, ergodic, and isotropic in the sense of Definition \ref{defstatiolattice}. Moreover, the choice $\mathcal{E}(\w)=\mathcal{N}(\w)$ yields stationary and isotropic edges. We state our result for general stochastic lattices satisfying all these assumptions. 
\begin{theorem}\label{MSapprox}
Assume that $\mathcal{L}$ is an admissible stochastic lattice that is stationary, ergodic and isotropic with admissible stationary and isotropic edges. Then there exist constants $c_1,c_2>0$ such that $\mathbb{P}$-a.s. the functionals $F_{\e}(\w)$ defined in \eqref{eq:defF} $\Gamma$-converge with respect to the $L^1(D)\times L^1(D)$-topology to the functional $F:L^1(D)\times L^1(D)\to [0,+\infty]$ with domain $GSBV^2(D)\times\{1\}$, on which
\begin{equation}\label{eq:limitMS}
F(u,1)=c_1\int_D |\nabla u|^2\,\mathrm{d}x+c_2\mathcal{H}^{d-1}(S_u).
\end{equation}
\end{theorem}
\begin{remark}\label{r.fidelityterm}
As explained in the introduction, a discrete version of the fidelity term as in \eqref{eq:fidelity} can be added to the functional $F_{\e}(\w)$ obtaining discrete functionals of the form
\begin{equation}\label{eq:fidelityterm}
F_{\e}^g(\w)(u,v)=F_{\e}(\w)(u,v)+\gamma\sum_{\e x\in\e\Lw\cap D}\e^d|u(\e x)-g_{\e}(\e x)|^2,
\end{equation}
where $g\in L^2(D)$ is a given input datum and $g_{\e}$ a suitable discretization of $g$ on $\e\Lw$. 
The fidelity term leads to an additional term $c_3\int_D|u-g|^2\dx$ in the $\Gamma$-limit in \eqref{eq:limitMS}, where $c_3>0$ is proportional to the constant $\gamma$ in \eqref{eq:fidelityterm}. For details we refer the reader to the analogous result proved in \cite[Theorem 3.8]{R18} for weak membrane approximations.	
\end{remark}
\subsection{Connection to weak-membrane energies}
In this subsection we show how the discretizations of the Ambrosio-Tortorelli functional in \eqref{eq:defF}, \eqref{eq:defbulk} and \eqref{eq:defsurf} are related to the weak-membrane energies. 
In fact, neglecting the second sum in \eqref{eq:defsurf}, we can associate a weak-membrane model to the discrete Ambrosio-Tortorelli functional by optimizing $v\mapsto F_{\e}(\w)(u,v)$ for fixed $u$ (cf. Proposition \ref{prop:ATandWeak}). This connection, which we find interesting in itself, also allows us to take advantage of some of the estimates established in \cite{R18} which  turns out to be useful in the proof of our main convergence result.

We first explain what we mean by (generalized) weak-membrane energy. Consider a bounded and monotone increasing function $f:[0,+\infty)\to[0,+\infty)$ such that $f(0)=0$ and $f'(0)=1$. Then, given $u:\e\Lw\to\R$ and $A\in\mathcal{A}(D)$ we set
\begin{equation}\label{eq:defmembrane}
G_{\e}(\w)(u,A):=\frac{1}{2}\sum_{\e x\in\e\Lw\cap A}\e^{d-1}f\Bigg(\e\sum_{\e y\in\e\mathcal{E}(\w)(x)\cap A}\left|\frac{u(\e x)-u(\e y)}{\e}\right|^2\Bigg).
\end{equation}
In our present random setting these functionals are a special case of those considered in \cite{R18}. While our weak membrane energies depend on non-pairwise interactions, we remark that in the context of computer vision they were introduced and studied in \cite{BlZi,GeGe,Ma} in a simpler form accounting only for pairwise interactions. 

For our purpose it will be convenient to consider weak-membrane energies with a special choice of $f$. Namely, for a given parameter $\alpha>0$ we set $f_\alpha(t):=t(1+t/\alpha)^{-1}$ and we notice that $f_\alpha$ satisfies all assumptions listed above. We then define $G_{\e,\alpha}$ according to \eqref{eq:defmembrane} with $f=f_\alpha$. The following convergence result for the sequence $(G_{\e,\alpha}(\w))$, which can be compared with Theorem \ref{mainrep}, is a consequence of \cite[Theorem 3.3 and Remark 3.4]{R18}. We recall it here for the reader's convenience.
\begin{theorem}[\cite{R18}]\label{t.weakmembrane}
Let $\Lw$ be an admissible stochastic lattice with admissible edges $\mathcal{E}(\w)$ in the sense of Definitions \ref{defadmissible} \& \ref{defgoodedges}. For every sequence $\e\to 0$ there exists a subsequence $\e_n\to 0$ (possibly depending on the realization) such that for every $\alpha>0$ and every $A\in\Ard$ the functionals $G_{\e_n,\alpha}(\w)(\cdot,A)$ $\Gamma$-converge in the strong $L^1(D)$-topology to a free discontinuity functional $G_{0,\alpha}(\w)(\cdot,A):L^1(D)\to [0,+\infty]$ with domain $GSBV^2(A)\cap L^1(D)$, where it is given by 
\begin{equation*}
G_{0,\alpha}(\w)(u,A)=\int_A f(\w,x,\nabla u)\dx+\int_{S_u\cap A}s_\alpha(\w,x,\nu_u)\,\mathrm{d}\mathcal{H}^{d-1},
\end{equation*}
where $f(\w,x,\xi)$ coincides with the integrand in Theorem \ref{ACGmain} and the surface tension can be equivalently characterized by the two formulas
\begin{align*}
s_\alpha(\w,x_0,\nu) &=\limsup_{\varrho\to 0}\varrho^{1-d}\lim_{\delta\to 0}\limsup_{n\to +\infty}\inf\{G_{\e_n,\alpha}(\w)(u,Q_{\nu}(x_0,\varrho)):\,u\in\mathcal{S}_{\e_n,\delta}^{\w}(u_{x_0,\nu}^{1,0},Q_{\nu}(x_0,\varrho))\}\\
&=\limsup_{\varrho\to 0}\varrho^{1-d}\lim_{\delta\to 0}\limsup_{n\to +\infty}\inf\{I_{\e_n,\alpha}(\w)(w,Q_{\nu}(x_0,\varrho)):\,w\in\mathcal{S}_{\e_n,\delta}^{\w}(u_{x_0,\nu}^{1,-1},Q_{\nu}(x_0,\varrho))\},
\end{align*}
with the energy $I_{\e,\alpha}(\w)$ defined on functions $w:\e\Lw\to\{\pm 1\}$ via
\begin{equation*}
I_{\e,\alpha}(\w)(w,A):=\frac{\alpha}{4}\sum_{\e x\in\e\Lw\cap A}\e^{d-1}\max\{|w(\e x)-w(\e y)|:\e y\in\e\mathcal{E}(\w)(x)\cap A\}.
\end{equation*}
In particular, we have $s_\alpha(\w,x_0,\nu)=\alpha s_1(\w,x_0,\nu)$ and the following estimates
\begin{equation*}
\frac{1}{C}|\xi|^2\leq q(\w,x,\xi)\leq C|\xi|^2,\quad\quad\frac{\alpha}{C}\leq s_\alpha(\w,x,\nu)\leq C\alpha.
\end{equation*}
\end{theorem}
In what follows, we show that the Ambrosio-Tortorelli approximation can be interpreted as a weak membrane energy $G_{\e,\beta}$, provided we neglect the term containing the discrete gradient of the edge variable $v$. Indeed, the following proposition holds true:
\begin{proposition}\label{prop:ATandWeak}
	Let $\Lw$ be an admissible stochastic lattice with admissible edges $\mathcal{E}(\w)$ and let $G_{\e,\beta}(\w)$ be defined as in \eqref{eq:defmembrane} with $f=f_\beta$. Then for all $u:\e\Lw\to\R$ and $A\in\mathcal{A}(D)$ it holds that
	\begin{equation*}
	G_{\e,\beta}(\w)(u,A)=\min_{v:\e\Lw\to [0,1]} \Big(F_{\e}^{b}(\w)(u,v,A)+\frac{\beta}{2}\sum_{\e x\in\e\Lw\cap A}\e^{d-1}(v(\e x)-1)^2\Big).
	\end{equation*}
\end{proposition}
\begin{proof}
Recalling the definition of the bulk term $F_{\e}^{b}(\w)(u,v,A)$ in \eqref{eq:defbulk}, we can derive a pointwise optimality condition for the minimization problem, which reads (neglecting the constraint $0\leq v\leq 1$)
\begin{equation*}
\e^dv(\e x)\sum_{\e y\in\e\mathcal{E}(\w)(x)\cap A}\left|\frac{u(\e x)-u(\e y)}{\e}\right|^2+\beta \e^{d-1}v(\e x)=\beta\e^{d-1}\quad\quad\text{for all }\e x\in\e\Lw\cap A.
\end{equation*}
Rearranging terms we find that for $\e x\in\e\Lw\cap A$ we have
\begin{equation*}
v(\e x)=\Bigg(1+\frac{\e}{\beta}\displaystyle\sum_{\e y\in\e\mathcal{E}(\w)(x)\cap A}\left|\frac{u(\e x)-u(\e y)}{\e}\right|^2\Bigg)^{-1},
\end{equation*}
so that a posteriori $v(\e x)\in (0,1]$ and thus it is a minimizer of the constrained problem as well. Inserting this formula for $v$ yields the claim after some algebraic manipulations.
\end{proof}
\subsection{Discretization with mesh-size proportional to $\e$ and optimality of the lattice-scaling}
In this section we present a version of Theorem \ref{mainrep} when the mesh-size is not equal to the elliptic-approximation parameter $\e$, but only proportional to it. More precisely, we let $(\kappa_\e)$ be a sequence of positive parameters, decreasing as $\e$ decreases and such that $\lim_\e\kappa_\e=0$ and for every $u,v\in\mathcal{PC}_{\kappa_\e}^\w$ we set
\begin{equation}\label{def:Fdelta}
F_{\e,\kappa_\e}(\w)(u,v):=F_{\kappa_\e}^b(\w)(u,v)+\frac{\beta}{2}\Bigg(\sum_{\kappa_\e x\in\kappa_\e\Lw\cap D}\hspace*{-1em}\kappa_\e^d\frac{(v(\kappa_\e x)-1)^2}{\e}+\sum_{\substack{(x,y)\in\mathcal{E}(\w)\\\kappa_\e x,\kappa_\e y\in D}}\hspace*{-0.5em}\e\kappa_\e^d\left|\frac{v(\kappa_\e x)-v(\kappa_\e y)}{\kappa_\e}\right|^2\Bigg).
\end{equation}
When $\kappa_\e=\ell\e$ for some $\ell\in (0,+\infty)$ we have the following integral-representation result for the functionals $F_{\e,\kappa_\e}(\w)$, similar to Theorem \ref{mainrep}.
\begin{theorem}\label{t.representationl}
Let $\Lw$ be an admissible stochastic lattice with admissible edges $\mathcal{E}(\w)$ and for a given $\ell\in(0,+\infty)$ let $F_{\e,\kappa_\e}(\w)$ be as in \eqref{def:Fdelta} with $\kappa_\e=\ell\e$. For every sequence $\e\to 0$ there exist a subsequence $\e_n\to 0$ (possibly depending on the realization) and a functional $F_\ell(\w):L^1(D)\times L^1(D)\to[0,+\infty]$ of the form
\begin{align*}
F_\ell(\w)(u,v)=
\begin{cases}
\displaystyle\int_D f(\w,x,\nabla u)\dx+\int_{S_u}\varphi_\ell(\w,x,\nu_u)\dHd, &\text{if $u\in GSBV^2(D)$, $v=1$ a.e. in $D$,}\\
+\infty &\text{otherwise,}
\end{cases}
\end{align*}
such that $F_{\e_n,\kappa_{\e_n}}(\w)$ $\Gamma$-converges in the strong $L^1(D)\times L^1(D)$-topology to $F_\ell(\w)$. Moreover, the volume integrand $f$ is given by Theorem \ref{ACGmain} and $\varphi_\ell(\w,\cdot,\cdot)$ is a measurable function which satisfies for every $x_0\in D$ and every $\nu\in S^{d-1}$ the estimate
\begin{equation}\label{est:phil1}
\beta\ell s_1(\w,x_0,\nu)\leq\varphi_\ell(\w,x_0,\nu)\leq\beta\left(\ell+\frac{M}{\ell}\right)s_1(\w,x_0,\nu).
\end{equation}
Here $M$ is as in \eqref{neighbours} and $s_1$ is the surface integrand of the $\Gamma$-limit given by Theorem \ref{t.weakmembrane}, which exists upon passing possibly to a further subsequence. In particular, we have
\begin{equation*}
\lim_{\ell\to+\infty}\frac{1}{\ell}\varphi_\ell(\w,x_0,\nu)=\beta s_1(\w,x_0,\nu),\qquad\frac{\ell}{C}\leq\varphi_\ell(\w,x_0,\nu)\leq C\ell.
\end{equation*}
\end{theorem}
\begin{remark}\label{r.otherresults}
Although we don't state it separately, note that the statements of Theorems \ref{mainthm1} \& \ref{MSapprox} remain valid for the functionals $F_{\e,\kappa_{\e}}(\w)$ with $\kappa_{\e}=\ell\e$ with minor modifications in the proof and with surface densities depending on $\ell$ as in the theorem above.
\end{remark}
\noindent Theorem \ref{t.representationl} shows that the surface density $\varphi_\ell(\w,x,\nu)$ blows up linearly in $\ell$ when $\ell\to +\infty$. Thus, one expects that (similar to the result in \cite[Theorem 2.1(iii) and Theorem 3.1(ii)]{BBZ}) one cannot approximate the Mumford-Shah (or any other free-discontinuity) functional by discretizing the Ambrosio-Tortorelli functional via finite differences on an admissible lattice $\kappa_{\e}\Lw$ with $\kappa_{\e}/\e\to +\infty$. In fact, Corollary \ref{c.optimal} states that in this regime the $u_\e$-component of any sequence $(u_\e,v_\e)$ with $(u_\e)$ equibounded in $L^2(D)$ and such that $F_{\e,\kappa_\e}(u_\e,v_\e)<+\infty$ converges up to subsequences to some $u\in W^{1,2}(D)$. Thus interfaces are ruled out in the limit.
\begin{corollary}[Optimality of the lattice-space scaling]\label{c.optimal}
Let $\Lw$ be admissible with admissible edges $\mathcal{E}(\w)$ and consider a sequence $\kappa_{\e}>0$ with $\kappa_\e$ decreasing as $\e$ decreases and $\lim_\e\kappa_\e=0$ such that $\lim_{\e\to 0}\frac{\kappa_{\e}}{\e}=+\infty$. Let $F_{\e,\kappa_\e}(\w)$ be as in \eqref{def:Fdelta} and let $u_{\e},v_{\e}:\kappa_{\e}\Lw\to\R$ be such that $|v_{\e}|\leq 1$ and
\begin{equation*}
\|u_{\e}\|_{L^2(D)}+F_{\e,\kappa_{\e}}(\w)(u_{\e},v_{\e})\leq C\quad \text{for all $\e>0$}.
\end{equation*}
Then, up to subsequences, $u_{\e}\to u$ in $L^1(D)$ for some $u\in W^{1,2}(D)$.
\end{corollary}
\begin{remark}\label{r.gammalimit}
Under the assumptions of Corollary \ref{c.optimal}, (up to subsequences) the $\Gamma$-limit agrees with the one given by Theorem \ref{ACGmain}. Indeed, an upper bound is given by setting $v\equiv 1$, while the lower bound is obtained via comparison with weak membrane energies $G_{\e,\alpha}(\w)$ for any $\alpha>0$ in the case of a limit function $u\in W^{1,2}(D)$. We leave the details to the interested reader.
\end{remark}
\section{Separation of scales: proof of Theorem \ref{mainrep}, Theorem \ref{t.representationl} and Corollary \ref{c.optimal}}\label{s.proofs}
The main part of this section is devoted to the proof of the integral-representation result Theorem \ref{mainrep}. 
\subsection{Integral representation in $SBV^2$}
As a first step towards the proof of Theorem \ref{mainrep}, using the so-called localization method of $\Gamma$-convergence together with the general result \cite[Theorem 1]{BFLM} we prove the following preliminary result.
\begin{proposition}\label{limitsbvp}
Let $\Lw$ be an admissible stochastic lattice with admissible edges $\mathcal{E}(\w)$. Given any sequence $\e\to 0$ there exists a subsequence $\e_n$ (possibly depending on the realization) such that for all $A\in\Ard$ the functionals $F_{\e_n}(\w)(\cdot,\cdot,A)$ $\Gamma$-converge in the strong $L^1(D)\times L^1(D)$-topology to a functional $F(\w)(\cdot,\cdot,A):L^1(D)\times L^1(D)\to[0,+\infty]$. If $u\in SBV^2(A)$ then $F(\w)(u,1,A)$ can be written as
\begin{equation*}
F(\w)(u,1,A)=
\int_A h(\w,x,\nabla u)\dx+\int_{S_u\cap A}\varphi(\w,x,{u^+}-{u^-},\nu_u)\,\mathrm{d}\mathcal{H}^{d-1},
\end{equation*}
where, for $x_0\in D$, $\nu\in S^{d-1}$, $a\in\R$ and $\xi\in\R^{d}$, the integrands are given by
\begin{equation}\label{derivationformula}
\begin{split}
h(\w,x_0,\xi)&=\limsup_{\varrho\to 0}\varrho^{-d}m^\w(u_{x_0,\xi},Q_{\nu}(x_0,\varrho)),
\\
\varphi(\w,x_0,a,\nu)&=\limsup_{\varrho\to 0}\varrho^{1-d}m^\w(u_{x_0,\nu}^{a,0},Q_{\nu}(x_0,\varrho))
\end{split}
\end{equation} 
with the functions $u_{x_0,\nu}^{a,0}$ and $u_{x_0,\xi}$ defined in \eqref{eq:purejump} and \eqref{def:affine}, respectively, and the function $m^\w(\bar{u},A)$ defined for any $\bar{u}\in SBV^2(D)$ and $A\in\Ard$ by
\begin{equation*}
m^\w(\bar{u},A):=\inf\{F(\w)(u,1,A):\;u\in SBV^2(A),\,u=\bar{u}\text{ in a neighborhood of }\partial A\}.
\end{equation*}
\end{proposition}
\noindent In order to prove this result we will analyze the localized $\Gamma\hbox{-}\liminf$ and $\Gamma\hbox{-}\limsup$ $F^{\prime}(\w),F^{\prime\prime}(\w):L^{1}(D)\times L^1(D)\times\mathcal{A}(D)\to [0,+\infty]$ of the functionals $F_{\e}(\w)$, which are defined as
\begin{equation*}
\begin{split}
F^{\prime}(\w)(u,v,A)&:=\inf\{\liminf_{\e\to 0}F_{\e}(\w)(u_{\e},v_{\e},A):\;u_{\e}\to u\, \text{ and }\,v_{\e}\to v \text{ in }L^{1}(D)\},\\
F^{\prime\prime}(\w)(u,v,A)&:=\inf\{\limsup_{\e\to 0}F_{\e}(\w)(u_{\e},v_{\e},A):\;u_{\e}\to u\, \text{ and }\,v_{\e}\to v\text{ in }L^{1}(D)\}.
\end{split}
\end{equation*}

\begin{remark}\label{proxi}
Both functionals are $L^1(D)\times L^1(D)$-lower semicontinuous. Moreover, for any $w\in L^{1}(D)$ there exists indeed a sequence $w_{\e}\in\mathcal{PC}_{\e}^{\w}$ such that $w_{\e}\to w$ in $L^{1}(D)$.
\end{remark} 
Our aim is to apply the integral representation of \cite[Theorem 1]{BFLM}. To this end, below we establish several properties of $F^{\prime}(\w)$ and $F^{\prime\prime}(\w)$. The next remark about truncations allows to reduce some of the arguments used in the forthcoming proofs to the case of bounded functions.
\begin{remark}\label{trunc}
Let $u_{\e},v_{\e}\in\mathcal{PC}_{\e}^{\w}$. For any $k>0$ let $T_ku_{\e}$ denote the truncation of $u_\e$ at level $k$. Then it is immediate to see that $F_{\e}(\w)(T_ku_{\e},v_{\e},A)\leq F_{\e}(\w)(u_{\e},v_{\e},A)$ for any $A\in\mathcal{A}(D)$. In particular, whenever $u\in L^{\infty}(D)$ we can compute $F^{\prime}(\w)(u,1,A)$ and $F^{\prime\prime}(\w)(u,1,A)$ considering sequences $u_{\e}\in\mathcal{PC}_{\e}^{\w}$ such that $|u_{\e}(\e x)|\leq \|u\|_{\infty}$ for all $x\in\Lw$. Moreover, also $F^{\prime}$ and $F^{\prime\prime}$ decrease by truncation in $u$. Thus, since in addition both functionals are $L^1(D)$-lower semicontinuous, for all $u\in L^1(D)$ we have
\begin{equation*}
\begin{split}
\lim_{k\to +\infty}F^{\prime}(\w)(T_ku,1,A)&=F^{\prime}(\w)(u,1,A),\\
\lim_{k\to +\infty}F^{\prime\prime}(\w)(T_ku,1,A)&=F^{\prime\prime}(\w)(u,1,A).
\end{split}
\end{equation*}
Moreover, since $F_\e(\w)$ is invariant under translation in $u$, we deduce that also both $F'(\w)(\cdot,1,A)$ and $F''(\w)(\cdot,1,A)$ are invariant under translation in $u$.
\end{remark}
\noindent We next show that $F^{\prime\prime}(\w)$ is local.
\begin{lemma}[Locality]\label{local}
Let $A\in\Ard$. If $u,\tilde{u}\in L^1(D)$ and $u=\tilde{u}$ a.e. on $A$, then $F^{\prime\prime}(\w)(u,1,A)=F^{\prime\prime}(\w)(\tilde{u},1,A)$.
\end{lemma}
\begin{proof}
Due to Remark \ref{proxi} there exist sequences $(u_{\e},v_{\e}),(\tilde{u}_{\e},\tilde{v}_{\e})\in\mathcal{PC}_{\e}^{\w}\times\mathcal{PC}_\e^\w$ converging to $(u,1)$ and $(\tilde{u},1)$ in $L^1(D)\times L^1(D)$, respectively, and such that
\begin{equation*}
\limsup_{\e\to 0}F_{\e}(\w)(u_{\e},v_{\e},A)=F^{\prime\prime}(\w)(u,1,A),\quad\quad\limsup_{\e\to 0}F_{\e}(\w)(\tilde{u}_{\e},\tilde{v}_{\e},A)=F^{\prime\prime}(\w)(\tilde{u},1,A).
\end{equation*}
Define $u^0_{\e},v^0_{\e}\in\mathcal{PC}_{\e}^{\w}$ by their values on $\e\Lw$ as
\begin{equation*}
\begin{split}
u^0_{\e}(\e x)&=\mathds{1}_{A}(\e x)u_{\e}(\e x)+(1-\mathds{1}_A(\e x))\tilde{u}_{\e}(\e x),
\\
v^0_{\e}(\e x)&=\mathds{1}_{A}(\e x)v_{\e}(\e x)+(1-\mathds{1}_A(\e x))\tilde{v}_{\e}(\e x).
\end{split}
\end{equation*}
Using that $|\partial A|=0$ and the equiintegrability of $u_{\e},\tilde{u}_{\e},v_{\e}$, and $\tilde{v}_{\e}$, one can show that $u^0_{\e}\to \tilde{u}$ and $v^0_{\e}\to 1$ in $L^1(D)$. Then by definition
\begin{equation*}
F^{\prime\prime}(\w)(\tilde{u},1,A)\leq\limsup_{\e\to 0}F_{\e}(\w)(u^0_{\e},v^0_{\e},A)=\limsup_{\e\to 0}F_{\e}(\w)(u_{\e},v_{\e},A)=F^{\prime\prime}(\w)(u,1,A).
\end{equation*}
Exchanging the roles of $u$ and $\tilde{u}$ we conclude.
\end{proof}
The next lemma provides a lower bound for $F^{\prime}$. We also obtain equicoercivity under an additional equiintegrability assumption.
\begin{lemma}[Compactness and lower bound]\label{compact}
Assume that $A\in\mathcal{A}^R(D)$ and $u_{\e},v_{\e}\in\mathcal{PC}_{\e}^{\w}$ are such that 
\begin{equation*}
\sup_{\e}F_{\e}(\w)(u_{\e},v_{\e},A)<+\infty.
\end{equation*}
Then $v_{\e}\to 1$ in $L^1(A)$. If $u_{\e}$ is equiintegrable on $A$, then there exists a subsequence (not relabeled) such that $u_{\e}\to u$ in $L^1(A)$ for some $u\in GSBV^2(A)$. Moreover we have the estimate
\begin{equation*}
\frac{1}{c}\left(\int_A|\nabla u|^2\dx+\mathcal{H}^{d-1}(S_u\cap A)\right)\leq F^{\prime}(\w)(u,1,A)
\end{equation*}
for some constant $c>0$ independent of $\w,A$ and $u$.
\end{lemma}

\begin{proof}
Since $0\leq v_{\e}\leq 1$, boundedness of the energy and Remark \ref{voronoi} imply that $v_{\e}\to 1$ in $L^1(A)$. Moreover, due to Proposition \ref{prop:ATandWeak} we have 
\begin{equation*}
F_{\e}(\w)(u_{\e},v_{\e},A)\geq \min_{v:\e\Lw\to [0,1]}F_{\e}(\w)(u_{\e},v,A)\geq G_{\e,\beta}(\w)(u_{\e},A).
\end{equation*}
Hence the compactness statement and the lower bound on the $\Gamma\hbox{-}\liminf$ are a direct consequence of the corresponding result for weak-membrane energies (cf. Theorem \ref{t.weakmembrane} or \cite[Lemma 5.6]{R18}).
\end{proof}

As a next step we prove the corresponding upper bound for $F^{\prime\prime}(\w)$.
\begin{lemma}[Upper bound]\label{bounds}
Let $u\in L^1(D)$. There exists a constant $c>0$ independent of $\w$ and $u$ such that for all $A\in\mathcal{A}^R(D)$ with $u\in GSBV^2(A)$ it holds that
\begin{equation*}
F^{\prime\prime}(\w)(u,1,A)\leq c\left(\int_A|\nabla u|^2\dx+\mathcal{H}^{d-1}(S_u\cap A)\right).
\end{equation*}
\end{lemma}

\begin{proof}
We compare the functionals with weak-membrane energies in the sense of an appropriate upper bound. To this end, let $v:\e\Lw\to [0,1]$ and fix an edge $(x,y)\in\mathcal{E}(\w)$. We assume without loss of generality that $v(\e x)\leq v(\e y)$. Then
\begin{equation*}
|v(\e x)-v(\e y)|=v(\e y)-v(\e x)\leq 1-v(\e x),
\end{equation*}
so that $|v(\e x)-v(\e y)|^2\leq (v(\e x)-1)^2$. From \eqref{neighbours} we then conclude that 
\begin{equation*}
\sum_{\substack{(x,y)\in\mathcal{E}(\w)\\ \e x,\e y\in A}}\e^{d-1}|v(\e x)-v(\e y)|^2\leq 2M\sum_{\e x\in\e\Lw\cap A}\e^{d-1}(v(\e x)-1)^2.
\end{equation*}
In particular, applying Proposition \ref{prop:ATandWeak}, for every $u:\e\Lw\to\R$ we deduce the upper bound 
\begin{align*}
\min_{v:\e\Lw\to [0,1]}F_{\e}(\w)(u,v,A) &\leq\min_{v:\e\Lw\to [0,1]}\Big(F_{\e}^{b}(\w)(u,v,A)+\frac{\beta(1+M)}{2}\sum_{\e x\in\e\Lw\cap A}\e^{d-1}(v(\e x)-1)^2\Big)\\
&= G_{\e,\beta(1+M)}(\w)(u,A).
\end{align*}
Hence the statement follows by comparison with the upper bound for weak-membrane energies (cf. Theorem \ref{t.weakmembrane} or \cite[Lemma 5.7]{R18}). Note that any sequence of optimal $v_{\e}$'s will convergence to $1$ since the energy remains bounded for any target function $u\in GSBV^2(A)$.
\end{proof}

The following technical lemma establishes an almost subadditivity of the set function $A\mapsto F^{\prime\prime}(\w)(u,A)$.
\begin{proposition}[Almost subadditivity]\label{subadd}
Let $A,B\in\Ard$. Moreover let $A^{\prime}\in\Ard$ be such that $A^{\prime}\subset\subset A$. Then, for all $u\in L^1(D)$,
\begin{equation*}
F^{\prime\prime}(\w)(u,1,A^{\prime}\cup B)\leq F^{\prime\prime}(\w)(u,1,A)+F^{\prime\prime}(\w)(u,1,B).
\end{equation*}
\end{proposition}
\begin{proof}
Let $A,A',B$ and $u$ be as in the statement. It suffices to consider the case where both $F''(\w)(u,1,A)$ and $F''(\w)(u,1,B)$ are finite. Moreover, Remark \ref{trunc} allows us to restrict to the case $u\in L^\infty(D)$. We choose sequences $(u_\e,v_\e),(\tilde{u}_\e,\tilde{v}_\e)\in\mathcal{PC}_{\e}^{\w}\times\mathcal{PC}_{\e}^{\w}$ both converging to $(u,1)$ in $L^1(D)\times L^1(D)$ and satisfying
\begin{equation}\label{recovery}
\limsup_{\e\to 0}F_{\e}(\w)(u_{\e},v_{\e},A)=F^{\prime\prime}(\w)(u,1,A),\qquad\limsup_{\e\to 0}F_{\e}(\w)(\tilde{u}_\e,\tilde{v}_{\e},B)=F^{\prime\prime}(\w)(u,1,B).
\end{equation}
In view of Remark \ref{trunc} we may further assume that $\|u_\e\|_{\infty},\|\tilde{u}_\e\|_\infty\leq\|u\|_\infty$. Hence, since also $0\leq v_\e,\tilde{v}_\e\leq 1$ we actually have $(u_\e,v_\e),(\tilde{u}_\e,\tilde{v}_\e)\to (u,1)$ in $L^2(D)\times L^2(D)$. 

For fixed $N\in\N$ we now construct a sequence $(\hat{u}_\e,\hat{v}_\e)\in\mathcal{PC}_\e^\w\times\mathcal{PC}_\e^\w$ converging to $(u,1)$ in $L^2(D)\times L^2(D)$ such that
\begin{equation}\label{fundest}
\limsup_{\e\to 0}F_{\e}(\w)(\hat{u}_\e,\hat{v}_\e,A'\cup B)\leq\left(1+\frac{C}{N}\right)\left(F''(\w)(u,1,A)+F''(\w)(u,1,B)\right),
\end{equation}
for some constant $C>0$ depending only on $A,A',B$ and $u$. Then the result follows by the arbitrariness of $N\in\N$. We will obtain the required sequence $(\hat{u}_\e,\hat{v}_\e)$ by a classical averaging procedure, adapting the construction in \cite[Proposition 5.2]{BBZ} to a stochastic lattice. To this end we first need to introduce some notation. We consider an auxiliary function $w_\e\in\mathcal{PC}_\e^\w$ defined as
\begin{align*}
w_\e(\e x):=\min\{v_\e(\e x),\tilde{v}_\e(\e x)\}\quad\mbox{for every $\e x\in\e\Lw\cap D$.}
\end{align*}
In particular, $w_{\e}\to 1$ in $L^1(D)$. Moreover, we fix $h\leq\dist(A',A^c)$ and for every $i\in\{1,\ldots,N\}$ we set
\begin{align*}
A_i:=\left\{x\in A\colon \dist(x,A')<\frac{ih}{2N}\right\},
\end{align*}
and we also introduce the layer-like sets
\begin{align*}
S_\e^i:=\{x\in A'\cup B\colon \dist(x,A_{i+2}\setminus A_{i-3})<2M\e\}.
\end{align*}
For every $i\in\{2,\ldots,N\}$ let $\Theta_i$ be a smooth cut-off function between the sets $A_{i-1}$ and $A_i$, \ie $\Theta_i\equiv 1$ on $A_{i-1}$, $\Theta_i\equiv 0$ on $\R^d\setminus A_{i}$ and $\|\nabla\Theta_i\|_\infty\leq\frac{4N}{h}$. 

For every $i\in\{4,\ldots,N-2\}$ we now define a pair $(\hat{u}_\e^i,\hat{v}_\e^i)\in\mathcal{PC}_\e^\w\times\mathcal{PC}_\e^\w$ by setting
\begin{align*}
\hat{u}_\e^i(\e x):=\Theta_i(\e x)u_\e(\e x)+(1-\Theta_i(\e x))\tilde{u}_\e(\e x),
\end{align*}
and
\begin{align*}
\hat{v}_\e^i(\e x):=
\begin{cases}
\Theta_{i-2}(\e x)v_\e(\e x)+(1-\Theta_{i-2}(\e x))w_\e(\e x) &\mbox{if $\e x\in \e\Lw\cap A_{i-2}$,}\\
w_\e(\e x) &\mbox{if $\e x\in\e\Lw\cap(A_{i+1}\setminus A_{i-2})$,}\\
\Theta_{i+2}(\e x)w_\e(\e x)+(1-\Theta_{i+2}(\e x))\tilde{v}_{\e}(\e x) &\mbox{if $\e x\in\e\Lw\cap (D\setminus A_{i+1})$.}
\end{cases}
\end{align*}
Note that for fixed $i\in\{4,\ldots,N-2\}$ we have $(\hat{u}_\e^i,\hat{v}_\e^i)\to(u,1)$ in $L^2(D)\times L^2(D)$ by convexity. Moreover, we can estimate $F_{\e}(\w)(\hat{u}_\e^i,\hat{v}_\e^i,A'\cup B)$ as
\begin{align}\label{est:subad1}
F_{\e}(\w)(\hat{u}_\e^i,\hat{v}_\e^i, A'\cup B) &\leq F_{\e}(\w)(u_\e,v_\e, A_{i-3})+F_{\e}(\w)(\tilde{u}_\e,\tilde{v}_\e,B\setminus\overline{A_{i+2}})+F_{\e}(\w)\left(\hat{u}_\e^i,\hat{v}_\e^i,S_{\e}^i\right)\nonumber\\
&\leq F_{\e}(\w)(u_\e,v_\e,A)+F_{\e}(\w)(\tilde{u}_\e,\tilde{v}_\e,B)+F_{\e}(\w)\left(\hat{u}_\e^i,\hat{v}_\e^i,S_{\e}^i\right).
\end{align}
Hence, in view of \eqref{recovery}, estimate \eqref{fundest} follows if we can show that the last term on the right hand side of \eqref{est:subad1} can be bounded by $C/N$ for a suitable choice of $i$. We start estimating the bulk term. First observe that for every pair $(x,y)\in\mathcal{E}(\w)$ it holds that
\begin{align}\label{eq:graduhat}
\hat{u}_\e^i(\e x)-\hat{u}_\e^i(\e y) &=\Theta_i(\e x)(u_\e(\e x)-u_\e(\e y))+(1-\Theta_i(\e x))(\tilde{u}_\e(\e x)-\tilde{u}_\e (\e y))\nonumber\\
&\hspace*{1em}+(\Theta_i(\e x)-\Theta_i(\e y))(u_\e(\e y)-\tilde{u}_\e(\e y)).
\end{align}
In addition, for every $\e x\in\e\Lw\cap D$ the properties of the cut-off function $\Theta_i$ imply that
\begin{align*}
\hat{v}_\e^i(\e x)\Theta_i(\e x)\leq v_\e(\e x)\quad\mbox{and}\quad\hat{v}_\e^i(\e x)(1-\Theta_i(\e x))\leq \tilde{v}_\e(\e x).
\end{align*}
Thus, using the mean-value theorem for $\Theta_i$ and the convexity inequality $(a+b+c)^2\leq 3(a^2+b^2+c^2)$, from \eqref{eq:graduhat} we obtain
\begin{align*}
\hat{v}_\e^i(\e x)^2\left|\frac{\hat{u}_\e^i(\e x)-\hat{u}_\e^i(\e y)}{\e}\right|^2 
&\leq 3\,\Big(v_\e(\e x)^2\left|\frac{u_\e(\e x)-u_\e(\e y)}{\e}\right|^2+\tilde{v}_\e(\e x)^2\left|\frac{\tilde{u}_\e(\e x)-\tilde{u}_\e(\e y)}{\e}\right|^2\Big)\\
&\hspace*{1em}+C\frac{N^2}{h^2}|u_\e(\e y)-\tilde{u}_ \e(\e y)|^2
\end{align*}
for every pair $(x,y)\in\mathcal{E}(\w)$ with $\e x,\e y\in S_\e^{i}$. Summing the above estimate over all such pairs $(x,y)$ we infer that
\begin{equation}\label{eq:subad:bulk2}
F_{\e}^{b}(\w)\left(\hat{u}_\e^i,\hat{v}_\e^i,S_\e^{i}\right) \leq 3\Big(F_{\e}^{b}(\w)\left(u_\e,v_\e,S_\e^{i}\right)+F_{\e}^{b}(\w)\left(\tilde{u}_\e,\tilde{v}_\e,S_\e^{i}\right)\Big)+CN^2\hspace{-1em}\sum_{\e y\in\e\Lw\cap S_\e^{i}}\hspace{-1em}\e^d|u_\e(\e y)-\tilde{u}_ \e(\e y)|^2.
\end{equation}
Next we consider the surface term. Since the function $x\mapsto (x-1)^2$ is convex, we obtain from the definition of $\hat{v}_{\e}^i$ that
\begin{equation}\label{eq:boundlocal}
(\hat{v}_{\e}^i(\e x)-1)^2\leq 
(v_{\e}(\e x)-1)^2+(\tilde{v}_{\e}(\e x)-1)^2.
\end{equation}
For the finite differences, observe that we can equivalently write
\begin{equation*}
w_{\e}(\e x)=
\begin{cases}
\Theta_{i-2}(\e x)v_{\e}(\e x)+(1-\Theta_{i-2}(\e x))w_{\e}(\e x) &\mbox{if $\e x\in\e\Lw\setminus A_{i-2}$,}
\\
\Theta_{i+2}(\e x)w_{\e}(\e x)+(1-\Theta_{i+2}(\e x))\tilde{v}_{\e}(\e x) &\mbox{if $\e x\in\e\Lw\cap A_{i+1}$.}
\end{cases}
\end{equation*}
Then, by the analogue of formula \eqref{eq:graduhat}, we can estimate
\begin{align*}
|\hat{v}_{\e}(\e x)-\hat{v}_{\e}(\e y)|^2\leq &3\left(|v_{\e}(\e x)-v_{\e}(\e y)|^2+|\tilde{v}_{\e}(\e x)-\tilde{v}_{\e}(\e y)|^2+|w_{\e}(\e x)-w_{\e}(\e y)|^2\right)
\\
&+C\frac{N^2}{h^2}\e^2|v_{\e}(\e y)-\tilde{v}_{\e}(\e y)|^2,
\end{align*}
where we used that the distance between the sets $\R^d\setminus A_{i+1}$ and $A_{i-2}$ is of order $\tfrac{h}{N}\gg M\e$ to reduce the number of possible interactions with respect to the case-by-case definition of $\hat{v}_{\e}^i$. Inserting the elementary inequality
\begin{equation*}
|w_\e(\e x)-w_\e(\e y)|^2\leq\max\left\{|v_\e(\e x)-v_\e(\e y)|^2,|\tilde{v}_\e(\e x)-\tilde{v}_\e(\e y)|^2\right\},
\end{equation*}
the above estimate can be continued to
\begin{equation}\label{eq:vgradbound}
|\hat{v}_{\e}(\e x)-\hat{v}_{\e}(\e y)|^2\leq 4\left(|v_{\e}(\e x)-v_{\e}(\e y)|^2+|\tilde{v}_{\e}(\e x)-\tilde{v}_{\e}(\e y)|^2\right)+C\frac{N^2}{h^2}\e^2|v_{\e}(\e y)-\tilde{v}_{\e}(\e y)|^2.
\end{equation}
Combining \eqref{eq:boundlocal} and \eqref{eq:vgradbound} and summing over all pairs $(x,y)\in\mathcal{E}(\w)$ gives
\begin{equation*}
F_{\e}^{s}(\w)(\hat{v}_{\e}^i,S_{\e}^i)\leq 4\left(F_{\e}^{s}(\w)(v_{\e},S_{\e}^i)+F_{\e}^{s}(\w)(\tilde{v}_{\e},S_{\e}^i)\right)+CN^2\e\sum_{\e y\in\e\Lw\cap S_{\e}^i}\e^d|v_{\e}(\e y)-\tilde{v}_{\e}(\e y)|^2.
\end{equation*}
Combining the above inequality with \eqref{eq:subad:bulk2} then yields
\begin{align}\label{eq:beforeaveraging}
F_{\e}(\w)(\hat{u}_{\e}^i,\hat{v}_{\e}^i,S_{\e}^i)\leq& 4\left(F_{\e}(\w)(u_{\e},v_{\e},S_{\e}^i)+F_{\e}(\w)(\tilde{u}_{\e},\tilde{v}_{\e},S_{\e}^i)\right)\nonumber
\\
&+CN^2\sum_{\e y\in\e\Lw\cap S_{\e}^i}\e^d\left(|u_{\e}(\e y)-\tilde{u}_{\e}(\e y)|^2+\e|v_{\e}(\e y)-\tilde{v}_{\e}(\e y)|^2\right).
\end{align}
We eventually notice that for every $i,j\in\{4,\ldots,N-2\}$ we have $S_\e^i\cap S_\e^j=\emptyset$ whenever $|i-j|>5$ and $\e>0$ is small enough. Moreover, for $i\in\{4\ldots,N-2\}$ we have $S_\e^i\subset B$ and $S_\e^i\wcont A$. Thus, summing \eqref{eq:beforeaveraging} over $i\in\{4,\ldots,N-2\}$ and averaging, we find $i(\e)\in\{4,\ldots,N-2\}$ satisfying
\begin{align*}
& F_{\e}(\w) \left(\hat{u}_\e^{i(\e)},\hat{v}_\e^{i(\e)},S_\e^{i(\e)}\right)\leq\frac{1}{N-5}\sum_{i=4}^{N-2}F_{\e}(\w)\left(\hat{u}_\e^{i},\hat{v}_\e^{i},S_\e^{i}\right)
\\
&\hspace{1em}\leq\frac{C}{N}(F_{\e}(\w)(u_\e,v_\e,A)+F_\e(\tilde{u}_\e,\tilde{v}_\e,B))+CN\|u_{\e}-\tilde{u}_{\e}\|^2_{L^2(A)}+CN\e\|v_{\e}-\tilde{v}_{\e}\|^2_{L^2(A)},
\end{align*}
where we used Remark \ref{voronoi} to pass from the sum to the integral norms. Since $u_\e$ and $\tilde{u}_\e$ have the same limit in $L^2(D)$ and $0\leq v_\e,\tilde{v}_\e\leq 1$, thanks to \eqref{recovery} we obtain the required sequence satisfying \eqref{fundest} by setting $(\hat{u}_\e,\hat{v}_\e):=\big(\hat{u}_\e^{i(\e)},\hat{v}_\e^{i(\e)}\big)$.
\end{proof}
The next lemma is a standard consequence of Proposition \ref{subadd} and Lemma \ref{bounds}. A proof can be found for example in \cite[Lemma 5.9]{R18}.
\begin{lemma}[Inner regularity]\label{almostmeasure}
Let $u\in L^1(D)$. Then for any $A\in\Ard$ it holds that
\begin{equation*}
F^{\prime\prime}(\w)(u,1,A)=\sup_{A^{\prime}\subset\subset A}F^{\prime\prime}(\w)(u,1,A^{\prime}).
\end{equation*}
\end{lemma}
\noindent Now we are in a position to establish the main result of this subsection.
\begin{proof}[Proof of Proposition \ref{limitsbvp}]
Having at hand Remark \ref{trunc} and Lemmata \ref{local}, \ref{compact}, \ref{bounds}, and \ref{almostmeasure} as well as Proposition \ref{subadd}, the well-known arguments how to apply the general integral representation theorem \cite[Theorem 1]{BFLM} in order to conclude, can be found, \eg in \cite[Proposition 5.2]{R18}. 
\end{proof}  

\subsection{Characterization of the bulk density}
In this subsection we argue that the function $h$ given by Proposition \ref{limitsbvp} agrees with the density of the $\Gamma$-limit of the sequence of discrete quadratic functionals $u\mapsto F_{\e}^{b}(\w)(u,1,D)$ defined in \eqref{eq:defbulk}. 

\begin{proposition}[Characterization of the bulk density]\label{p.gradientpartsequal}
Let $\e_n$ and $F(\w)$ be as in Proposition \ref{limitsbvp}. Then the $\Gamma$-convergence result of Theorem \ref{ACGmain} holds along the sequence $\e_n$ and for a.e. $x_0\in D$ and every $\xi\in\R^{d}$ it holds that
\begin{equation*}
|B_1|h(\w,x_0,\xi)=\lim_{\varrho\to 0}\varrho^{-d}F(\w)(u_{x_0,\xi},1,B_{\varrho}(x_0))=|B_1|f(\w,x_0,\xi),
\end{equation*}
where $f(\w,\cdot,\cdot)$ is an (equivalent) integrand of the $\Gamma$-limit of $F_{\e_n}^b(\w)(\cdot,1,D)$ (cf. Theorem \ref{ACGmain}).	
\end{proposition}
\begin{proof}
The first equality characterizing the function $h$, which does not rely on the discrete functionals, but only on the structure and growth of the continuum limit, can be proven as in \cite[Lemma 5.11]{R18}. Hence we only prove the second inequality. By Theorem \ref{ACGmain}, upon passing temporarily to a further subsequence (not relabeled), we may assume that the sequence $F_{\e_n}^b(\w)(\cdot,1,D)$ $\Gamma$-converges to some integral functional $F^b(\w)(\cdot,D)$ with density $f(\w,\cdot,\cdot)$. Fix $x_0\in D$ satisfying the first equality.

Since $v_{\e}\equiv 1$ is an admissible phase-field for any trial recovery sequence of the affine function $u_{x_0,\xi}$ and $F_{\e}(\w)(u,1,B_{\rho}(x_0))=F_{\e}^{b}(\w)(u,1,B_{\rho}(x_0))$ for every $u\in\mathcal{PC}_{\e}^{\w}$, we deduce that 
\begin{equation*}
\varrho^{-d}F(\w)(u_{x_0,\xi},1,B_{\varrho}(x_0))\leq |B_1|\dashint_{B_{\rho}(x_0)}f(\w,x,\xi)\,\mathrm{d}x.
\end{equation*}
In order to prove the reverse inequality, note that due to Proposition \ref{prop:ATandWeak} we have
\begin{equation*}
\varrho^{-d}F_{\e}(\w)(u,v,B_{\varrho}(x_0))\geq \varrho^{-d}G_{\e,\beta}(\w)(u,B_{\varrho}(x_0)).
\end{equation*}
Hence, possibly passing to a further subsequence, we obtain that 
\begin{align*}
\varrho^{-d}F(\w)(u_{x_0,\xi},1,B_{\varrho}(x_0))\geq\varrho^{-d}\, \Gamma\hbox{-}\lim_{\e_n\to 0}G_{\e_n,\beta}(\w)(u_{x_0,\xi},B_{\varrho}(x_0))
=|B_1|\dashint_{B_{\varrho}(x_0)}f(\w,x,\xi)\dx,
\end{align*}
where the last equality follows from Theorem~\ref{t.weakmembrane} and the fact that $u_{x_0,\xi}\in W^{1,2}_{\rm loc}(\R^d)$. 
Using the uniform local Lipschitz continuity of $f$ in the third variable (which is a consequence of the quadratic dependence and local boundedness) one can pass to the limit in $\rho$ by Lebesgue's differentiation theorem except for a null set independent of $\xi$, which yields
\begin{equation*}
\lim_{\varrho\to 0}\varrho^{-d}F(\w)(u_{x_0,\xi},1,B_{\varrho}(x_0))= |B_1|f(\w,x_0,\xi)
\end{equation*}
for a.e. $x_0\in D$ and every $\xi\in\R^d$. 

Hence we proved the claim along the chosen subsequence. In particular, along any subsequence of $\e_n$ the $\Gamma$-limit of $F_{\e}^{b}(\w)(\cdot,1,D)$ is uniquely defined by the integrand $h(\w,x,\xi)$, so that the $\Gamma$-limit along the sequence $\e_n$ exists by the Urysohn-property of $\Gamma$-convergence, although the integrand might differ on a negligible set depending on the subsequence.
\end{proof}

\subsection{Characterization of the surface density} 

Having identified the bulk term, we now show that the surface integrand $\varphi(\w,x,a,\nu)$ can be computed with the discrete functional $F_{\e}(\w)$ restricted to functions $u$ taking only the two values $a$ and $0$ and functions $v$ that vanish on all couples $(\e x,\e y)$ where $u$ jumps. This implies in particular that along such sequences $F_{\e}^{b}(\w)(u,v)=0$, so that $F_{\e}(\w)(u,v)=F_{\e}^{s}(\w)(v)$. Nevertheless the variable $u$ enters the procedure in the form of a non-convex constraint (cf. \eqref{def:surf:int}).

We first study the asymptotic minimization problems given by Proposition \ref{limitsbvp} and their connection to boundary value problems for the discrete functionals $F_{\e}(\w)$. As a first step, we compare the two quantities
\begin{equation}\label{def:med}
\begin{split}
m_{\e,\delta}^{\w}(\bar{u},\bar{v},A)&=\inf\{F_{\e}(\w)(u,v,A):\;(u,v)\in\mathcal{PC}^{\w}_{\e,\delta}(\bar{u},A)\times\mathcal{PC}^{\w}_{\e,M\e}(\bar{v},A)\},\\
m^\w(\bar{u},A)&=\inf\{F(\w)(u,1,A):\;u\in SBV^2(A),\,u=\bar{u}\text{ in a neighborhood of }\partial A\},
\end{split}
\end{equation}
where the limit functional $F(\w)$ is given (up to subsequences) by Proposition \ref{limitsbvp} and $\mathcal{PC}_{\e,\delta}^{\w}(\bar{u},A)$ is as in \eqref{def:discreteboundarycond}. 
Along the subsequence $\e_n$ provided by Proposition \ref{limitsbvp} we can prove the following result about the asymptotic behavior of $m_{\e_n,\delta}^{\w}(\bar{u},\bar{v},Q)$ on cubes $Q=Q_{\nu}(x_0,\varrho)$ when first $\e_n\to 0$ and then $\delta\to 0$.
\begin{lemma}[Approximation of minimum values]\label{approxminprob}
Let $\e_n$ and $F(\w)$ be as in Proposition \ref{limitsbvp}. Then, for $u_{x_0,\nu}^{a,0}$ as in \eqref{eq:purejump} and $v_{x_0,\nu}^{\e}$ given by \eqref{eq:barvboundary}, it holds that
\begin{align*}
m^\w(u_{x_0,\nu}^{a,0},Q_{\nu}(x_0,\varrho))&=\lim_{\delta\to 0}\liminf_{n\to+\infty} m_{\e_n,\delta}^\w(u_{x_0,\nu}^{a,0},v_{x_0,\nu}^{\e_n},Q_{\nu}(x_0,\rho))
\\
&=\lim_{\delta\to 0}\limsup_{n\to+\infty}m_{\e_n,\delta}^\w(u_{x_0,\nu}^{a,0},v_{x_0,\nu}^{\e_n},Q_{\nu}(x_0,\varrho))
\end{align*}
with the cube $Q_{\nu}(x_0,\varrho)$ defined in \eqref{eq:defcube} and the succeeding line.
\end{lemma}
\begin{proof}
By monotonicity the limits with respect to $\delta$ exist. To reduce notation, we replace $\e_n$ by $\e$ in what follows and write $Q=Q_{\nu}(x_0,\varrho)$. Moreover, we set $\bar{u}:=u_{x_0,\nu}^{a,0}$ and $\bar{v}_\e:=v_{x_0,\nu}^{\e}$. For every $\e>0$ let $u_{\e}\in\mathcal{PC}^{\w}_{\e,\delta}({\bar{u}},Q)$ and $v_{\e}\in\mathcal{PC}_{\e,M\e}^{\w}(\bar{v}_{\e},Q)$ be such that $m_{\e,\delta}^\w({\bar{u}},\bar{v}_{\e},Q)=F_{\e}(\w)(u_{\e},v_{\e},Q)$. Note that these minimizers exist as the optimization problem is finite dimensional. Due to Remark \ref{trunc} we can assume without loss of generality that $|u_{\e}(\e x)|\leq |a|$ for all $x\in\Lw$. Testing the pointwise evaluation of the functions $\bar{u}$ and $\bar{v}_{\e}$ as competitors for the minimization problem, we see that for $\e$ small enough
\begin{equation}\label{eq:bounded}
F_{\e}(\w)(u_{\e},v_{\e},Q)\leq F_{\e}(\w)(\bar{u},\bar{v}_{\e},Q)\leq F_{\e}^{s}(\w)(\bar{v}_{\e},Q)\leq C,
\end{equation}
where in the second inequality we used the implication
\begin{equation}\label{eq:implication}
u_{x_0,\nu}^{a,0}(\e x)\neq u_{x_0,\nu}^{a,0}(\e y)\quad\;\implies\quad\;\begin{cases}
|\langle \e x-x_0,\nu\rangle|\leq |\e x-\e y|\leq M\e
\\
|\langle \e y-x_0,\nu\rangle|\leq |\e x-\e y|\leq M\e
\end{cases}
\quad\;\implies\quad\; \bar{v}_{\e}(\e x)=\bar{v}_{\e}(\e y)=0.
\end{equation}
and the last bound in \eqref{eq:bounded} follows from counting lattice points in an $2M\e$ tubular neighborhood of the hyperplane $H_\nu(x_0)$. Hence Lemma \ref{compact} yields that, up to a subsequence (not relabeled), $u_{\e}\to u$ in $L^1(Q)$ for some $u\in SBV^2(Q)$ (recall the $L^{\infty}$-bound) and $v_{\e}\to 1$ in $L^1(Q)$. Using Remark \ref{voronoi}, we infer that $u={\bar{u}}$ on $(\R^d\setminus Q)+B_{\delta}(0)$. Consequently $u$ is admissible in the infimum problem defining $m^\w({\bar{u}},Q)$ and and the $\Gamma$-convergence result of Proposition \ref{limitsbvp} yields
\begin{equation*}
m^\w({\bar{u}},Q)\leq F(\w)(u,1,Q)\leq\liminf_{\e} F_{\e}(\w)(u_{\e},v_{\e},Q)\leq\liminf_{\e} m_{\e,\delta}^\w({\bar{u}},\bar{v}_{\e},Q).
\end{equation*}
As $\delta>0$ was arbitrary, we conclude that $m^\w({\bar{u}},Q)\leq\lim_{\delta\to 0}\liminf_{\e}m_{\e,\delta}^\w({\bar{u}},\bar{v}_{\e},Q)$.
	
In order to prove the second inequality, for given $\theta>0$ we let $u\in SBV^2(D)$ be such that $u={\bar{u}}$ in a neighborhood of $\partial Q$ and $F(\w)(u,1,Q)\leq m^\w({\bar{u}},Q)+\theta$. By Remark \ref{trunc} we can also assume that $u\in L^{\infty}(D)$. Due to $\Gamma$-convergence we find $u_{\e},v_{\e}\in\mathcal{PC}^{\w}_{\e}$ converging to $u$ and $1$ in $L^2(D)$ (again we rely on Remark \ref{trunc}) and such that
\begin{equation}\label{eq:rec}
\lim_{\e\to 0}F_{\e}(\w)(u_{\e},v_{\e},Q)=F(\w)(u,1,Q).
\end{equation}
Our goal is to modify both sequences such that they attain the discrete boundary conditions. The argument is closely related to the proof of Proposition \ref{subadd}, so we just sketch some parts. Since $u=\bar{u}$ in a neighborhood of $\partial Q$, we find equally oriented cubes $Q^{\prime}\subset\subset Q^{\prime\prime}\subset\subset Q$ with
\begin{equation}\label{eq:uonlayer}
u={\bar{u}} \quad\text{ on }Q\setminus Q^{\prime}. 
\end{equation}
Fix $N\in\mathbb{N}$. For $h\leq \dist(Q^{\prime},\partial Q^{\prime\prime})$ and $i\in\{1,\dots,N\}$ we define the sets
\begin{equation*}
Q_i:=\left\{x\in Q:\;\dist(x,Q^{\prime})<i\frac{h}{2N}\right\}
\end{equation*}
and consider an associated cut-off function $\Theta_i\in C^{\infty}_c(Q_{i},[0,1])$ such that $\Theta_i\equiv 1$ on $Q_{i-1}$ and $\|\nabla\Theta_i\|_{\infty}\leq \frac{4N}{h}$. Set $w_{\e}=\min\{v_{\e},\bar{v}_{\e}\}$ and define $u^i_{\e},v_{\e}^i\in \mathcal{PC}_{\e}^{\w}$ by
\begin{equation*}
\hat{u}^i_{\e}(\e x)=\Theta_i(\e x)u_{\e}(\e x)+(1-\Theta_i(\e x)){\bar{u}}(\e x)
\end{equation*}
and
\begin{equation*}
\hat{v}_\e^i(\e x):=
\begin{cases}
\Theta_{i-2}(\e x)v_\e(\e x)+(1-\Theta_{i-2}(\e x))w_\e(\e x) &\mbox{if $\e x\in \e\Lw\cap Q_{i-2}$,}\\
w_\e(\e x) &\mbox{if $\e x\in\e\Lw\cap(Q_{i+1}\setminus Q_{i-2})$,}\\
\Theta_{i+2}(\e x)w_\e(\e x)+(1-\Theta_{i+2}(\e x))\bar{v}_{\e}(\e x) &\mbox{if $\e x\in\e\Lw\cap (D\setminus Q_{i+1})$.}
\end{cases}
\end{equation*}
Since we may assume that $u_{|D\setminus Q}={\bar{u}}$, by \eqref{eq:uonlayer} we have that $u^i_{\e}\to u$ in $L^1(D)$. Moreover, also $v^i_{\e}\to 1$ in $L^1(D)$ for all $i\in\{4,\dots,N-2\}$. Setting $S_{\e}^{i}:=\{x\in Q:\;\dist(x,Q_{i+2}\setminus{Q_{i-3}})<2M\e\}$, the energy can be estimated via
\begin{align}\label{eq:splitineq}
F_{\e}(\w)(\hat{u}^i_{\e},\hat{v}_{\e}^i,Q)&\leq F_{\e}(\w)(u_{\e},v_{\e},Q_{i-3})+F_{\e}(\w)(\bar{u},\bar{v}_{\e},Q\setminus\overline{Q_{i+2}})
+F_{\e}(\w)(\hat{u}_{\e}^i,\hat{v}_{\e}^i,S_{\e}^i)\nonumber
\\
&\leq F_{\e}(\w)(u_{\e},v_{\e},Q)+F_{\e}^{s}(\w)(\bar{v}_{\e},Q\setminus\overline{Q'})
+F_{\e}(\w)(\hat{u}_{\e}^i,\hat{v}_{\e}^i,S_{\e}^i),
\end{align}
where we used again \eqref{eq:implication}. The behavior of first term in the last line is controlled by \eqref{eq:rec}. In order to bound the second one, note that the structure of $\bar{v}_{\e}$ (cf. \eqref{eq:barvboundary}) and Remark \ref{voronoi} imply that
\begin{align}\label{eq:counting}
F_{\e}^{s}(\w)(\bar{v}_{\e},Q\setminus \overline{Q'})&\leq C \e^{d-1}\#\left\{\e x\in\e\Lw\cap Q\setminus \overline{Q'}:\dist(\e x,H_{\nu}(x_0))\leq 2M\e\right\}\nonumber
\\
&\leq C \frac{1}{\e} \left|(Q\setminus Q')\cap H_{\nu}(x_0)+B_{2M\e}(0)\right|.
\end{align}
Since the set $(\overline{Q\setminus Q'})\cap H_{\nu}(x_0)$ admits a $(d-1)$-dimensional Minkowski content that agrees (up to a multiplicative constant) with the Hausdorff measure of the closure, we conclude that
\begin{equation}\label{eq:closetobd}
\limsup_{\e\to 0}F_{\e}^{s}(\w)(\bar{v}_{\e},Q\setminus \overline{Q'})\leq C\mathcal{H}^{d-1}((Q\setminus Q')\cap H_{\nu}(x_0)),
\end{equation} 
where we used that $\mathcal{H}^{d-1}(\partial Q\cap H_{\nu}(x_0))=0$. For the last term $F_{\e}(\w)(\hat{u}_{\e}^i,\hat{v}_{\e}^i,S_{\e}^i)$ in \eqref{eq:splitineq} one can use the same arguments already used to prove \eqref{eq:beforeaveraging} in order to show that
\begin{align*}
F_{\e}(\w)(\hat{u}_{\e}^i,\hat{v}_{\e}^i,S_{\e}^i)\leq& C\left(F_{\e}(\w)(u_{\e},v_{\e}, S_{\e}^i)+F_{\e}(\w)(\bar{u},\bar{v}_{\e},S_{\e}^i)\right)
\\
&+CN^2\sum_{\e y\in \e\Lw\cap S^i_{\e}}\e^d\left(|u_{\e}(\e x)-\bar{u}(\e x)|^2+\e|v_{\e}(\e y)-\bar{v}_{\e}(\e y)|^2\right).
\end{align*}
By construction we have $S_{\e}^i\cap S_{\e}^j=\emptyset$ for $|i-j|> 5$ and $S_{\e}^i\subset\subset Q\setminus \overline{Q^{\prime}}$ for $i\in\{4,\dots,N-2\}$. Averaging the previous inequality we find an index $i(\e)\in\{4,\dots,N-2\}$ such that
\begin{align*}
F_{\e}(\w)(\hat{u}_{\e}^{i(\e)},\hat{v}_{\e}^{i(\e)},S_{\e}^{i(\e)})\leq&\frac{1}{N-5}\sum_{i=4}^{N-2}F_{\e}(\w)(\hat{u}_{\e}^i,\hat{v}_{\e}^i,S_{\e}^i)
\\
\leq&\frac{C}{N}\big(F_{\e}(\w)(u_{\e},v_{\e},Q)+F_{\e}^{s}(\w)(\bar{v}_{\e},Q\setminus \overline{Q^{\prime}})\big)
\\
&+CN\left(\|u_{\e}-{\bar{u}}_{\e}\|^2_{L^2(Q\setminus Q^{\prime})}+\e\|v_{\e}-\bar{v}_{\e}\|^2_{L^2(Q\setminus Q^{\prime})}\right).
\end{align*}
Due to \eqref{eq:uonlayer} we have that $u_{\e}-\bar{u}_{\e}\to 0$ in $L^2(Q\setminus Q^{\prime})$. Moreover, $\hat{u}_{\e}^{i(\e)}(\e x)=\bar{u}(\e x)$ and $\hat{v}_{\e}^{i(\e)}=\bar{v}_{\e}(\e x)$ for all $\e x\in \e\Lw\cap Q\setminus Q^{\prime\prime}$, so that $\hat{u}^{i(\e)}_{\e}\in\mathcal{PC}_{\e,\delta}^{\w}(\bar{u},Q)$ and $\hat{v}_{\e}^{i(\e)}\in\mathcal{PC}_{\e,M\e}(\bar{v}_{\e},Q)$ for all $\e,\delta>0$ small enough. Hence from \eqref{eq:rec}, \eqref{eq:splitineq}, and \eqref{eq:closetobd} we deduce that
\begin{align*}
\limsup_{\e} m_{\e,\delta}^\w({\bar{u}},\bar{v}_\e,Q)&\leq\limsup_{\e} F_{\e}(\w)(\hat{u}^{i(\e)}_{\e},\hat{v}_{\e}^{i(\e)},Q)
\\
&\leq \left(1+\frac{C}{N}\right)\left(m^\w(\bar{u},Q)+\theta+\mathcal{H}^{d-1}((Q\setminus Q^{\prime})\cap H_{\nu}(x_0))\right)
\end{align*}
As $\theta>0$ was arbitrary, the claim follows letting first $\delta\to 0$, then $N\to +\infty$ and finally $Q^{\prime}\uparrow Q$.
\end{proof}

Our next aim is to provide a simplified form of the discrete minimization problem that is suitable for subadditivity estimates. To this end we will compare the two quantities $m_{\e,\delta}^{\w}(u_{x_0,\nu}^{a,0},v_{x_0,\nu}^{\e},Q_{\nu}(x_0,\varrho))$ and $\varphi_{\e,\delta}^\w(u_{x_0,\nu}^{a,0},Q_\nu(x_0,\varrho))$ given by \eqref{def:surf:int}.
Namely, we show that we have the following equivalent characterization for the surface density.
\begin{lemma}[Construction of a competitor for $\varphi_{\e,\delta}^\w$]\label{separationofscales1}
Let $\e_n\to 0$. Then, for all $x_0\in D$, all $a\in\R$ and all $\nu\in S^{d-1}$ it holds that
\begin{equation*}
\limsup_{\varrho\to 0}\varrho^{1-d}\lim_{\delta\to 0}\limsup_{n\to+\infty} \varphi_{\e_n,\delta}^{\w}(u_{x_0,\nu}^{a,0},Q_{\nu}(x_0,\varrho))
\\
=\limsup_{\varrho\to 0}\varrho^{1-d}\lim_{\delta\to 0}\limsup_{n\to+\infty}m_{\e_n,\delta}^{\w}(u_{x_0,\nu}^{a,0},v_{x_0,\nu}^{\e_n},Q_{\nu}(x_0,\varrho)).
\end{equation*}
\end{lemma}
\begin{remark}\label{r.jumpindependence}
The condition $F_{\e}^{b}(\w)(u,v,Q_{\nu}(x_0,\varrho))=0$ in the definition of $\varphi_{\e,\delta}^{\w}(u_{x_0,\nu}^{a,0},Q_\nu(x_0,\varrho))$ implies that the latter is independent of the jump opening $a$. More precisely, for every $a\in\R$ we have $\varphi_{\e,\delta}^{\w}(u_{x_0,\nu}^{a,0},Q_\nu(x_0,\varrho))=\varphi_{\e,\delta}^{\w}(u_{x_0,\nu}^{1,0},Q_\nu(x_0,\varrho))$. Thus, combining \eqref{derivationformula} with Lemma~\ref{approxminprob} and Lemma~\ref{separationofscales1} above, we obtain that the surface integrand $\varphi$ in Theorem \ref{limitsbvp} is given by
\begin{equation*}
\varphi(\omega,x_0,a,\nu)=\varphi(\omega,x_0,1,\nu)=\limsup_{\varrho\to 0}\varrho^{1-d}\lim_{\delta\to 0}\limsup_{n\to+\infty}\varphi_{\e_n,\delta}^{\w}(u_{x_0,\nu}^{1,0},v_{x_0,\nu}^{\e_n},Q_{\nu}(x_0,\varrho))
\end{equation*}
for $x_0\in D$, $a\in\R$, and $\nu\in S^{d-1}$.
\end{remark}
\begin{proof}[Proof of Lemma~\ref{separationofscales1}]
Note that it suffices to bound the left hand side from above by the right hand side. To reduce notation, we set $Q_{\varrho}:=Q_{\nu}(x_0,\varrho)$ and write $\e$ instead of $\e_n$. If $a=0$ then both sides are zero. Thus we assume that $a> 0$ (the case $a<0$ can be treated similarly). Fix $(u_{\e},v_{\e})\in\mathcal{PC}_{\e,\delta}^{\w}(u^{a,0}_{x_0,\nu},Q_{\varrho})\times \mathcal{PC}_{\e,M\e}^{\w}(v^{\e}_{x_0,\nu},Q_{\varrho})$ such that
\begin{equation}\label{eq:almostopt}
F_{\e}(\w)(u_{\e},v_{\e},Q_{\varrho})\leq C\varrho^{d-1},
\end{equation}
which exists at least for small $\e$ taking for instance $u_{\e}=u_{x_0,\nu}^{a,0}$ and $v_{\e}=v_{x_0,\nu}^{\e}$. In particular, $0\leq v_{\e}\leq 1$. In what follows we construct sequences $\tilde{u}_{\e}\in \mathcal{S}_{\e,\delta}^{\w}(u^{a,0}_{x_0,\nu},Q_{\varrho})$ and $\tilde{v}_{\e}\in\mathcal{PC}_{\e,M\e}^{\w}(v_{x_0,\nu}^{\e},Q_{\rho})$ such that
\begin{equation}\label{eq:firsttermzero}
\sum_{\substack{(x,y)\in \mathcal{E}(\w)\\ \e x,\e y\in Q_{\varrho}}}\tilde{v}_{\e}(\e x)^2|\tilde{u}_{\e}(\e x)-\tilde{u}_{\e}(\e y)|^2=0
\end{equation} 
and which have almost the same energy. We fix $\eta\in (0,1/2)$ and consider the set of points
\begin{equation*}
L_{v_{\e}}(\eta):=\{\e x\in\e\Lw\cap Q_{\varrho}:\;v_{\e}(\e x)>\eta\}.
\end{equation*} 
For $t\in\R$ we define
\begin{equation*}
L_{u_{\e}}(t):=\{\e x\in\e\Lw\cap Q_{\varrho}:\;u_{\e}(\e x)>t\}.
\end{equation*}
To reduce notation, we also introduce the set 
\begin{equation*}
\mathcal{R}_{\e}(t):=\{(x,y)\in\mathcal{E}(\w):\;\e x\in Q_{\varrho}\cap L_{u_\e}(t),\,\e y\in Q_{\varrho}\setminus L_{u_\e}(t)\text{ or vice versa}\}.
\end{equation*}
Observe that for $(x,y)\in\mathcal{E}(\w)$ with $\e x,\e y\in Q_{\varrho}$ we have $(x,y)\in \mathcal{R}_{\e}(t)$ if and only if $t\in [u_{\e}(\e x),u_{\e}(\e y))$ or $t\in [u_{\e}(\e y),u_{\e}(\e x))$. Hence for such $(x,y)$ the following coarea-type estimate holds true:
\begin{equation*}
\int_{0}^{a}|v_{\e}(\e x)|\mathds{1}_{\{(x,y)\in\mathcal{R}_{\e}(t)\}}\,\mathrm{d}t\leq|v_{\e}(\e x)||u_{\e}(\e x)-u_{\e}(\e y)|. 
\end{equation*}
Summing this estimate, we infer from H\"older's inequality that
\begin{align*}\label{coarea}
\int_{0}^{a}\sum_{(x,y)\in \mathcal{R}_{\e}(t)}\e^{d-1}|v_{\e}(\e x)|\,\mathrm{d}t&\leq \sum_{\substack{ (x,y)\in\mathcal{E}(\w)\\ \e x,\e y\in Q_{\varrho}}}\e^{d}|v_{\e}(\e x)|\Big|\frac{u_{\e}(\e x)-u_{\e}(\e y)}{\e}\Big|
\\
&\leq C\e^{\frac{d}{2}}(\#(\e\Lw\cap Q_{\varrho}))^{\frac{1}{2}}\bigg(\sum_{\substack{ (x,y)\in\mathcal{E}(\w)\\ \e x,\e y\in Q_{\varrho}}}\e^{d}v_{\e}(\e x)^2\Big|\frac{u_{\e}(\e x)-u_{\e}(\e y)}{\e}\Big|^2\bigg)^{\frac{1}{2}}.
\end{align*}
The last sum is bounded by the energy, while for $\e=\e(\varrho)$ small enough the cardinality term can be bounded via $\#(\e\Lw\cap Q_{\varrho})\leq C(\varrho\e^{-1})^d$. Hence in combination with \eqref{eq:almostopt} we obtain 
\begin{equation*}
\int_{0}^{a}\sum_{(x,y)\in \mathcal{R}_{\e}(t)}\e^{d-1}|v_{\e}(\e x)|\,\mathrm{d}t\leq  C\varrho^{d-\frac{1}{2}}.
\end{equation*}
From this inequality we deduce the existence of some $t_{\e}\in (0,a)$ such that
\begin{equation}\label{eq:goodchoice}
\sum_{(x,y)\in \mathcal{R}_{\e}(t_{\e})}\e^{d-1}|v_{\e}(\e x)|\leq Ca^{-1}\varrho^{d-\frac{1}{2}}.
\end{equation}
Define $\tilde{u}_{\e}$ and $\tilde{v}_{\e}$ by its values on $\e\Lw$ setting
\begin{align*}
\tilde{u}_{\e}(\e x):=&
\begin{cases}
0 &\mbox{if $u_{\e}(\e x)\leq t_{\e}$,}\\
a &\mbox{if $u_{\e}(\e x)>t_{\e}$.}
\end{cases}
\\
\tilde{v}_{\e}(\e x):=&
\begin{cases}
0 &\mbox{if $(x,y)\in \mathcal{R}_{\e}(t_{\e})$ for some $\e y\in\e\Lw$,} 
\\
v_{\e}(\e x) &\mbox{otherwise.}
\end{cases}
\end{align*}
As $t_{\e}\in (0,a)$, the boundary conditions imposed on $u_{\e}$ imply that the function $\tilde{u}_{\e}$ satisfies $\tilde{u}_{\e}(\e x)=u_{x_0,\nu}^{a,0}(\e x)$ for all $\e x\in\e\Lw\cap\partial_{\delta}Q_{\varrho}$, so that $\tilde{u}_{\e}\in\mathcal{S}_{\e,\delta}^{\w}(u_{x_0,\nu}^{a,0},Q_{\varrho})$ as claimed. Moreover, whenever $\dist(\e x,\R^d\setminus  Q_{\varrho})\leq M\e$, then for all $\e y\in\e\Lw$ with $(x,y)\in\mathcal{E}(\w)$ we have $\dist(\e y,\R^d\setminus Q_{\varrho})\leq 2M\e\ll\delta$. Hence the boundary conditions on $u_{\e}$ are active and $(x,y)\in \mathcal{R}_{\e}(t_{\e})$ implies that $|\langle \e x-x_0,\nu\rangle|\leq M\e$, so that $v_{\e}(\e x)=0$. Consequently $\tilde{v}_{\e}(\e x)=v_{\e}(\e x)$ and therefore $\tilde{v}_{\e}\in\mathcal{PC}_{\e,M\e}(v_{x_0,\nu}^{\e},Q_{\varrho})$. In order to verify condition \eqref{eq:firsttermzero}, observe that for any pair $(x,y)\in\mathcal{E}(\w)$ with $\e x,\e y\in Q_{\varrho}$ we have $\tilde{u}_{\e}(\e x)\neq \tilde{u}_{\e}(\e y)$ if and only if $(x,y)\in\mathcal{R}_{\e}(t_{\e})$, so that by its very definition $v_{\e}(\e x)=0$. Hence \eqref{eq:firsttermzero} holds true. Next we estimate the energy difference. Recall that $0\leq v_{\e},\tilde{v}_{\e}\leq 1$. We first estimate the energy term involving the discrete gradients of $\tilde{v}_\e$. Consider first the case when $\tilde{v}_{\e}(\e x)=0\neq v_{\e}(\e x)$ and $\tilde{v}_{\e}(\e y)=v_{\e}(\e y)$. Then  we have
	\begin{align*}
	|\tilde{v}_{\e}(\e x)-\tilde{v}_{\e}(\e y)|^2=&|v_{\e}(\e y)|^2
	\\
	\leq&
	\begin{cases}
	(1+\eta)|v_{\e}(\e x)-v_{\e}(\e y)|^2+\left(1+\frac{1}{\eta}\right)\eta^2 &\mbox{if $\e x\notin L_{v_{\e}}(\eta)$,}
	\\
	1 &\mbox{if $\exists \,\e x^{\prime}\in \e\Lw:\,(x,x^{\prime})\in\mathcal{R}_{\e}(t_{\e})$.}
	\end{cases}
	\end{align*}
	The symmetric conclusion holds true when we exchange the roles of $x$ and $y$. In all remaining cases we have $|\tilde{v}_{\e}(\e x)-\tilde{v}_{\e}(\e y)|\leq |v_{\e}(\e x)-v_{\e}(\e y)|$. Hence we obtain the global bound
	\begin{align}\label{eq:gradv}
	\sum_{\substack{ (x,y)\in\mathcal{E}(\w)\\ \e x,\e y\in Q_{\varrho}}}\e^{d-1}|\tilde{v}_{\e}(\e x)-\tilde{v}_{\e}(\e y)|^2\leq& (1+\eta)\sum_{\substack{ (x,y)\in\mathcal{E}(\w)\\ \e x,\e y\in Q_{\varrho}}}\e^{d-1}|v_{\e}(\e x)-v_{\e}(\e y)|^2\nonumber
	\\
	&+C\eta \e^{d-1}\# \left(\e\Lw\cap Q_{\varrho}\setminus L_{v_{\e}}(\eta)\right)\nonumber
	\\
	&+C\e^{d-1} \#\{(x,y)\in\mathcal{R}_{\e}(t_{\e}):\,\e x\in L_{v_{\e}}(\eta)\}.
	\end{align}
	Next we bound the 'singe-well'-term. Since the function $x\mapsto (x-1)^2$ is $2$-Lipschitz on $[0,1]$, we obtain
	\begin{equation*}
	(\tilde{v}_{\e}(\e x)-1)^2\leq 
	\begin{cases}
	1 &\mbox{if $\exists\,\e y\in\e\Lw: \,(x,y)\in\mathcal{R}_{\e}(t_{\e})$,}
	\\
	(v_{\e}(\e x)-1)^2+2\eta &\mbox{if $\e x\in\e\Lw\cap Q_{\varrho}\setminus L_{v_{\e}}(\eta)$,}
	\\
	(v_{\e}(\e x)-1)^2 &\mbox{otherwise.}	
	\end{cases}
	\end{equation*}
	Summing this estimate over all $\e x\in\e\Lw\cap Q_{\varrho}$ and adding the result to \eqref{eq:gradv}, we infer from \eqref{eq:almostopt} that
	\begin{align}\label{eq:errorseparated}
	F_{\e}(\w)(\tilde{u}_{\e},\tilde{v}_{\e},Q_{\varrho})\leq& F_{\e}(\w)(u_{\e},v_{\e},Q_{\varrho})+C\eta\varrho^{d-1}+C\eta\e^{d-1}\#\left(\e\Lw\cap Q_{\varrho}\setminus L_{v_{\e}}(\eta)\right)\nonumber
	\\
	&+C\e^{d-1} \#\{(x,y)\in\mathcal{R}_{\e}(t_{\e}):\,\e x\in L_{v_{\e}}(\eta)\}
	\end{align}
	We claim that the last three terms can be made small relatively to $\varrho^{d-1}$ by choosing the order of limits as in the statement. On the one hand, note that since $\eta\in(0,1/2)$ we have by  \eqref{eq:almostopt}
	\begin{equation}\label{eq:sublevelbound}
	\e^{d-1}\#\left(\e\Lw\cap Q_{\varrho}\setminus L_{v_{\e}}(\eta)\right)\leq C\sum_{\e x\in\e\Lw\cap Q_{\varrho}}\e^{d-1}(v_{\e}(\e x)-1)^2\leq C\varrho^{d-1}.
	\end{equation}
	On the other hand, since $v_{\e}\geq 0$, from \eqref{eq:goodchoice} we deduce
	\begin{equation}\label{eq:R_ebound}
	\e^{d-1} \#\{(x,y)\in\mathcal{R}_{\e}(t_{\e}):\,\e x\in L_{v_{\e}}(\eta)\}\leq \frac{1}{\eta}\sum_{(x,y)\in\mathcal{R}_{\e}(t_{\e})}\e^{d-1}|v_{\e}(\e x)|\leq C\frac{1}{\eta a}\varrho^{d-\frac{1}{2}}.
	\end{equation}
	Inserting \eqref{eq:sublevelbound} and \eqref{eq:R_ebound} in \eqref{eq:errorseparated} we obtain the estimate
	\begin{equation*}
	\varrho^{1-d}F_{\e}(\w)(\tilde{u}_{\e},\tilde{v}_{\e},Q_{\varrho})\leq \varrho^{1-d}F_{\e}(\w)(u_{\e},v_{\e},Q_{\varrho})+C\eta+C\frac{1}{\eta a}\varrho^{\frac{1}{2}}.
	\end{equation*}
	Taking the appropriate infimum on each side, then letting first $\e\to 0$, then $\delta\to 0$ and $\varrho\to 0$, we conclude by the arbitrariness of $\eta>0$.
\end{proof} 
Gathering Proposition \ref{limitsbvp}, Proposition \ref{p.gradientpartsequal}, Lemma \ref{approxminprob} and Lemma \ref{separationofscales1} we can now prove Theorem \ref{mainrep}.
\begin{proof}[Proof of Theorem \ref{mainrep}]
Let $\Lw$ be an admissible lattice with admissible edges and let $\e_n$ and $F(\w)$ be the subsequence and the functional provided by Proposition \ref{limitsbvp}. Thanks to Proposition \ref{p.gradientpartsequal} we know that along the subsequence $\e_n$ also the functionals $F_{\e_n}^b(\w)(\cdot,1,A)$ $\Gamma$-converge to $F^b(\w)(\cdot,A)$ for every $A\in\Areg(D)$ with $F^b(\w)$ given by Theorem \ref{ACGmain}. Combining Propositions \ref{limitsbvp} and \ref{p.gradientpartsequal} we then deduce that for every $A\in\Areg(D)$ and every $u\in SBV^2(A)$ we have
\begin{align*}
F(\w)(u,1,A)=\int_A f(\w,x,\nabla u)\dx+\int_{S_u\cap A}\varphi(\w,x,u^+-u^-,\nu_u)\dHd,
\end{align*}
where $f(\w,\cdot,\cdot)$ is given by Theorem \ref{ACGmain} and $\varphi(\w,\cdot,\cdot,\cdot)$ is determined by the derivation formula \eqref{derivationformula}. Moreover, Lemma \ref{approxminprob} together with Lemma \ref{separationofscales1} ensure that the surface integrand $\varphi$ does not depend on the jump opening $u^+-u^-$ (see also Remark \ref{r.jumpindependence}). In fact, for every $A\in\Areg(D)$ and every $u\in SBV^2(A)$ we obtain
\begin{align}\label{gammalimsbv}
F(\w)(u,1,A)=\int_A f(\w,x,\nabla u)\dx+\int_{S_u\cap A}\varphi(\w,x,\nu_u)\dHd,
\end{align}
where $\varphi(\w,\cdot,\cdot):D\times S^{d-1}\to[0,+\infty)$ is given by the asymptotic formula \eqref{eq:formula_phi}. Finally, using a standard truncation argument (see, \eg the proof of \cite[Theorem 3.3]{R18} for more details), thanks to Remark \ref{trunc} we deduce that formula \eqref{gammalimsbv} extends to the whole $GSBV^2(A)$.
\end{proof}
\subsection{Optimality of the lattice-space scaling}
We close this section by proving Theorem \ref{t.representationl} and the optimality of the lattice-space scaling.
\begin{proof}[Proof of Theorem \ref{t.representationl}]
Let $\Lw$ be an admissible lattice with admissible edges $\mathcal{E}(\w)$ and for every $\e>0$ let $F_{\e,\kappa_\e}(\w)$ be as in \eqref{def:Fdelta} with $\kappa_\e=\ell\e$ for some $\ell\in(0,+\infty)$. It is convenient to rewrite the energy as
\begin{align*}
F_{\e,\kappa_\e}(\w)(u,v)=F_{\kappa_\e}^b(\w)(u,v)+F_{\e,\ell}^s(\w)(v),
\end{align*}
where
\begin{align*}
F_{\e,\ell}^s(\w)(v):=\frac{\beta}{2}\Big(\ell\sum_{\kappa_\e x\in\kappa_\e\Lw\cap D}\hspace*{-1em}\kappa_\e^{d-1}(v(\kappa_\e x))^2+\frac{1}{\ell}\sum_{\substack{(x,y)\in\mathcal{E}(\w)\\\kappa_\e x,\kappa_\e y\in D}}\hspace*{-0.5em}\kappa_\e^{d-1}\left|v(\kappa_\e x)-v(\kappa_\e y)\right|^2\Big).
\end{align*}
It is then easy to see that Lemmata \ref{local}--\ref{almostmeasure} are satisfied also for the functionals $F_{\e,\kappa_\e}$ with the constant $c$ in Lemma \ref{compact} and Lemma \ref{bounds} depending on $\ell$. As a consequence, Proposition \ref{limitsbvp} holds for $F_{\e,\kappa_\e}$ and yields a limit functional $F_{\ell}(\w)$. Moreover, Proposition \ref{p.gradientpartsequal} remains unchanged if $F_{\e}$ is replaced by $F_{\e,\kappa_\e}$. Finally, Lemma \ref{approxminprob} and Lemma \ref{separationofscales1} are still valid for $m_{\ell}^{\w},m_{\e,\ell,\delta}^\w$ and $\varphi_{\e,\ell,\delta}^\w$, where for every $\delta>0$,$m_{\ell}^{\w}$ and $m_{\e,\ell,\delta}^\w$ are as in \eqref{def:med} with $F_{\ell}(\w)$ instead of $F(\w)$ and $F_{\e,\kappa_\e}$ instead of $F_{\e}$, and $\varphi_{\e,\ell,\delta}^\w$ is as in \eqref{def:surf:int} with $F_\e^b(\w)$ and $F_{\e,s}(\w)$ replaced by  $F_{\kappa_\e}^b(\w)$ and $F_{\e,\ell}^s(\w)$, respectively. Moreover, $\mathcal{S}_{\e,\delta}^\w$ and $\mathcal{PC}_{\e,\delta}^\w$  are replaced by $\mathcal{S}_{\kappa_\e,\delta}^\w$ and $\mathcal{PC}_{\kappa_\e,\delta}^\w$. Thus, arguing as in the proof of Theorem \ref{mainrep} we obtain the required integral representation of $F_{\ell}(\w)$ on $GSBV^2(D)$, where now the surface integrand $\varphi_\ell(\w,\cdot,\cdot)$ can be equivalently characterized by the formulas
\begin{align*}
\varphi_\ell(\w,x_0,\nu) &=\limsup_{\varrho\to 0}\varrho^{1-d}\lim_{\delta\to 0}\limsup_{n\to+\infty}m_{\e_n,\ell,\delta}^w(u_{x_0,\nu}^{1,0},v_{x_0,\nu}^{\kappa_{\e_n}}, Q_\nu(x_0,\varrho))\\
&=\limsup_{\varrho\to 0}\varrho^{1-d}\lim_{\delta\to 0}\limsup_{n\to+\infty}\varphi_{\e_n,\ell,\delta}^\w(u_{x_0,\nu}^{1,0},Q_\nu(x_0,\varrho)).
\end{align*}
Notice that thanks to the separation of scales only the surface integrand $\varphi_\ell(\w,\cdot,\cdot)$ may depend on the ratio $\ell$, while the volume integrand $f(\w,\cdot,\cdot)$ is independent of $\ell$.

In order to verify the estimate in \eqref{est:phil1} we use again the connection to weak-membrane energies. To this end let $\e_n$ be a subsequence such that $F_{\e_n,\kappa_{\e_n}}(\w)$ $\Gamma$-converges to $F_\ell(\w)$ and set $\kappa_n:=\kappa_{\e_n}=\ell\e_n$. Upon passing to a further subsequence we can assume that also $G_{\kappa_n,\alpha}(\w)$ $\Gamma$-converges for every $\alpha>0$. Let $(x_0,\nu)\in\R^d\times S^{d-1}$ and for $\delta>0$, $\varrho>0$ arbitrary let $(u,v)$ be admissible for $m_{\e_n,\ell,\delta}^\w(u_{x_0,\nu}^{1,0},v_{x_0,\nu}^{\kappa_n},Q_\nu(x_0,\varrho))$. Clearly, $u$ is admissible for the minimization problem
\begin{align*}
\inf\{G_{\kappa_n,\beta\ell}(\w)(u,Q_{\nu}(x_0,\varrho)):\,u\in\mathcal{S}_{\kappa_n,\delta}^{\w}(u_{x_0,\nu}^{1,0},Q_{\nu}(x_0,\varrho))\}.
\end{align*}
Moreover, due to Proposition \ref{prop:ATandWeak} we have
\begin{align*}
F_{\e_n,\kappa_n}(\w)(u,v,Q_\nu(x_0,\varrho))\geq G_{\kappa_n,\beta\ell}(u,Q_\nu(x_0,\varrho)).
\end{align*}
Hence, passing to the infimum and taking the appropriate limits in $n$, $\delta$ and $\varrho$, thanks to Theorem \ref{t.weakmembrane} we deduce that
\begin{align*}
\varphi_\ell(\w,x_0,\nu)\geq s_{\beta\ell}(\w,x_0,\nu)=\beta\ell s_1(\w,x_0,\nu).
\end{align*}
We continue proving the upper estimate in \eqref{est:phil1}. For $\delta>0$, $\varrho>0$ fixed we choose $w:\kappa_n\Lw\to\{\pm 1\}$ admissible for the minimization problem 
\begin{align*}
\inf\{I_{\kappa_n,\beta(\ell+\tfrac{M}{\ell})}(\w)(w,Q_{\nu}(x_0,\varrho)):\,w\in\mathcal{S}_{\kappa_n,\delta}^{\w}(u_{x_0,\nu}^{1,-1},Q_{\nu}(x_0,\varrho))\},
\end{align*}
and we observe that the $u$-component of the pair $(u,v)\in\mathcal{PC}_{\kappa_n}^\w\times\mathcal{PC}_{\kappa_n}^\w$ defined as
\begin{align*}
u(\kappa_n x) &:=
\begin{cases}
1 &\text{if $w(\kappa_n x)=1$,}\\
0 &\text{if $w(\kappa_n x)=-1$,}
\end{cases}
\\
v(\kappa_n x) &:=
\begin{cases}
0 &\text{if $\max\{|w(\kappa_n x)-w(\kappa_n y)|\colon \kappa_n y\in\kappa_n\mathcal{E}(\w)(x)\cap Q_\nu(x_0,\varrho)\}=2$,}\\
1 &\text{otherwise}
\end{cases}
\end{align*}
belongs to $\mathcal{S}_{\kappa_n,\delta}^{\w}(u_{x_0,\nu}^{1,0},Q_{\nu}(x_0,\varrho))$. Moreover, arguing as in the proof of Lemma \ref{bounds} we obtain
\begin{align*}
F_{\kappa_n,\ell}^s(\w)(v,Q_\nu(x_0,\varrho)) &\leq\frac{\beta}{2}\left(\ell+\frac{M}{\ell}\right)\sum_{\kappa_n x\in\kappa_n\Lw\cap Q_\nu(x_0,\varrho)}\hspace*{-2em}\kappa_n^{d-1}(v(\kappa_n x)-1)^2
=I_{\kappa_n,\beta(\ell+\tfrac{M}{\ell})}(\w)(w,Q_\nu(x_0,\varrho)).
\end{align*}
However, in general $v$ is not admissible for $\varphi_{\e_n,\ell,\delta}^\w(u_{x_0,\nu}^{1,0},Q_\nu(x_0,\varrho))$ due to the boundary conditions. Nevertheless, $F_{\kappa_n}^b(\w)(u,v)=0$, hence using only the boundary conditions of $u$ we can argue as in the first part of the proof of Lemma \ref{approxminprob} to show that
\begin{equation*}
m_{\ell}^{\w}(u_{x_0,\nu}^{1,0},Q_{\nu}(x_0,\varrho))\leq \liminf_{n\to +\infty}F_{\kappa_n,\ell}^s(\w)(v,Q_\nu(x_0,\varrho)).
\end{equation*}
Since $w$ was arbitrarily chosen, passing to the infimum and taking again the appropriate limits in $n$, $\delta$, $\varrho$ finally yields
\begin{align*}
\varphi_\ell(\w,x_0,\nu)\leq s_{\beta(\ell+\tfrac{M}{\ell})}(w,x_0,\nu)=\beta\left(\ell+\frac{M}{\ell}\right)s_1(\w,x_0,\nu).
\end{align*}
\end{proof}
\noindent Eventually we prove Corollary \ref{c.optimal}.
\begin{proof}[Proof of Corollary \ref{c.optimal}]
Let $\Lw$ be an admissible lattice with admissible edges $\mathcal{E}(\w)$ and suppose now that $\kappa_\e$ is such that $\kappa_\e/\e\to+\infty$ as $\e\to 0$. Note that by Proposition \ref{prop:ATandWeak}, for every $\ell>0$ there exists $\e_\ell>0$ such that for every $\e\in (0,\e_\ell)$ we have
\begin{equation*}
C\geq  F_{\kappa_{\e}}^b(\w)(u_{\e},v_{\e})+\frac{\ell}{2}\sum_{\kappa_{\e}x\in\kappa_{\e}\Lw\cap D}\kappa_{\e}^{d-1}(v_{\e}(\kappa_{\e}x)-1)^2 \geq G_{\kappa_\e,\ell}(\w)(u_\e).
\end{equation*}
Since $u_{\e}$ is bounded in $L^{2}(D)$, the compactness result for weak-membrane energies (cf. \cite[Lemma 5.6]{R18}) yields that up to a subsequence, $u_{\e}\to u$ in $L^1(D)$ for some $u\in GSBV^2(D)\cap L^2(D)$. It remains to show that $u\in W^{1,2}(D)$. To do so, we prove that the sequence $(T_ku)$ is bounded in $W^{1,2}(D)$ uniformly with respect to $k$, then we may conclude by letting $k\to+\infty$. Thanks to Theorem \ref{t.weakmembrane}, up to passing to a further subsequence (not relabeled), we can assume that $G_{\kappa_\e,\ell}(\w)$ $\Gamma$-converges to $G_{\ell}(\w)$. Thus, the growth conditions for the integrands in Theorem \ref{t.weakmembrane} imply that
\begin{equation*}
C\geq \int_D q(\w,x,\nabla u)\,\mathrm{d}x+\int_{S_{u}} s_{\ell}(\w,x,\nu_{u})\,\mathrm{d}\mathcal{H}^{d-1}\geq\frac{1}{C}\int_D|\nabla u|^2\dx+\frac{\ell}{C}\Hd(S_{u}),
\end{equation*}
for every $\ell>0$, so that $\mathcal{H}^{d-1}(S_{u})=0$. In particular, for every $k>0$ we have $\Hd(S_{T_ku})=0$ and $\sup_{k}\|\nabla T_ku\|_{L^2}\leq\|\nabla u\|_{L^2}\leq C$. Since $T_ku\in SBV(D)\cap L^{\infty}(D)$ and $u\in L^2(D)$ this implies that $(T_ku)$ is bounded in $W^{1,2}(D)$ uniformly with respect to $k$ and we conclude.
\end{proof}
\section{Stochastic homogenization: Proof of Theorems \ref{mainthm1} and \ref{MSapprox}}\label{s.stochhom}
In this section we prove Theorem \ref{mainthm1}. In particular we establish the existence of the limit defining $\varphi_{\rm hom}$ in \eqref{ex:hom:surf}. Similar arguments have already been used by the second and third author in \cite[Theorem 5.5]{ACR}, \cite[Theorem 5.8]{BCR} (see also \cite[Sections 5 and 6]{CDMSZ17a}). The main step consists in defining a suitable subadditive stochastic process (see Definition \ref{d.subadprocess} below), which then allows us to apply the subadditive ergodic theorem which we recall in Theorem \ref{t.subadergodic} below. To this end, we first need to introduce some notation. 

For every $a,b\in\Z^{d-1}$ with $a_i<b_i$ for $i=1,\ldots,d-1$ we define the $(d-1)$-dimensional interval $[a,b):=\{x\in\R^{d-1}\colon a_i\leq x_i <b_i\text{ for }i=1\ldots,d-1\}$ and we set $\mathcal{I}:=\{[a,b)\colon a,b\in\Z^{d-1},\, a_i<b_i\text{ for }i=1,\ldots,d-1\}$.
\begin{definition}\label{d.subadprocess}
A discrete subadditive stochastic process is a function $\mu:\mathcal{I}\to L^1(\Omega)$ satisfying the following properties:
\begin{enumerate}[label=(\roman*)]
\item (subadditivity) for every $I\in\mathcal{I}$ and every finite partition $(I_k)_{k\in K}\subset\mathcal{I}$ of $I$ a.s. we have \begin{equation*}
\mu(I,\w)\leq\sum_{k\in K}\mu(I_k,\w);
\end{equation*}
\item (boundedness from below) there holds 
$$
\inf\left\{\frac{1}{|I|}\int_\Omega \mu(I,\w)\dP(\w):\,I\in\mathcal{I}\right\}>-\infty.
$$
\end{enumerate}
\end{definition}
We make use of the following pointwise subadditive ergodic theorem (see \cite[Theorem 2.4]{AkKr}).
\begin{theorem}\label{t.subadergodic}
Let $\mu:\mathcal{I}\to L^1(\Omega)$ be a discrete subadditive stochastic process and let $I_k:=[-k,k)^{d-1}$. Suppose that there exists a measure preserving group action $\{\tau_z\}_{z\in\Z^{d-1}}$ such that $\mu$ is stationary with respect to $\{\tau_z\}_{z\in\Z^{d-1}}$,\ie
\begin{equation*}
\forall\, I\in\mathcal{I},\ \forall\, z\in\Z^{d-1}:\ \mu(I+z,\w)=\mu(I,\tau_z\w)\ \text{a.s.}
\end{equation*}
Then there exists a function $\Phi:\Omega\to\R$ such that, for $\mathbb{P}$-a.e. $\w$,
\begin{equation*}
\lim_{k\to+\infty}\frac{\mu(I_k,\w)}{\Hd(I_k)}=\Phi(\w).
\end{equation*}
\end{theorem}
As a first step towards the proof of Theorem \ref{mainthm1} we prove the following proposition.
\begin{proposition}\label{prop:ex:limit0}
Let $\mathcal{L}$ be an admissible stochastic lattice that is stationary with respect to a measure-preserving additive group action $\{\tau_z\}_{z\in\Z^d}$ with admissible, stationary edges in the sense of Definitions \ref{defadmissible} \& \ref{defgoodedges}. Then there exist $\widetilde{\Omega}\subset\Omega$ with $\mathbb{P}(\widetilde{\Omega})=1$ and a function $\varphi_{\rm hom}:\Omega\times S^{d-1}\to[0,+\infty)$ satisfying
\begin{align*}
\varphi_{\rm hom}(\w,\nu) &=\lim_{t\to+\infty}t^{1-d}\varphi_{1,M}^\w(u_{0,\nu}^{1,0},Q_{\nu}(0,t)),
\end{align*}
for every $\w\in\widetilde{\Omega}$ and every $\nu\in S^{d-1}$. Moreover, we have $\tau_z(\widetilde{\Omega})=\widetilde{\Omega}$ for every $z\in\Z^d$ and
\begin{align}\label{eq:shiftinv}
\varphi_{\rm hom}(\tau_z\w,\nu)=\varphi_{\rm hom}(\w,\nu)
\end{align}
for every $z\in\Z^d$, $\w\in\widetilde{\Omega}$, and $\nu\in S^{d-1}$.
\end{proposition}
In order to prove Proposition \ref{prop:ex:limit0} above we will use several times the following lemma.
\begin{lemma}\label{l.extension}
Let $z,z'\in \R^d$, $t,t'>0$ and $\nu\in S^{d-1}$ be such that the cubes $Q_\nu(z,t)$ and $Q_\nu(z',t')$ satisfy the following conditions
\begin{equation*}
{\rm (i)}\ Q_\nu(z,t)\subset Q_\nu(z',t'),\qquad {\rm (ii)}\ \dist(\partial Q_\nu(z,t),\partial Q_\nu(z',t'))>2M,\qquad {\rm (iii)}\ \dist(z',H_\nu(z))\leq\frac{t}{4}.
\end{equation*}
Then there exists a constant $c>0$ such that
\begin{equation*}
\varphi_{1,M}^\w(u_{z',\nu}^{1,0},Q_\nu(z',t'))\leq \varphi_{1,M}^\w(u_{z,\nu}^{1,0},Q_\nu(z,t))+c(|z-z'|+|t-t'|)(t')^{d-2}.
\end{equation*}
\end{lemma}
\begin{proof}
To shorten notation let us set $Q=Q_\nu(z,t)$ and $Q'=Q_\nu(z',t')$. Let us choose a pair $(u,v)\in\mathcal{S}_{1,M}^\w(u_{z,\nu}^{1,0},Q)\times\mathcal{PC}_{1,M}^\w(v_{z,\nu}^1,Q)$ satisfying $F_{1}^{b}(\w)(u,v,Q)=0$ and $F_{1}^{s}(v,Q)=\varphi_{1,M}^\w(u_{z,\nu}^{1,0},Q)$. Thanks to (ii) we can extend $u$ to a function $\tilde{u}\in\mathcal{S}_{1,M}^\w(u_{z',\nu}^{1,0},Q')$ by setting $\tilde{u}(x):=u_{z',\nu}^{1,0}(x)$ on $\Lw\cap Q'\setminus Q$. We now construct a function $\tilde{v}\in\mathcal{PC}_{1,M}^\w(v_{z',\nu}^1,Q')$ satisfying $F_{1}^{b}(\w)(\tilde{u},\tilde{v},Q')=0$. To this end we introduce some notation. 
We denote by
\begin{align*}
S_\nu(z,z'):=\{x\in\R^d\colon \min\{\langle z,\nu\rangle,\langle z',\nu\rangle\}\leq\langle x,\nu\rangle\leq\max\{\langle z,\nu\rangle\langle z',\nu\rangle\}\}
\end{align*}
the stripe enclosed by the two hyperplanes $H_\nu(z)$ and $H_\nu(z')$.
Moreover, the sets
\begin{align*}
 L_\nu(z):=\{x\in\R^d\colon|\langle x-z,\nu\rangle|\leq M\},\qquad L_\nu(z'):=\{x\in\R^d\colon|\langle x-z',\nu\rangle|\leq M\}
\end{align*}
are the layers of thickness $2M$ around $H_\nu(z)$ and $H_\nu(z')$. Finally, we set 
\begin{align*}
U_\nu(z,z'):=S_\nu(z,z')\cup L_\nu(z)\cup L_\nu(z').
\end{align*}
Notice that for any pair $(x,y)\in\mathcal{E}(\w)$ with at least one point not contained in $Q$ and $\tilde{u}(x)\neq\tilde{u}(y)$ one of the following conditions is satisfied:
\begin{enumerate}[label=(\alph*)]
\setlength{\itemsep}{3pt}
\item if $x\in Q$ and $y\in Q'\setminus Q$, since $|x-y|\leq M$ we have $x,y\in U_\nu(z,z')\cap\partial_M Q$;
\item if $x,y\in Q'\setminus Q$ then $\tilde{u}(x)\neq \tilde{u}(y)$ implies that $x$ and $y$ lie on two different sides of the hyperplane $H_\nu(z')$, hence $x,y\in L_\nu(z')$.
\end{enumerate}
This motivates to define $\tilde{v}$ on $\Lw$ by setting
\begin{align*}
\tilde{v}(x):=
\begin{cases}
v(x) &\text{if $x\in Q\setminus\left(U_\nu(z,z')\cap \partial_M Q\right)$},\\
0 &\text{if $x\in\left(U_\nu(z,z')\cap \partial_M Q\right)\cup\left(L_\nu(z')\setminus Q\right)$},\\
1 &\text{otherwise,}
\end{cases}
\end{align*}
(see Figure \ref{fig:constructionvtilde}).
\begin{figure}[h]
\centering
\def\svgwidth{0.35\columnwidth}
\begingroup%
  \makeatletter%
  \providecommand\color[2][]{%
    \errmessage{(Inkscape) Color is used for the text in Inkscape, but the package 'color.sty' is not loaded}%
    \renewcommand\color[2][]{}%
  }%
  \providecommand\transparent[1]{%
    \errmessage{(Inkscape) Transparency is used (non-zero) for the text in Inkscape, but the package 'transparent.sty' is not loaded}%
    \renewcommand\transparent[1]{}%
  }%
  \providecommand\rotatebox[2]{#2}%
  \newcommand*\fsize{\dimexpr\f@size pt\relax}%
  \newcommand*\lineheight[1]{\fontsize{\fsize}{#1\fsize}\selectfont}%
  \ifx\svgwidth\undefined%
    \setlength{\unitlength}{490.74755571bp}%
    \ifx\svgscale\undefined%
      \relax%
    \else%
      \setlength{\unitlength}{\unitlength * \real{\svgscale}}%
    \fi%
  \else%
    \setlength{\unitlength}{\svgwidth}%
  \fi%
  \global\let\svgwidth\undefined%
  \global\let\svgscale\undefined%
  \makeatother%
  \begin{picture}(1,0.93211685)%
    \lineheight{1}%
    \setlength\tabcolsep{0pt}%
    \put(0,0){\includegraphics[width=\unitlength,page=1]{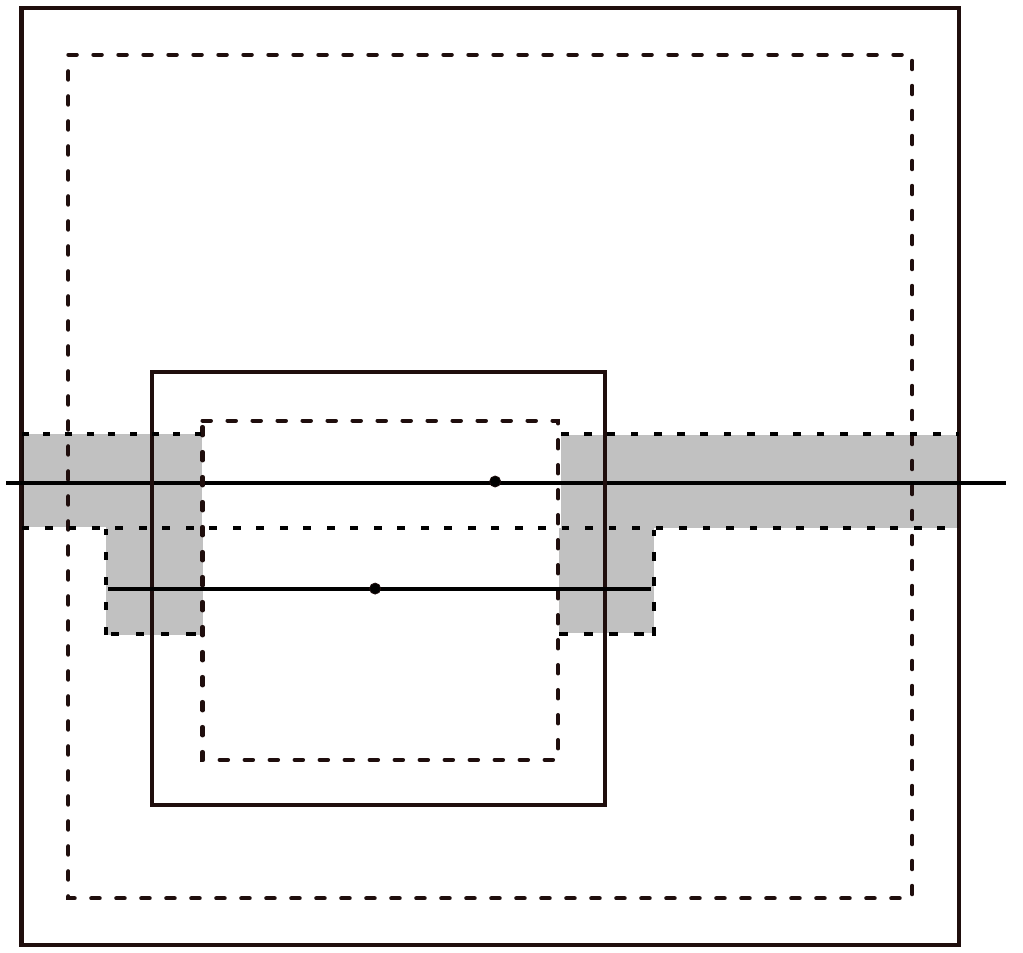}}%
    \put(0.46597716,0.42844094){\color[rgb]{0,0,0}\makebox(0,0)[lt]{\lineheight{0}\smash{\begin{tabular}[t]{l}$z'$\end{tabular}}}}%
    \put(0.34640979,0.31843896){\color[rgb]{0,0,0}\makebox(0,0)[lt]{\lineheight{0}\smash{\begin{tabular}[t]{l}$z$\end{tabular}}}}%
    \put(0,0){\includegraphics[width=\unitlength,page=2]{disegno.pdf}}%
    \put(0.54734741,0.26220818){\color[rgb]{0,0,0}\makebox(0,0)[lt]{\lineheight{0}\smash{\begin{tabular}[t]{l}{\tiny $M$}\end{tabular}}}}%
    \put(0.59469367,0.26220818){\color[rgb]{0,0,0}\makebox(0,0)[lt]{\lineheight{0}\smash{\begin{tabular}[t]{l}{\tiny $M$}\end{tabular}}}}%
    \put(0.66729128,0.32533652){\color[rgb]{0,0,0}\makebox(0,0)[lt]{\lineheight{0}\smash{\begin{tabular}[t]{l}{\tiny $M$}\end{tabular}}}}%
    \put(0,0){\includegraphics[width=\unitlength,page=3]{disegno.pdf}}%
    \put(0.99192615,0.48259272){\color[rgb]{0,0,0}\makebox(0,0)[lt]{\lineheight{0}\smash{\begin{tabular}[t]{l}$\nu$\end{tabular}}}}%
  \end{picture}%
\endgroup%

\caption{{\footnotesize The two cubes $Q_\nu(z,t)$ and $Q_\nu(z',t')$ and in gray the set $\left(U_\nu(z,z')\cap \partial_M Q\right)\cup\left(L_\nu(z')\setminus Q\right)$.}}
\label{fig:constructionvtilde}
\end{figure} 
Observe that thanks to (ii) we have $\tilde{v}\in\mathcal{PC}_{1,M}^\w(v_{z',\nu}^1,Q')$. Moreover, by construction $F_{1}^{b}(\w)(\tilde{u},\tilde{v},Q')=0$, thus $\tilde{v}$ is admissible for $\varphi_{1,M}^\w(u_{z',\nu}^{1,0},Q')$ and it remains to show that
\begin{align}\label{est:ext-final}
F_{1}^{s}(\w)(\tilde{v},Q')\leq F_{1}^{s}(\w)(v,Q)+c(|z-z'|+|t-t'|)(t')^{d-2},
\end{align}
then the result follows from the choice of the test pair $(u,v)$.

In order to prove \eqref{est:ext-final} we first notice that for any $x\in \mathcal{L}(\w)\cap Q$ by definition we have $\tilde{v}(x)\neq v(x)$ only if $x\in U_{\nu}(z,z')\cap \partial_M Q$. Similarly, for $(x,y)\in\mathcal{E}(\w)$ with $x,y\in Q$ we have $|\tilde{v}(x)-\tilde{v}(y)|\neq|v(x)-v(y)|$ only if at least one point belongs to $\left(U_\nu(z,z')\right)\cap \partial_M Q$. Thus, thanks to \eqref{neighbours} we immediately deduce
\begin{align}\label{est:ext-Q}
F_{1}^{s}(\w)(\tilde{v},Q)\leq F_{1}^{s}(\w)(v,Q)+C\#\left(\Lw\cap U_\nu(z,z')\cap \partial_M Q\right).
\end{align}
The remaining contributions can be estimated in the same way. In fact, for any $x\in\mathcal{L}(\w)\cap(Q'\setminus Q)$ we have $(\tilde{v}(x)-1)\neq 0$ only if $x\in U_\nu(z,z')\cap \partial_M Q$ or $x\in L_\nu(z')\setminus Q$. Finally, any pair $(x,y)\in\mathcal{E}(\w)$ with at least one point belonging to $Q'\setminus Q$ only gives a contribution if at least one point belongs to $U_\nu(z,z')\cap \partial_M Q$ or to $L_\nu(z')\setminus Q$. In combination with \eqref{est:ext-Q} this yields
\begin{align}\label{est:ext-Qprime}
F_{1}^{s}(\w)(\tilde{v},Q') &\leq F_{1}^{s}(\w)(v,Q)+C\#\left(\Lw\cap U_\nu(z,z')\cap \partial_M Q\right)\nonumber\\
&\hphantom{\leq F_{1}^{s}(\w)(v,Q)\;}+C\#\left(\Lw\cap L_\nu(z')\cap Q'\setminus Q \right)\nonumber\\
&\leq F_{1}^{s}(\w)(v,Q)+C\Hd\left(U_\nu(z,z')\cap\partial Q\right)+C\Hd\left(H_\nu(z')\cap Q'\setminus Q\right),
\end{align}
where to obtain the second inequality we have used Remark \ref{voronoi} and (iii). Eventually, since
\begin{align*}
 &\Hd\left(U_\nu(z,z')\cap\partial Q\right)\leq c|z-z'|t^{d-2},\\
 &\Hd\left(H_\nu(z')\cap Q'\setminus Q\right)\leq c(|z-z'|+|t-t'|)(t')^{d-2},
\end{align*}
we obtain \eqref{est:ext-final} from \eqref{est:ext-Qprime} upon noticing that by hypotheses $t<t'$.
\end{proof}
Having at hand Lemma \ref{l.extension} we now prove Proposition \ref{prop:ex:limit0}.
\begin{proof}[Proof of Proposition \ref{prop:ex:limit0}]
For definiteness we specify the orientation of the cube $Q_{\nu}$. Given $\nu\in S^{d-1}$, we choose the orthonormal basis as the columns of the orthogonal matrix $O_{\nu}$ induced by the linear mapping
\begin{equation*}
x\mapsto \begin{cases}
\displaystyle{2\frac{\langle x,\nu+e_d\rangle}{\|\nu+e_d\|^2}(\nu+e_d)-x} &\mbox{if $\nu\in S^{d-1}\setminus\{-e_d\}$,}
	\\
-x &\mbox{otherwise.}
\end{cases}
\end{equation*}The proof is divided into several steps. 

\medskip
\textbf{Step 1} Existence of $\varphi_{\rm hom}(\w,\nu)$ for rational directions $\nu\in S^{d-1}\cap\mathbb{Q}^d$.
\\
Let $\nu\in S^{d-1}\cap\mathbb{Q}^d$; then $O_\nu\in\mathbb{Q}^{d\times d}$ is such that $O_\nu e_d=\nu$ and the set $\{O_\nu e_j\colon j=1,\ldots d-1\}$ is an orthonormal basis for $H_{\nu}$. Moreover, there exists an integer $m=m(\nu)>4M$ such that $mO_\nu(z,0)\in\Z^d$ for every $z\in\Z^{d-1}$. We show that there exists a set $\Omega^\nu\subset\Omega$ of probability one such that the limit defining $\varphi_{\rm hom}(\w,\nu)$ exists for all $\w\in\Omega^\nu$. To this end, we define a suitable discrete stochastic process (depending on $\nu$) that satisfies all the conditions of Theorem \ref{t.subadergodic}. We start with some notation. For every $I=[a_1,b_1)\times\cdots\times[a_{d-1},b_{d-1})\in\mathcal{I}$ we define the set $I_d\subset\R^d$ as
\begin{align*}
I_d:=mO_\nu({\rm int}\, I\times (-s_{\rm max},s_{\rm max})),\quad\text{where}\ s_{\rm max}:=\max_{i=1,\ldots d-1}\frac{|b_i-a_i|}{2},
\end{align*}
and we define a stochastic process $\mu:\mathcal{I}\to L^1(\Omega)$ by setting
\begin{align*}
\mu(I,\w):=&\inf\Big\{F_{1}^{s}(\w)(v,I_d)\colon v\in\mathcal{PC}_{1,M}^\w(v_{0,\nu}^1,I_d):\,
\exists\, u\in \mathcal{S}_{1,M}^\w(u_{0,\nu}^{1,0},I_d),\ F_{1}^{b}(\w)(u,v,I_d)=0\Big\}
\\
&+C_\mu\Hdtwo(\partial I),
\end{align*}
where $C_\mu>0$ is a constant to be chosen later. Note that here we have chosen the same width for the boundary condition imposed on $u$ and $v$. Let us prove that $\mu(I,\cdot)\in L^1(\Omega)$. Using the measurability of $\mathcal{L}$ and $\mathcal E$ (cf. Definition \ref{defgoodedges}), one can show that for fixed $u,v\in\mathcal{PC}_{\e}^{\w}$  (interpreted as deterministic vectors $(u,v)\in\R^{\mathbb{N}}\times [0,1]^{\mathbb{N}}$) and $\lambda>0$ the function 
\begin{align*}
\mu_{\lambda,u,v}(I,\w)=&F_{1}^{s}(\w)(v, I_d)+C_{\mu}\mathcal{H}^{d-2}(\partial I)+\lambda F_{1}^{b}(\w)(u,v, I_d)
\\
&+\lambda\sum_{\substack{x\in\Lw\cap I_d\\ \dist(x,\partial I_d)\leq M}}\left(|u(x)-u_{0,\nu}^{0,1}(x)|^2+|v(x)-v_{0,\nu}^1(x)|^2\right)+\lambda\sum_{x\in\Lw\cap I_d}\dist^2(u(x),\{0,1\})
\end{align*}
is $\mathcal{F}$-measurable. Minimizing over the first $k$ components of the vectors $u$ and $v$ (while fixing the others to zero) preserves measurability and when $k\to +\infty$ we infer that $\w\mapsto \inf_{u,v}\mu_{\lambda,u,v}(I,\w)$ is measurable. Sending then $\lambda\to +\infty$ we finally conclude that also $\w\mapsto\mu(I,\w)$ is measurable as the pointwise limit of measurable functions. In order to show integrability, note that the function $v_{0,\nu}^1$ is admissible in the minimization problem defining $\mu(I,\w)$ (see also \eqref{eq:implication}) and, similar to the counting argument used to derive \eqref{eq:counting}, one can show that
\begin{equation}\label{eq:l_inftybound}
\mu(I,\w)\leq F_{1}^{s}(\w)(v_{0,\nu}^1,I_d)+C_{\mu}(\partial I)\leq C \mathcal{H}^{d-1}(I_d\cap H_{\nu})+C_{\mu}(\partial I)
\end{equation}
uniformly in $\w$, so that $\mu(I,\cdot)\in L^{\infty}(\Omega)$.

We next prove the stationarity of the process. To this end, for every $z\in\Z^{d-1}$ we set $z_m^\nu:=mO_\nu(z,0)$ and we define a measure-preserving group action $\{\tilde{\tau}_z\}_{z\in\Z^{d-1}}$ by setting $\tilde{\tau}_z:=\tau_{-z_m^\nu}$, where $\{\tau_z\}_{z\in\Z^d}$ is as in the statement. Note that for every $I\in\I$ and every $z\in\Z^{d-1}$ we have $(I-z)_d=I_d-z_{m}^\nu$. Moreover, since $z_{m}^{\nu}\in H_{\nu}\cap\Z^d$ and $\mathcal{L}$ is stationary with respect to $\{\tau_z\}_{z\in\Z^d}$, we have
\begin{equation*}
\begin{split}
v\in\mathcal{PC}_{1,M}^\w(v_{0,\nu}^1,(I-z)_d)&\iff v_z(\cdot)=v(\cdot-z^{\nu}_m)\in\mathcal{PC}_{1,M}^{\tilde{\tau}_z\w}(v_{0,\nu}^1,I_d)
\\
u\in\mathcal{S}_{1,M}^\w(u_{0,\nu}^{1,0},(I-z)_d)&\iff u_z(\cdot)=u(\cdot-z^{\nu}_m))\in\mathcal{S}_{1,M}^{\tilde{\tau}_z\w}(u_{0,\nu}^{1,0},I_d)
\end{split}
\end{equation*}
Applying once more the stationarity of $\mathcal{L}$ and the edges $\mathcal{E}$ we also obtain the identities $F_{1}^{s}(\w)(v,(I-z)_d)=F_{1}^{s}(\tilde{\tau}_z\w)(v_z,I_d)$ and $F_{1}^{b}(\w)(u,(I-z)_d)=F_{1}^{b}(\tilde{\tau}_z\w)(u_z,I_d)$, which yields $\mu(I-z,\w)=\mu(I,\tilde{\tau}_z\w)$, and hence the stationarity of the process.

Since $\mu(I,\w)\geq 0$, it remains to prove the subadditivity of the process. To this end, let $I\in\I$ and let $(I^i)_{i=1}^k$ be a finite family of pairwise disjoint $(d-1)$-dimensional intervals with $I=\bigcup_{i=1}^k I^i$. For fixed $i\in\{1,\ldots, k\}$ we choose $(u^i,v^i)\in\mathcal{S}_{1,M}^{\w}(u_{0,\nu}^{1,0},I_d^i)\times\mathcal{PC}_{1,M}^\w(v_{0,\nu}^1,I_d^i)$ such that $F_{1}^{b}(\w)(u^i,v^i,I_d^i)=0$ and
\begin{align*}
\mu(I^i,\w)=F_{1}^{s}(\w)(v^i,I_d^i)+C_\mu\Hdtwo(\partial I^i).
\end{align*}
Note that also the $d$-dimensional intervals $I_d^i$ are pairwise disjoint. This allows us to define a pair $(u,v)\in\mathcal{S}_{1,M}^\w(u_{0,\nu}^{1,0},I_d)\times\mathcal{PC}_{1,M}^\w(v_{0,\nu}^1,I_d)$ by setting 
\begin{align*}
u(x):=
\begin{cases}
u^i(x) &\text{if}\ x\in I_d^i\ \text{for some}\ 1\leq i\leq k,\\
u_{0,\nu}^{1,0}(x) &\text{otherwise},
\end{cases}
\quad v(x):=
\begin{cases}
v^i(x) &\text{if}\ x\in I_d^i\ \text{for some}\ 1\leq i\leq k,\\
v_{0,\nu}^1(x) &\text{otherwise}.
\end{cases}
\end{align*}
Since $mO_\nu({\rm int}\, I\times(-1/2,1/2))\subset{\rm int}\,\bigcup_{i=1}^k\overline{I_d^i}$ and $m>4M$, thanks to the boundary conditions satisfied by $(u,v)$ we have
\begin{align*}
F_{1}^{b}(\w)(u,v,I_d)=F_{1}^{b}(\w)(u,v,{\rm int}\,\bigcup_{i=1 }^k\overline{I_d^i}),\qquad F_{1}^{s}(\w)(v,I_d)=F_{1}^{s}(\w)(v,{\rm int}\,\bigcup_{i=1 }^k\overline{I_d^i}).
\end{align*}
Let us show that $F_{1}^{b}(\w)(u,v,{\rm int}\,\bigcup_{i=1 }^k\overline{I_d^i})=0$, so that $v$ is admissible for $\mu(I,\w)$. Since by construction $F_{1}^{b}(\w)(u,v,I_d^i)=0$ for every $i\in\{1,\ldots,k\}$, it suffices to show that for any $(x,y)\in\mathcal{E}(\w)$ with $x\in \overline{I_d^i}$ and $y\in \overline{I_d^j}$ for some $i\neq j$ we have $v(x)^2|u(x)-u(y)|^2=0$. To this end, we notice that for such a pair $(x,y)$ we have $\dist(x,\partial I_i)\leq |x-y|\leq M$, so that $u(\e x)=u_{0,\nu}^{1,0}(x)$ and $v(x)=v_{0,\nu}^1$. Similarly $u(y)=u_{0,\nu}^{1,0}(y)$ and $v(y)=v_{0,\nu}^1(y)$. In particular, $|u(x)-u(y)|\neq 0$ if and only if $x$ and $y$ lie on different sides of the hyperplane $H_{\nu}$. Since $|x-y|\leq M$ this implies $|\langle x,\nu\rangle|\leq M$, so that $v(x)=0$. We conclude that indeed $F_{1}^{b}(\w)(u,v,I_d)=0$. Moreover, by the definition of $v$ we have
\begin{align*}
F_{1}^{s}(\w)(v,{\rm int}\,\bigcup_{i=1 }^k\overline{I_d^i})\leq\sum_{i=1}^kF_{1}^{s}(\w)(v^i,I_d^i)+\frac{\beta}{2}\sum_{1\leq i\neq j\leq k}\Big(\sum_{x\in\overline{I_d^i}\cap\overline{I_d^j}}\hspace*{-0.5em}(v(x)-1)^2+\frac{1}{2}\sum_{\substack{(x,y)\in\mathcal{E}(\w)\\x\in\overline{I_d^i},y\in\overline{I_d^j}}}\hspace*{-0.5em}|v(x)-v(y)|^2\Big).
\end{align*}
Fix $i,j\in\{1,\ldots,k\}$, $i\neq j$ and let $x\in\overline{I_d^i}\cap\overline{I_d^j}$. Then $(v(x)-1)^2=(v_{0,\nu}^1(x)-1)^2\neq 0$ only if $|\langle x,\nu\rangle|\leq M$, so that
\begin{align}\label{est:dist1}
\dist(x,\overline{mO_\nu I^i}\cap\overline{mO_\nu I^j})\leq M.
\end{align}
Further, at the points $x\in\overline{I_d^i}$, $y\in\overline{I_d^j}$ such that $(x,y)\in\mathcal{E}(\w)$ $v$ satisfies the boundary conditions, so that $|v(x)-v(y)|=|v_{0,\nu}^1(x)-v_{0,\nu}^1(y)|\neq 0$ only if $|\langle x,\nu\rangle |\leq M<|\langle y,\nu\rangle|$ or $|\langle y,\nu\rangle|\leq M<|\langle x,\nu\rangle|$. Since $|x-y|\leq M$, in both cases we have $|\langle x,\nu\rangle|,|\langle y,\nu\rangle|\leq 2M$. Thus, denoting by $p_\nu$ the orthogonal projection onto the hyperplane $H_{\nu}$, we obtain $|p_\nu(x)-x|,|p_\nu(y)-y|\leq 2M$. Moreover the segment $[p_\nu(x),p_\nu(y)]$ intersects the $(d-2)$-dimensional set $\overline{mO_\nu I^i}\cap\overline{mO_\nu I^j}$ and we deduce that 
\begin{align}\label{est:dist2}
\dist(x,\overline{mO_\nu I^i}\cap\overline{mO_\nu I^j})\leq 3M,\quad \dist(y,\overline{mO_\nu I^i}\cap\overline{mO_\nu I^j})\leq 3M.
\end{align}
Gathering \eqref{est:dist1} and \eqref{est:dist2} yields the existence of a constant $C=C(R/r,M,m)>0$ such that
\begin{align}\label{est:FsId}
F_{1}^{s}(\w)(v,I_d)\leq\sum_{i=1}^k F_{1}^{s}(\w)(v^i,I_d^i)+ C\sum_{1\leq i\neq j\leq k} \Hdtwo(\overline{I^i}\cap\overline{I^j}).
\end{align}
Since $v$ is admissible for $\mu(I,\w)$, keeping in mind that
\begin{align*}
\Hdtwo(\partial I)=\sum_{i=1}^k\Hdtwo(\partial I^i)-\sum_{1\leq i\neq j\leq k}\Hdtwo(\overline{I^i}\cap\overline{I^j}),
\end{align*}
from \eqref{est:FsId} we deduce that
\begin{align*}
\mu(I,\w) &\leq F_{1}^{s}(\w)(v,I_d)+C_\mu\Hdtwo(\partial I)\leq\sum_{i=1}^k\mu(I^i,\w)+(C-C_\mu)\sum_{1\leq i\neq j\leq k}\Hdtwo(\overline{I^i}\cap\overline{I^j}),
\end{align*}
hence the subadditivity follows provided we choose $C_\mu >C$.

Since the contribution $C_{\mu}\mathcal{H}^{d-2}(\partial I)$ is of lower order with respect to the surface scaling $t^{d-1}$, applying Theorem \ref{t.subadergodic} yields the existence of a set $\Omega^\nu$ of full probability and a function $\varphi_{\rm hom}(\w,\nu)$ such that for every $\w\in\Omega^\nu$ there holds
\begin{align}\label{ex:limit:int}
\varphi_{\rm hom}(\w,\nu)=\lim_{k\to +\infty}\frac{1}{(2mk)^{d-1}}\,\varphi_{1,M}^\w\left(u_{0,\nu}^{1,0},Q_\nu(0,2mk)\right).
\end{align}
Thanks to Lemma \ref{l.extension} the passage from the integer sequence $(2mk)_{k\in\N}$ to arbitrary sequences is now straightforward. Indeed, let $t_k\to+\infty$ be arbitrary and set $t^-_{k}:=2m\lfloor t_k\rfloor$, $t^+_{k}:=2m(\lfloor t_k\rfloor+1)$. Applying Lemma \ref{l.extension} with the cubes $Q_\nu(0,t_k)$ and $Q_\nu(0,t^+_k)$ then yields
\begin{equation}\label{est:nonint1}
\varphi_{1,M}^\w(u_{0,\nu}^{1,0},Q_\nu(0,t^+_k))\leq \varphi_{1,M}^\w(0,t_k)+c(t^+_k)^{d-2}.
\end{equation}
Again applying Lemma \ref{l.extension} with cubes $Q_\nu(0,t^-_k)$ and $Q_\nu(0,t_k)$ gives
\begin{equation}\label{est:nonint2}
\varphi_{1,M}^\w(0,t_k)\leq \varphi_{1,M}^\w(u_{0,\nu}^{1,0},Q_\nu(0,t^-_k))+c(t_k)^{d-2}.
\end{equation}
Dividing by $(t_k)^{d-1}$ and gathering \eqref{ex:limit:int}, \eqref{est:nonint1} and \eqref{est:nonint2} we get
\begin{align*}
\limsup_{k\to+\infty}\frac{1}{t_k^{d-1}}\varphi_{1,M}^\w(0,t_k)\leq\varphi_{\rm hom}(\w,\nu)\leq\liminf_{k\to+\infty}\frac{1}{t_k^{d-1}}\varphi_{1,M}^\w(0,t_k).
\end{align*}
Since the sequence $(t_k)$ was arbitrarily chosen we deduce that for all $\w$ belonging to the set of full measure $\widehat{\Omega}:=\bigcap_{\nu\in S^{d-1}\cap \Q^d}\Omega^\nu$, for every $\nu\in S^{d-1}\cap\Q^d$ there exists the limit
\begin{align}\label{ex:limit:rational}
\varphi_{\rm hom}(\w,\nu)=\lim_{t\to+\infty}\frac{1}{t^{d-1}}\varphi_{1,M}^\w\left(u_{0,\nu}^{1,0},Q_\nu(0,t)\right).
\end{align}

\medskip
\noindent\textbf{Step 2} From rational to irrational directions.\\
We continue by proving that \eqref{ex:limit:rational} holds for every $\w\in\widehat{\Omega}$ and every $\nu\in S^{d-1}$. To this end, for every $\w\in\Omega$ and $\nu\in S^{d-1}$ we introduce the auxiliary functions
\begin{align*}
\overline{\varphi}(\w,\nu):=\limsup_{t\to +\infty}\frac{1}{t^{d-1}}\varphi_{1,M}^\w\left(u_{0,\nu}^{1,0},Q_\nu(0,t)\right),\qquad \underline{\varphi}(\w,\nu):=\liminf_{t\to +\infty}\frac{1}{t^{d-1}}\varphi_{1,M}^\w\left(u_{0,\nu}^{1,0},Q_\nu(0,t)\right),
\end{align*}
and we observe that for every $\w\in\widehat{\Omega}$ and $\nu\in S^{d-1}\cap\Q^d$ we have
\begin{align}\label{eq:ls-li}
\overline{\varphi}(\w,\nu)=\underline{\varphi}(\w,\nu)=\varphi_{\rm hom}(\w,\nu).
\end{align}
We now aim to extend this equality to every $\w\in\widehat{\Omega}$ and every $\nu\in S^{d-1}$ by density of $S^{d-1}\cap\Q^d$ in $S^{d-1}$.

Let $\w\in\widehat{\Omega}$ and $\nu\in S^{d-1}\setminus\mathbb{Q}^{d}$. As the inverse of the stereographic projection maps rational points to rational directions, we find a sequence $(\nu_j)\subset S^{d-1}\cap\mathbb{Q}^{d}$ converging to $\nu$. In particular, since $\nu\neq -e_d$, it follows by the continuity of $\nu\mapsto O_{\nu}$ that for fixed $\eta>0$ there exists an index $j_0=j_0(\eta)$ such that for all $j\geq j_0$ we have
\begin{itemize}
\item[(i)] $Q_{\nu_j}(0,1-\eta)\wcont Q_\nu(0,1)\wcont Q_{\nu_j}(0,1+\eta)$;
\item[(ii)] $\dist_{\mathcal{H}}\left(H_{\nu}\cap B_2,H_{\nu_j}\cap B_2\right)\leq \eta$,
\end{itemize}
where $\dist_{\mathcal{H}}$ denotes the Hausdorff distance. Using similar arguments as in the proof of Lemma \ref{l.extension} we aim to compare the two quantities $\varphi_{1,M}^\w(u_{0,\nu}^{1,0},Q_\nu(0,t))$ and $\varphi_{1,M}^\w(u_{0,\nu_j}^{1,0},Q_{\nu_j}(0,(1-\eta)t))$. To simplify notation we set
\begin{align*}
Q(t):=Q_\nu(0,t),\qquad Q_j^\eta(t):=Q_{\nu_j}(0,(1-\eta)t).
\end{align*}
For $j\geq j_0$ and $t>0$ we choose a pair $(u_j^t,v_j^t)\in\mathcal{S}_{1,M}^\w(u_{0,\nu_j}^{1,0},Q_j^\eta(t))\times \mathcal{PC}_{1,M}^\w(v_{0,\nu_j}^{1},Q_j^\eta(t))$ satisfying $F_{1}^{b}(\w)(u_j^t,v_j^t,Q_{j}^\eta(t))=0$ and $F_{1}^{s}(\w)(v_j^t,Q_{j}^\eta(t))=\varphi_{1,M}^\w(u_{0,\nu_j}^{1,0},Q_{j}^\eta(t))$. Moreover, we observe that thanks to (i) for t sufficiently large we have $\dist(Q_j^\eta(t),\partial Q(t))>2M$. This allows us to extend $u_j^t$ to a function $\tilde{u}_j^t\in\mathcal{S}_{1,M}^\w(u_{0,\nu}^{1,0},Q(t))$ by setting $\tilde{u}_j^t(x):=u_{0,\nu}^{1,0}(x)$ on $\Lw\cap Q(t)\setminus Q_j^\eta(t)$. We now construct a function $\tilde{v}_j^t\in\mathcal{PC}_{1,M}^\w((v_{0,\nu}^{1},Q(t))$ satisfying $F_{1}^{b}(\w)(\tilde{u}_j^t,\tilde{v}_j^t,Q(t))=0$ and which has almost the same energy as $v_j^t$. To this end, we consider the cone
\begin{align*}
K(\nu,\nu_j):=\{x\in\R^d\colon\langle x,\nu\rangle\langle x,\nu_j\rangle\leq 0\},
\end{align*}
and we set
\begin{align*}
L(\nu):=\{x\in\R^d\colon |\langle x,\nu\rangle|\leq M\},\qquad L(\nu_j):=\{x\in\R^d\colon |\langle x,\nu_j\rangle|\leq M\}.
\end{align*}
We denote by $U(\nu,\nu_j):=K(\nu,\nu_j)\cup L(\nu)\cup L(\nu_j)$ the union of the three sets above.
We then define $\tilde{v}_j^t$ by its values on $\Lw$ via
\begin{align*}
\tilde{v}_j^t(x):=
\begin{cases}
v_j^t(x) &\text{if $x\in Q_{j}^\eta(t)\setminus (U(\nu,\nu_j)\cap \partial_M Q_j^\eta(t))$},\\
0 &\text{if $x\in\left(U(\nu,\nu_j)\cap \partial_M Q_j^\eta(t)\right)\cup\left( L(\nu)\setminus Q_j^\eta(t)\right)$}.\\
1 &\text{otherwise},
\end{cases}
\end{align*}
(see Figure \ref{fig:rotated}).
\begin{figure}[h]
\centering
\def\svgwidth{0.35\columnwidth}
\begingroup%
  \makeatletter%
  \providecommand\color[2][]{%
    \errmessage{(Inkscape) Color is used for the text in Inkscape, but the package 'color.sty' is not loaded}%
    \renewcommand\color[2][]{}%
  }%
  \providecommand\transparent[1]{%
    \errmessage{(Inkscape) Transparency is used (non-zero) for the text in Inkscape, but the package 'transparent.sty' is not loaded}%
    \renewcommand\transparent[1]{}%
  }%
  \providecommand\rotatebox[2]{#2}%
  \newcommand*\fsize{\dimexpr\f@size pt\relax}%
  \newcommand*\lineheight[1]{\fontsize{\fsize}{#1\fsize}\selectfont}%
  \ifx\svgwidth\undefined%
    \setlength{\unitlength}{696.20945247bp}%
    \ifx\svgscale\undefined%
      \relax%
    \else%
      \setlength{\unitlength}{\unitlength * \real{\svgscale}}%
    \fi%
  \else%
    \setlength{\unitlength}{\svgwidth}%
  \fi%
  \global\let\svgwidth\undefined%
  \global\let\svgscale\undefined%
  \makeatother%
  \begin{picture}(1,0.95052576)%
    \lineheight{1}%
    \setlength\tabcolsep{0pt}%
    \put(0,0){\includegraphics[width=\unitlength,page=1]{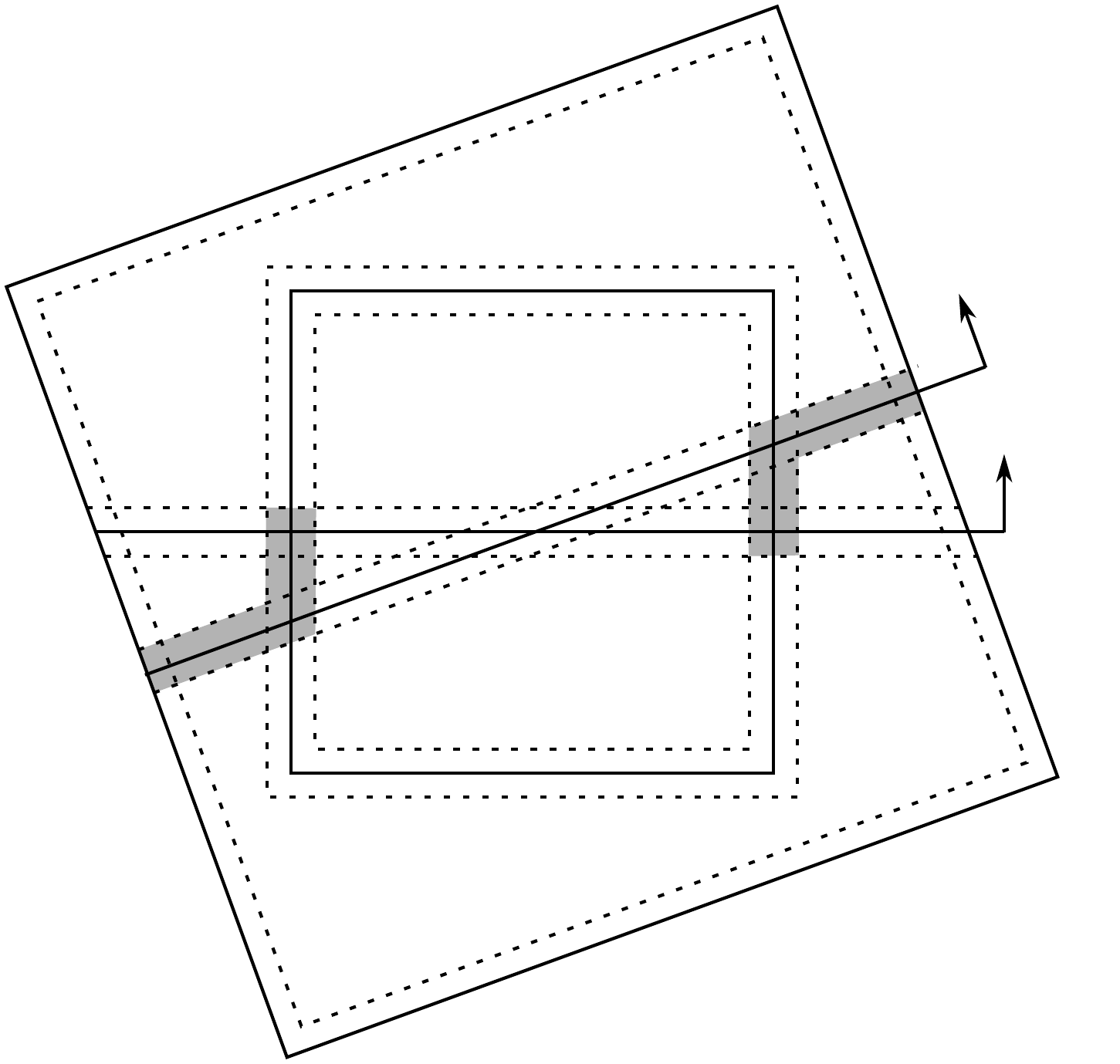}}%
    \put(0.92,0.48){\color[rgb]{0,0,0}\makebox(0,0)[lt]{\lineheight{1.25}\smash{\begin{tabular}[t]{l}$\nu_j$\end{tabular}}}}%
    \put(0.88,0.64300799){\color[rgb]{0,0,0}\makebox(0,0)[lt]{\lineheight{1.25}\smash{\begin{tabular}[t]{l}$\nu$\end{tabular}}}}%
    \put(0,0){\includegraphics[width=\unitlength,page=2]{zeichnung.pdf}}%
  \end{picture}%
\endgroup%

\caption{{\footnotesize The two cubes $Q_j^\eta(t)$ and $Q(t)$ and in gray the set $U(\nu,\nu_j)\cap\partial_M Q_j^\eta(t)\cup\left( L(\nu)\setminus Q_j^\eta(t)\right) $.}}
\label{fig:rotated}
\end{figure} 
Let us now verify that $F_{1}^{b}(\w)(\tilde{u}_j^t,\tilde{v}_j^t,Q(t))=0$. First observe that for all $x\in\Lw\cap Q_j^\eta(t)$ we have $\tilde{v}_j^t(x)\in\{0,v_j^t(x)\}$ and hence $F_{1}^b(\w)(\tilde{u}_j^t,\tilde{v}_j^t,Q_j^\eta(t))=0$ by hypotheses. Suppose now that $(x,y)\in\mathcal{E}(\w)\cap (Q(t)\times Q(t))$ with at least one point belonging to $Q(t)\setminus Q_j^\eta(t)$ and $\tilde{u}_j^t(x)\neq \tilde{u}_j^t(y)$. Then we can distinguish the following two cases:
\begin{enumerate}[label=(\alph*)]
\setlength{\itemsep}{3pt}
	\item $x\in Q_j^\eta(t)$ and $y\in Q(t)\setminus Q_j^\eta(t)$: since $|x-y|\leq M$ we have $\tilde{u}_j^t(x)=u_j^t(x)=u_{0,\nu_j}^{1,0}(x)$. Moreover, by definition it holds that $\tilde{u}_j^t(y)=u_{0,\nu}^{1,0}(y)$. In particular, $u_{0,\nu_j}^{1,0}(x)\neq u_{0,\nu}^{1,0}(y)$. The latter implies that $x,y\in U(\nu,\nu_j)$, so that $\tilde{v}_j^t(x)=\tilde{v}_j^t(y)=0$, which yields $\tilde{v}_j^t(x)^2|\tilde{u}_j^t(x)-\tilde{u}_j^t(y)|^2=0$;
	\item $x,y\in Q(t)\setminus Q_j^\eta(t)$: then necessarily $x,y\in L(\nu)$, so that $\tilde{v}_j^t(x)=\tilde{v}_j^t(y)=0$ and we conclude again.
\end{enumerate}
The above discussion shows that $\tilde{v}_j^t$ is admissible for $\varphi_{1,M}^\w(u_{0,\nu}^{1,0},Q(t))$ (note that $\tilde{v}_j^t$ also satisfies the correct boundary conditions). Moreover, the same reasoning as in Lemma \ref{l.extension} leads to the estimate
\begin{align*}
F_{1}^{s}(\w)(\tilde{v}_j^t,Q(t)) &\leq F_{1}^{s}(\w)(v_j^t,Q_j^\eta(t))+C\left(\Hd(K(\nu,\nu_j)\cap \partial Q_j^\eta(t))+\Hd(H_\nu\cap Q(t)\setminus Q_j^\eta(t)\right)\\
&\leq F_{1}^{s}(\w)(v_j^t, Q_j^\eta(t))+C\eta\,t^{d-1},
\end{align*}
where the second inequality follows thanks to (ii).
Dividing the above inequality by $t^{d-1}$ and passing to the upper limit as $t\to +\infty$, in view of the choice of $v_j^t$ we obtain
\begin{align*}
\overline{\varphi}(\w,\nu)\leq\limsup_{t\to+\infty}\frac{1}{t^{d-1}} F_{1}^{s}(\w)(\tilde{v}_j^t, Q(t))\leq \varphi_{\rm hom}(\w,\nu_j)+C\eta.
\end{align*}
Thus, letting first $j\to+\infty$ and then $\eta\to 0$ gives $\overline{\varphi}(\w,\nu)\leq\liminf_j\varphi_{\rm hom}(\w,\nu_j)$. A similar argument, now using the second inclusion in (i), leads to the inequality $\limsup_j\varphi_{\rm hom}(\w,\nu_j)\leq\underline{\varphi}(\w,\nu)$. Hence the equality \eqref{eq:ls-li} extends to all $\nu\in S^{d-1}$ and the limit in \eqref{ex:limit:rational} exists for all directions. 

\medskip
\noindent{\bf Step 3} Shift invariance in the probability space\\
Next we find a set $\widetilde{\Omega}\subset\widehat{\Omega}$ on which $\varphi(\cdot,\nu)$ is invariant under the group action $\{\tau_z\}_{z\in\Z^d}$ for every $\nu\in S^{d-1}$. Namely, we define the set
\begin{align*}
\widetilde{\Omega}:=\bigcap_{z\in\Z^d}\tau_z(\widehat{\Omega}),
\end{align*}
which has full measure since $\tau_z$ is measure preserving. Moreover, as a consequence of Definition \ref{def.groupaction} every map $\tau_z$ is bijective, so that for every $z\in\Z^d$ we have $\tau_z(\widetilde{\Omega})=\widetilde{\Omega}\subset\widehat{\Omega}$, hence the limit defining $\varphi_{\rm hom}(\tau_z\w,\nu)$ exists for every $z\in\Z^d$ and every $\nu\in S^{d-1}$. Thus, it remains to prove that $\varphi_{\rm hom}(\tau_z\w,\nu)$ and $\varphi_{\rm hom}(\w,\nu)$ coincide. To this end it suffices to show that
\begin{align}\label{est:phitauz-phi}
\varphi_{\rm hom}(\tau_z\w,\nu)\leq\varphi_{\rm hom}(\w,\nu)
\end{align}
holds for every $z\in\Z^d$, $\w\in\widetilde{\Omega}$, and $\nu\in S^{d-1}$, then the opposite inequality follows by applying \eqref{est:phitauz-phi} with $z$ replaced by $-z$ and $\w$ replaced by $\tau_z\w$.

Let $z,\w,\nu$ be as above. There exists $N=N(z)>0$ such that for all $t>0$ it holds that
\begin{align}\label{cond:Qnz}
Q_\nu(0,t)\subset Q_\nu(-z,N+t),\quad 2M<\dist(\partial Q_\nu(0,t),\partial Q_\nu(-z,N+t)).
\end{align}
An argument similar to the one used to prove the stationarity of the stochastic process shows that 
\begin{align*}
\varphi_{\rm hom}(\tau_z\w,\nu)&=\lim_{t\to+\infty}\frac{1}{(N+t)^{d-1}}\varphi_{1,M}^\w(u_{-z,\nu}^{1,0},Q_{\nu}(-z,N+t))\nonumber\\
&=\lim_{t\to+\infty}\frac{1}{t^{d-1}}\varphi_{1,M}^\w(u_{-z,\nu}^{1,0},Q_{\nu}(-z,N+t)).
\end{align*}
Moreover, in view of \eqref{cond:Qnz} for $t$ sufficiently large the cubes $Q_\nu(0,t)$ and $Q_\nu(-z,N+t)$ satisfy all the conditions of Lemma \ref{l.extension}. Hence we deduce that
\begin{align*}
\varphi_{1,M}^\w(u_{-z,\nu}^{1,0},Q_\nu(-z,N+t))\leq\varphi_{1,M}^\w(u_{0,\nu}^{1,0},Q_\nu(0,t))+c(|z|+N)(t+N)^{d-2},
\end{align*}
and we obtain \eqref{est:phitauz-phi} by dividing the above inequality by $t^{d-1}$ and passing to the limit as $t\to+\infty$.
\end{proof}
It is by now standard to show that in the limit defining $\varphi_{\rm hom}$ the cubes $Q_\nu(0,t)$ can be replaced by $Q_\nu(tx,t\varrho)$ with $x\in\R^d$, $\rho>0$ arbitrary. In fact, the following proposition can be proved by repeating the arguments in the proof of \cite[Theorem 5.8]{BCR} (see also \cite[Theorem 5.5]{ACR}) and applying Lemma \ref{l.extension} and Proposition \ref{prop:ex:limit0} above. We thus omit its proof here.
\begin{proposition}\label{prop:ex:limit:x}
Let $\mathcal{L}$ be an admissible stationary stochastic lattice with admissible stationary edges in the sense of Definitions \ref{defadmissible} \& \ref{defgoodedges}. Then there exists $\Omega'\subset\Omega$ with $\mathbb{P}(\Omega')=1$ such that for every $\w\in\Omega'$ and every $x\in D$, $\nu\in S^{d-1}$, $\varrho>0$ there holds
\begin{align}\label{ex:limit:x}
\varphi_{\rm hom}(\w,\nu)=\lim_{t\to+\infty}(t\varrho)^{1-d}\varphi_{1,M}^\w\left(u_{tx,\nu}^{1,0},Q_\nu(tx,t\rho)\right),
\end{align}
where $\varphi_{\rm hom}$ is given by Proposition \ref{prop:ex:limit0}. In particular, the limit in \eqref{ex:limit:x} exists and is independent of $x$ and $\varrho$.
\end{proposition}
We finally prove Theorem \ref{mainthm1}. Combining \cite[Theorem 2]{ACG2} and 
\begin{proof}[Proof of Theorem \ref{mainthm1}]
Proposition \ref{prop:ex:limit0} above yields the existence of a set $\Omega'\subset\Omega$ with $\mathbb{P}(\Omega')=1$ such that for all $\w\in\Omega'$ the limit in \eqref{ex:hom:surf} exists for every $\xi\in\R^d$ and every $\nu\in S^{d-1}$ and \eqref{ex:limit:x} holds true. In addition, \cite[Theorem 2]{ACG2} proves the existence of a set of full measure (without loss of generality $\Omega'$) such that the limit in~\eqref{ex:hom:bulk} exists for every $\omega\in\Omega'$. Moreover, since $\mathcal{L}$ is an admissible stochastic lattice with admissible edges $\mathcal{E}$, for every $\w\in\Omega'$ and every $\e\to 0$ Theorem \ref{mainrep} provides us with a subsequence $\e_n$ and a functional $F(\w):L^1(D)\times L^1(D)\to[0,+\infty]$ of the form
\begin{align*}
F(\w)(u,1)=\int_D f(\w,x,\nabla u)\dx+\int_{S_u}\varphi(\w,x,\nu_u)\dHd,\quad u\in GSBV^2(D),
\end{align*}
such that $F_{\e_n}(\w)$ $\Gamma$-converges to $F(\w)$ in the strong $L^1(D)\times L^1(D)$-topology. Thanks to Proposition \ref{p.gradientpartsequal} we know that
\begin{align*}
f(\w,x_0,\xi)=f_{\rm hom}(\w,\xi)\quad\text{for a.e. $x_0\in D$ and every $\xi\in\R^d$},
\end{align*}
with $f_{\rm hom}(\w,\xi)$ given by \eqref{ex:hom:bulk}, where we have used that thanks to \cite[Theorem 2]{ACG2} $f_{\rm hom}$ does not depend on $x_0$.
Moreover, combining the asymptotic formula for $\varphi$ in Proposition \ref{limitsbvp} with Lemma \ref{approxminprob}, Lemma \ref{separationofscales1}, and a change of variables yields
\begin{align}\label{as:form-resc}
\varphi(\w,x,\nu)=\limsup_{\varrho\to 0}\lim_{\delta\to 0}\limsup_{n\to+\infty}\frac{1}{(t_n\varrho)^{d-1}}\varphi_{1,t_n\delta}^\w(u_{t_n x,\nu}^{1,0},Q_\nu(t_nx,t_n\varrho)),
\end{align}
where $t_n=\e_n^{-1}$.
Since for every fixed $\delta>0$ we have $\delta t_n>M$ for $t_n$ sufficiently large, from \eqref{as:form-resc} and Proposition \ref{prop:ex:limit:x} we immediately deduce that $\varphi(\w,x,\nu)\geq\varphi_{\rm hom}(\w,\nu)$ for every $x\in D$, $\nu\in S^{d-1}$. 

To prove the opposite inequality we fix $\varrho>0$ and $\delta\in (0,\varrho)$. Then a procedure similar to the one used in the proof of Lemma \ref{l.extension} allows to extend any pair
\begin{equation*}
(u_n,v_n)\in\mathcal{S}_{1,M}^\w(u_{t_nx,\nu}^{1,0},Q_\nu(t_n x,t_n(\varrho-\delta)))\times\mathcal{PC}_{1,M}^\w(v_{t_nx,\nu}^1,Q_\nu(t_nx,t_n(\varrho-\delta)))
\end{equation*}
to $Q_\nu(t_n x,t_n\varrho)$ in such a way that $v_n$ is admissible for $\varphi_{1,t_n\delta}^\w(u_{t_nx,\nu}^{1,0},Q_\nu(t_nx,t_n\varrho))$ and
\begin{align*}
F_{1}^{s}(\w)(v_n,Q_\nu(t_nx,t_n\varrho))\leq F_{1}^{s}(\w)(v_n,Q_\nu(t_nx,t_n(\varrho-\delta)))+Ct_n^{d-1}\delta.
\end{align*}
Passing to the infimum and dividing the above inequality by $(t_n\varrho)^{d-1}$ we obtain $\varphi(\w,x,\nu)\leq\varphi_{\rm hom}(\w,\nu)$ by letting first $n\to +\infty$ and then $\delta\to 0$. Hence the limit is determined uniquely independent of the subsequence. The claim then follows from the Urysohn-property of $\Gamma$-convergence and the fact that the ergodicity of the group action makes the functions $\varphi_{\rm hom}$ and $f_{\rm hom}$ deterministic due to \eqref{eq:shiftinv} and \cite[Theorem 2]{ACG2}, respectively. 
\end{proof}
Finally, we prove the approximation of the Mumford-Shah functional in the isotropic case.
\begin{proof}[Proof of Theorem \ref{MSapprox}]
Due Theorem \ref{mainthm1} it only remains to show that $f_{\rm hom}(\xi)=c_1|\xi|^2$ and $\varphi(\nu)=c_2$ for some constants $c_1,c_2>0$. By Theorem \ref{mainrep} the function $f$ is a non-negative quadratic form. Reasoning exactly as for the vectorial case treated in \cite[Theorem 9]{ACG2} one can show that ergodicity and isotropy imply $f(R\xi)=f(\xi)$ for all $\xi\in\mathbb{R}^{d}$ and all $R\in SO(d)$. Hence $f$ is constant on $S^{d-1}$ and has to be of the form $f(\xi)=c_1|\xi|^2$ for some $c_1>0$. 
	
We next show that $\varphi_{\rm hom}(R\nu)=\varphi_{\rm hom}(\nu)$ for all $R\in SO(d)$. Recall that $\tau'_R$ denotes a measure preserving map such that $\mathcal{L}\circ\tau'_R=R\mathcal{L}$ for all $R\in SO(d)$. Next observe that by this isotropy property of $\mathcal{L}$ we have the equivalences
\begin{align*}
u\in \mathcal{S}_{1,M}^{\w}(u^{a,0}_{0,R\nu},Q_{R\nu}(0,t))&\iff u\circ R\in \mathcal{S}_{1,M}^{\tau'_{R^T}\w}(u^{a,0}_{0,\nu},Q_{\nu}(0,t)),
\\
v\in\mathcal{PC}^{\w}_{1,M}(v_{0,R\nu}^{1},Q_{R\nu}(0,t))&\iff v\circ R\in \mathcal{PC}_{1,M}^{\tau'_{R^T}\w}(v_{0,\nu}^{1},Q_{\nu}(0,t)).
\end{align*}
Moreover, by the joint isotropy of $\mathcal{L}$ and of the edges $\mathcal{E}$, it holds that 
\begin{equation*}
F_{1}^{b}(\w)(u,v,Q_{R\nu}(0,t))=F^{b}_{1}(\tau'_{R^T}\w)(u\circ R,v\circ R,Q_{\nu}(0,t)),\;\; F_{1}^{s}(\w)(v,Q_{R\nu}(0,t))=F^{s}_{1}(\tau'_{R^T}\w)(v\circ R,Q_{\nu}(0,t)).
\end{equation*}
Hence, from definition \eqref{def:surf:int} we conclude that
\begin{equation*}
\varphi_{1,M}^{\w}(u_{0,R\nu}^{a,0},Q_{R\nu}(0,t))=\varphi_{1,M}^{\tau'_{R^T}\w}(u_{0,\nu}^{a,0},Q_{\nu}(0,t)).
\end{equation*}
Since $\varphi_{\rm hom}$ is deterministic by ergodicity, we can take expectations in the asymptotic formula given by \eqref{ex:hom:surf} and due to the fact that $\tau^{\prime}_{R}$ is measure preserving, by dominated convergence and a change of variables we obtain
\begin{align*}
\varphi_{\rm hom}(R\nu)&=\lim_{t\to +\infty}\frac{1}{t^{d-1}}\int_{\Omega}\varphi^{\w}_{1,M}(u_{0,R\nu}^{a,0},Q_{R\nu}(0,t))\,\mathrm{d}\mathbb{P}(\w)=\lim_{t\to +\infty}\frac{1}{t^{d-1}}\int_{\Omega}\varphi_{1,M}^{\tau'_{R^T}\w}(u_{0,\nu}^{a,0},Q_{\nu}(0,t))\,\mathrm{d}\mathbb{P}(\w)
\\
&=\lim_{t\to +\infty}\frac{1}{t^{d-1}}\int_{\Omega}\varphi_{1,M}^{\w'}(u_{0,\nu}^{a,0},Q_{\nu}(0,t))\,\mathrm{d}\mathbb{P}(\w')=\varphi_{\rm hom}(\nu).
\end{align*}
We finish the proof setting $c_2=\varphi_{\rm hom}(e_1)>0$.
\end{proof}
\section{Numerical results}\label{sec:numerics}
We complement the theoretical results proved in the previous sections with two numerical examples that illustrate the isotropic behavior of the random discretization considered in this paper. 

We start describing how to create the random lattice. The construction of the random lattice is based on the {\it random parking model} with parameter $r>0$, that we briefly describe below. On a fixed bounded domain $D$ one constructs a point set as follows  
\begin{itemize}
	\item [(1)] Choose a point $x_1\in D$ according to a uniform distribution.
	\item [(2)] For $i\geq 2$ choose the $i^{\rm th}$ point $x_i\in D$ according to a uniform distribution and accept it if $|x_i-x_j|\geq r$ for all $j<i$.
\end{itemize}
One obtains the so-called {\it jamming limit} repeating this process ad infinitum. When the domain $D$ invades the whole space in a suitable sense (for instance, take the sequence $D_n=(-n,n)^d$), then it was proven in \cite[Theorem 2.2]{Pe} that the corresponding jamming limits converge weakly (in the sense of measures) to a point process on the whole space $\R^d$, namely the {\it random parking process} in $\R^d$. This limit point process together with the associated Voronoi edges satisfy all the assumptions of Theorem \ref{MSapprox}. In practice, however one has to work with the finite approximation described in (1)-(2). Nevertheless, using similar arguments as in \cite[Lemma 2.5]{glpe} one can prove that the $\Gamma$-limit remains the same provided the random parking model is constructed on a sequence of rescaled Lipschitz set $D'/\e$ with $D\subset\subset D'$ via the graphical construction as described in \cite[Section 2.1]{glpe}. 

For standard images we create the random parking process in a finite rectangular box $Q=[0,X]\times [0,Y]$ (usually $X$ and $Y$ are the pixel dimensions of the original image).
We notice that the distance test (2) does not need to be done for all points. With an auxiliary list we reduce this test to a uniformly bounded number of points in each iteration step. In this way one can add points until condition (i) in Definition \ref{defadmissible} is satisfied inside $Q$ with a sufficiently small $R$. Even though in the (theoretical) jamming limit one can ensure that $R\leq 2r$, for our purposes, with $r=0.7$ in pixel units, it suffices to ensure that $R\leq 4r$. Instead of checking this condition, a more efficient stopping criterion for creating the random lattice is to stop the iteration process after a certain number ($300$ in the following examples) of unsuccessful iterations (cf. Figure \ref{Fig_latt} for an example).
Finally we mention that the stochastic lattice has to be created just once and can be saved for future usage. 
\begin{figure}[h]
	\includegraphics[trim={0 0 3.5cm 1.7cm},clip,scale=0.35]{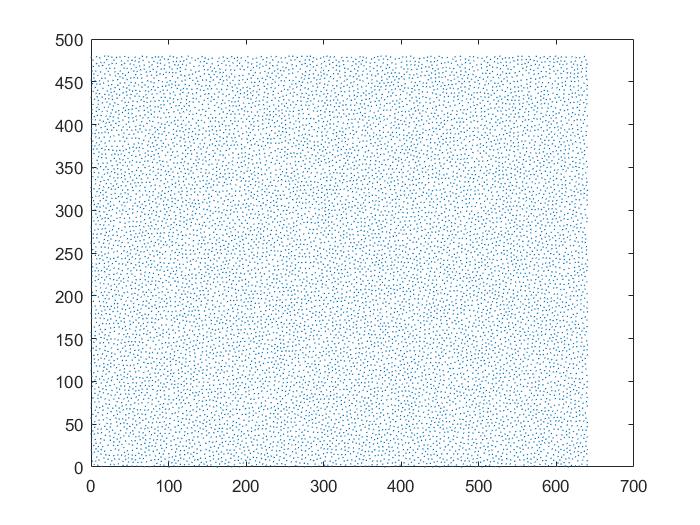}
	\caption{A realization of a stochastic lattice with $r=4$ on $Q=[0,640]\times[0,480]$.}\label{Fig_latt}
\end{figure}
After the lattice has been created one has to compute the Delaunay triangulation in order to obtain the Voronoi neighbors. In this step we also delete long edges close to the boundary of $Q$. 

In what follows, we compare the number of points and interactions of a discretization of a ($640\times 480$)-image with respect $\Z^2$ and a realization of $\Lw$ with $r=0.7$ which was the parameter we used in our examples. Clearly, with $\Z^2$ one has $640\times 480=307.200$ points and the number of interactions per point equals $4$ (except some points at the boundary that we neglect). For the random lattice, the average of $10$ realizations yields $320.630$ points with a maximal deviation of $(-10.893,+5.643)$. The average number of interactions per point equals $6$ up to boundary corrections of order $10^{-4}$. In Figure \ref{Fig_NN} we display a typical distribution of the number of interactions per point. One can conclude that the average number of points does not change significantly when one uses random lattices instead of $\Z^2$, while the number of interactions increases by a factor of $3/2$.
\begin{figure}[h]
	\includegraphics[trim={0 0 2.0cm 0.8cm},clip,scale=0.35]{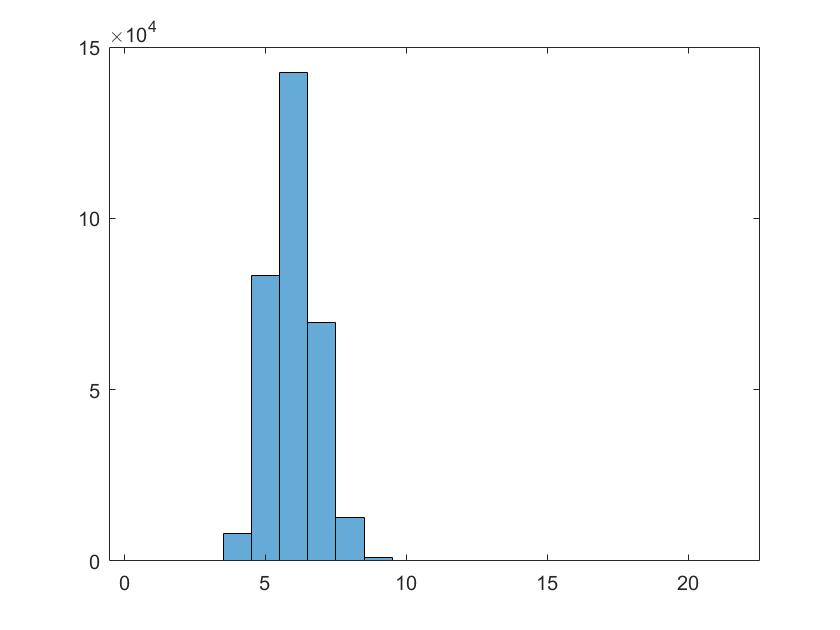}
	\caption{The distribution of the number of interactions per particle of a typical realization on $Q=[0,640]\times[0,480]$ with $r=0.7$.}\label{Fig_NN}
\end{figure}

Having created the lattice together with the edges one has to define a discrete version of the original image on the stochastic grid to construct the discrete fidelity term in \eqref{eq:fidelityterm}. To this end, at every point $x_i$ we define $g_{\e}(x_i)$ to be the value at the pixel obtained by taking componentwise the integer part of the coordinates of $x_i$ (of course other choices are possible).

After this preparation we apply the well-known method of alternate minimization for the Ambrosio-Tortorelli functionals. For this method, given a starting guess $u_0$ one minimizes the discrete functional with respect to $v$ and finds a first candidate $v_0$. Then, for fixed $v_0$ one minimizes with respect to $u$ and finds a candidate $u_1$. Note that each minimization requires to solve a linear equation. In the examples presented below we repeat this procedure until for two iterative solutions $u_k$, $u_{k+1}$ it holds $\|u_k-u_{k+1}\|/\|u_k\|<10^{-5}$.

In what follows we apply the procedure described above to simple but meaningful test images that help to illustrate the anisotropic behavior of the functionals in \eqref{eq:introBBZ} obtained by discretizing $AT_\e$ on a periodic lattice in contrast to the isotropic behavior of the discretization on a stochastic lattice in \eqref{eq:defF}.
In fact, we present two examples showing that the discretization on a square lattice prefers jump sets whose normal has a small supremum norm. Notice that this is also consistent with the results in \cite[Theorem 6.1 and Proposition 7.1]{BBZ}.
The first example (Figure \ref{rotsquares}) is the reconstruction of differently oriented squares. The tuning parameters are chosen as $\beta=29.85$ and $\gamma=5000$ in the periodic case, while in the stochastic case $\beta=25$. We notice that the lower constant $\beta$ in the stochastic setting makes the weight of the surface term per lattice cell comparable in both models. In fact it takes into account that the term $(v-1)^2$ has the same weight in both models, while the proportion between the gradient term in the square lattice and the gradient term in the stochastic lattice is $2/3$. The latter corresponds to the proportion between the number of interactions in the square lattice and the average number of interactions in the stochastic one.
\begin{figure}[h]
 \includegraphics[width=\textwidth]{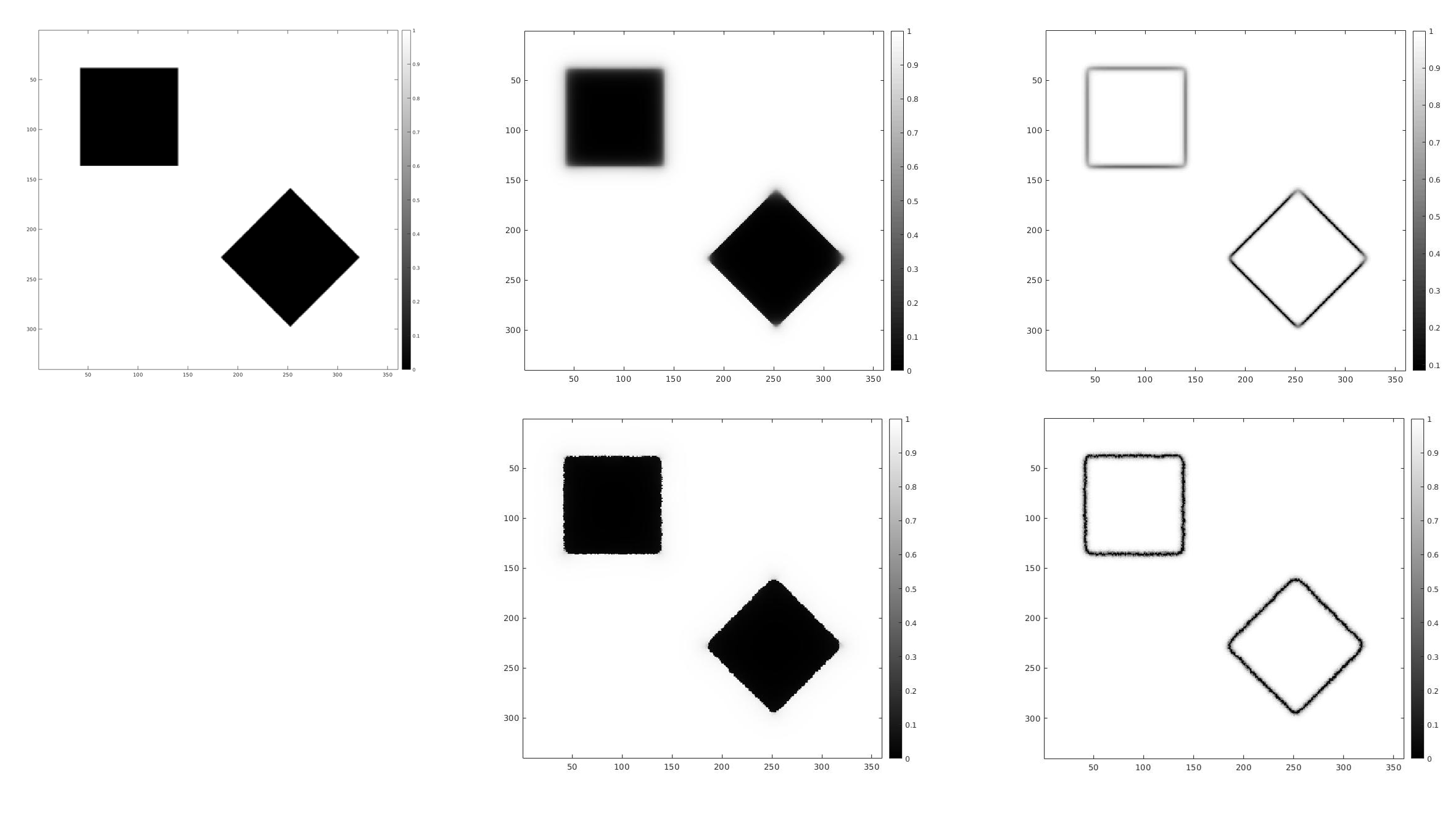}
 	\caption{Above (from left to right): original image; reconstructed image using the discretization on a square lattice, corresponding edge variable; bottom (from left to right): reconstructed image using random discretization, corresponding edge variable}\label{rotsquares}
\end{figure}

The second example (Figure \ref{circle}) shows the reconstruction of a circle. For the periodic  functionals the tuning parameters are chosen as $\beta=28$, $\gamma=4500$, while for the stochastic ones we chose $\beta=23$. We display in each case the reconstructed image together with the corresponding edge variable and a binary plot of the sublevelset $\{v\leq 0.2\}$, the latter one making the anisotropic behavior more evident.
\begin{figure}[h]
\includegraphics[width=\textwidth]{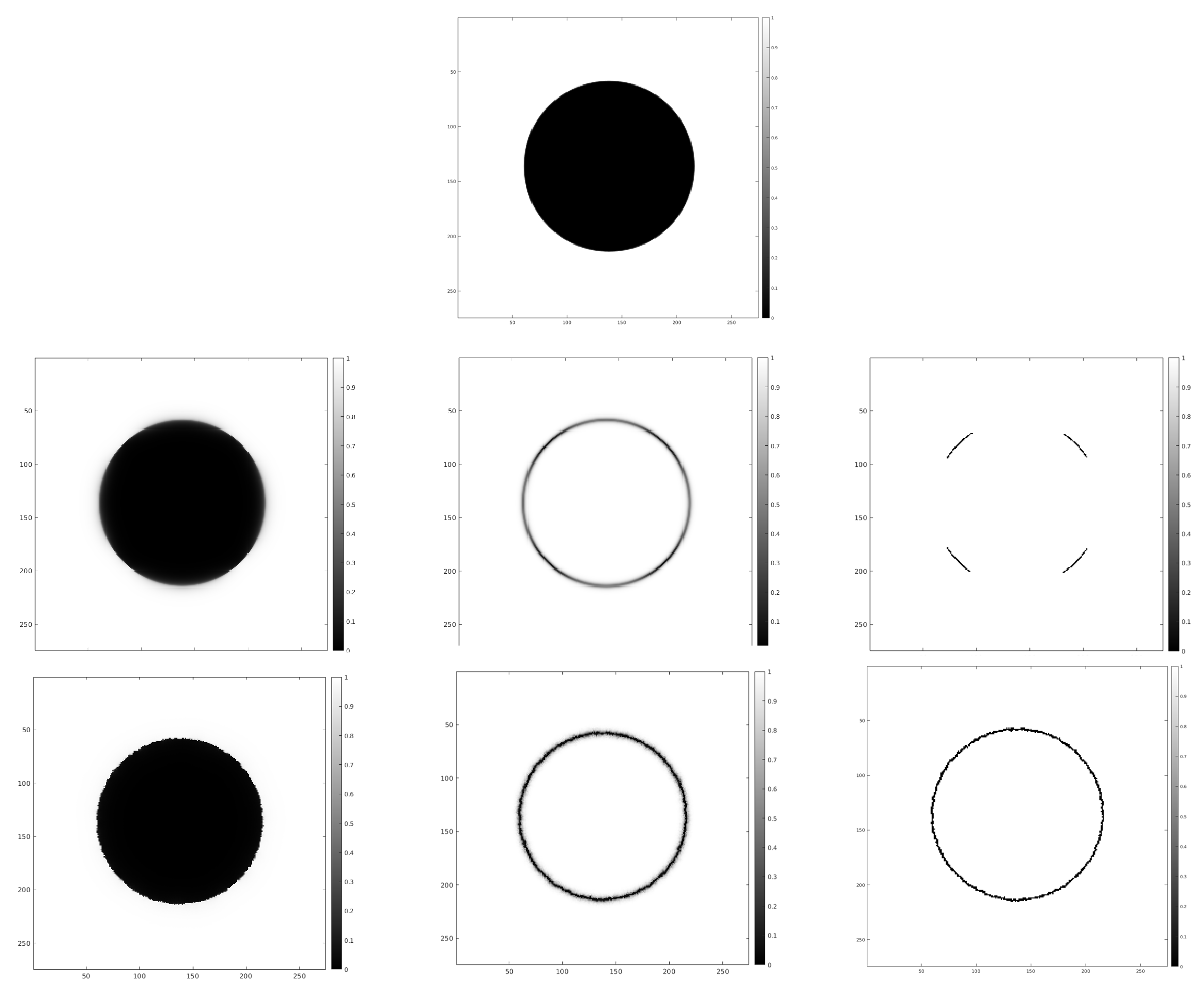}
 	\caption{Above: original image; middle (from left to right): reconstructed image using the discretization on a square lattice, corresponding edge variable, binary plot of the edge variable with threshold $0.2$; bottom (from left to right): reconstructed image using random discretization, corresponding edge variable, binary plot of the edge variable with threshold $0.2$}\label{circle}
\end{figure}
\section*{}
\subsection*{Acknowledgment}
\noindent The work of A.B. and M.C. was supported by the DFG Collaborative Research Center TRR 109, ``Discretization in Geometry and Dynamics''. M.R. acknowledges financial support from the European Research Council under the European Community's Seventh Framework Program (FP7/2014-2019 Grant Agreement QUANTHOM 335410).

\subsection*{Conflict of interest}
\noindent The authors declare that they have no conflict of interest.

\end{document}